\documentclass{amsart}

\usepackage[hmargin=3.5cm,top=3cm,bottom=3cm,includehead]{geometry}

\usepackage{babel}
\usepackage{pifont}

\usepackage{amsmath, amssymb, amsthm, mathrsfs, bbm}
\usepackage{graphicx,color,enumitem} 
\usepackage{multicol}
\usepackage{cancel}
\usepackage{abraces}
\usepackage{upgreek}
\usepackage{stmaryrd}

\usepackage[pdfborder={0 0 0}]{hyperref}

\usepackage{xparse}
\let\realItem\item 
\makeatletter
\NewDocumentCommand\myItem{ o }{%
   \IfNoValueTF{#1}%
      {\realItem}
      {\realItem[#1]\def\@currentlabel{#1}}
}
\makeatother

\usepackage{enumitem}    
\setlist[enumerate]{
    before=\let\item\myItem,       
    label=\textnormal{(\arabic*)}, 
    widest=(2')                    
}

\newcommand\bigfrown[2][\textstyle]{\ensuremath{%
  \array[b]{c}\text{\resizebox{6ex}{.7ex}{$#1\frown$}}\\[-1.3ex]#1#2\endarray}}

\usepackage[dvipsnames]{xcolor}

\newcommand\N{{\mathbb N}}

\newcommand\R{{\mathbb R}}

\newcommand\M{{\mathbb M}}

\renewcommand{\d}{\mathrm{d}}
\newcommand{\dv}{\mathrm{d} v}
\newcommand{\dt}{\mathrm{d} t}
\newcommand{\ds}{\mathrm{d} s}
\newcommand{\dr}{\mathrm{d} r}
\newcommand{\dx}{\mathrm{d} x}

\def\AA{{\mathcal A}}
\def\BB{{\mathcal B}}
\def\CC{{\mathcal C}}
\def\DD{{\mathcal D}}

\def\GG{{\mathcal G}}
\def\HH{{\mathcal H}}

\def\JJ{{\mathcal J}}
\def\KK{{\mathcal K}}

\def\MM{{\mathcal M}}

\def\OO{{\mathcal O}}

\def\QQ{{\mathcal Q}}

\def\SS{{\mathcal S}}
\def\TT{{\mathcal T}}
\def\UU{{\mathcal U}}
\def\VV{{\mathcal V}}

\def\VV{{\mathcal V}}
\def\XX{{\mathcal X}}

\def\ZZ{{\mathcal Z}}

\def\AAA{{\mathscr A}}

\def\CCC{{\mathscr C}}
\def\DDD{{\mathscr D}}

\def\III{{\mathscr I}}

\def\LLL{{\mathscr L}}
\def\MMM{{\mathscr M}}

\def\OOO{{\mathscr O}}
\def\PPP{{\mathscr P}}
\def\QQQ{{\mathscr Q}}
\def\RRR{{\mathscr R}}
\def\SSS{{\mathscr S}}
\def\TTT{{\mathscr T}}
\def\UUU{{\mathscr U}}

\def\ZZZ{{\mathscr Z}}

\def\CCCC{{\mathfrak C}}

\def\HHHH{{\mathfrak H}}
\def\IIII{{\mathfrak I}}

\def\MMMM{{\mathfrak M}}
\def\NNNN{{\mathfrak N}}

\def\RRRR{{\mathfrak R}}
\def\SSSS{{\mathfrak S}}

\def\WWWW{{\mathfrak W}}

\def\ffff{{\mathfrak f}}

\def\nnnn{{\mathfrak n}}

\def\MMMM{{\mathfrak M}}

\def\restrict#1{\raise-.5ex\hbox{\ensuremath|}_{#1}}

\def\Lloc{L_{\rm\tiny loc}}

\newcommand{\wto}{\rightharpoonup}

\def\eps{{\varepsilon}}

\newcommand{\la}{\langle}
\newcommand{\ra}{\rangle}
\newcommand{\rra}{\rangle\!\rangle}
\newcommand{\lla}{\langle\!\langle}
\newcommand{\lv}{\lvert}
\newcommand{\rv}{\lvert}
\newcommand{\lvv}{\lVert}
\newcommand{\rvv}{\rVert }
\newcommand{\lvvv}{\lvert\!\lvert\!\lvert}
\newcommand{\rvvv}{\rvert\!\rvert\!\rvert}

\newcommand{\grad}{\nabla}

\DeclareMathOperator{\Div}{div}
\DeclareMathOperator{\supp}{supp}
\newcommand\Ind{{\mathbf 1}}

\newtheorem{theo}{Theorem}[section]
\newtheorem{prop}[theo]{Proposition}
\newtheorem{lem}[theo]{Lemma}

\newtheorem*{thm*}{Theorem}
\theoremstyle{remark}
\newtheorem{rem}[theo]{Remark}

\newtheorem*{ex*}{Example}
\theoremstyle{definition}

\numberwithin{equation}{section}
\newcommand{\be}{\begin{equation}}
\newcommand{\ee}{\end{equation}}
\newcommand{\ba}{\begin{aligned}}
\newcommand{\ea}{\end{aligned}}
\newcommand{\beqn}{\begin{equation*}}
\newcommand{\eeqn}{\end{equation*}}
\newcommand{\bear}{\begin{eqnarray}}
\newcommand{\eear}{\end{eqnarray}}
\newcommand{\bean}{\begin{eqnarray*}}
\newcommand{\eean}{\end{eqnarray*}}
\newcommand{\bal}{\begin{aligned}}
\newcommand{\eal}{\end{aligned}}

\title[The Boltzmann equation on smooth and cylindrical domains.]{The Boltzmann equation on smooth and cylindrical domains with Maxwell boundary conditions.}
\author[R. Medina]{R. MEDINA}

\address[R.~Medina]{Centre de Recherche en Math\'ematiques de
  la D\'ecision (CEREMADE, CNRS UMR 7534),
  Universit\'es PSL \& Paris-Dauphine, Place de Lattre de
  Tassigny, 75775 Paris 16, France}
\email{richard.medina-rodriguez@dauphine.psl.eu}
\date{\today}



\begin{document}

\begin{abstract}
In this article we study the well-posedness of the Boltzmann equation near its hydrodynamic limit on a bounded domain. We consider two types of domains, namely $C^2$ domains with Maxwell boundary conditions where the accommodation coefficient  is a continuous space dependent function $\iota \in [\iota_0,1]$ for any $\iota_0 \in (0,1]$, or cylindrical domains with diffusive reflection on the bases of the cylinder and specular reflection on the rest of the boundary. Furthermore, we work with polynomial, stretched exponential and inverse gaussian weights to construct the Cauchy theory near the equilibrium. We remark that all methods are quantitative thus all the constants are constructive and tractable.
\end{abstract}

\maketitle

\tableofcontents

\section{Introduction}\label{S1_Intro} 

This paper investigates the well-posedness and long-time behavior of the Boltzmann equation in the regime close to the \emph{hydrodynamic limit}. We consider this problem in bounded spatial domains equipped with Maxwell boundary conditions, and where we distinguish between two types of geometries: general $C^2$ smooth domains and cylindrical domains. 

We remark that we take an interest in cylindrical domains motivated by its physical applications, the mathematical complexity of dealing with irregular domains, and as a first step towards more complex settings in future works. 

The strategy to achieve our results employs and extends the \emph{hypocoercivity} techniques developed in \cite{MR4581432}, as well as the \emph{stretching method} introduced in \cite{GuoZhou18}. In particular, we remark that during this paper we craft a more delicate $L^2-L^\infty$ theory than the one employed in \cite{GuoZhou18} and more recently \cite{GuoZhou2024}, which allows us to achieve quantitative decay estimates towards equilibrium.

\subsection{Framework}\label{sec:Framework}
We consider a small $\eps>0$ and we study the following Boltzmann equation
\be
\varepsilon\partial_\tau F = -v\cdot\grad_y F + \varepsilon^{-1} \QQ(F,F) \quad\text{ in }\UU:= (0,+\infty)\times\Omega\times\mathbb{R}^3,\label{eq:BE}
\ee
where $F = F(\tau, y, v)$ is a density function representing particles which at time $\tau\in (0,\infty)$ are located at position $y\in \Omega\subset \R^3$ and moving with velocity $v\in \R^3$. 

The presence of the small parameter $\eps > 0$ in the equation reflects the fact that the system is close to its \emph{hydrodynamic limit}. For a detailed discussion of the physical interpretation and the main mathematical results concerning this type of limit, we refer to \cite{MR2683475}.

The \emph{Boltzmann collision operator} $\QQ$ represents the collisions between particles inside $\Omega$, and is given by the bilinear form
\beqn
\QQ(G,H):={1\over 2}\int_{\R^3} \int_{\mathbb{S}^{2}} \BB\, \left[ G(v_*')H(v') +H(v_*')G(v') -G(v_*)H(v) - G(v) H(v_*) \right] \d\sigma \d v_* ,
\eeqn
where we have defined 
\begin{equation*}
    v':= v -((v-v_*)\cdot \sigma)\sigma,\qquad v_*':= v_* +((v-v_*)\cdot \sigma)\sigma,
\end{equation*}
with $\sigma \in \mathbb{S}^2$, and the \emph{collision kernel} $\BB=\BB(\lvert v-v_*\rvert,\sigma)$. We remark that $\BB$ describes the \emph{type of interaction} particles exhibit and, during this paper, we will choose the so-called \emph{hard spheres} model by taking
\begin{equation*}
\BB(\lvert v-v_*\rvert,\sigma):= \lvert (v-v_*)\cdot\sigma\rvert.
\end{equation*}

We assume $\Omega$ to be a bounded domain such that $\lv \Omega\rv=1$, and we assume that there exists $\delta \in W^{1,\infty}(\R^3, \R)$ in such a way that $\Omega = \{y\in \R^3,\, \delta (y)>0 \}$, and $\lvert \delta (y) \rvert = \mathrm{dist}(y,\partial \Omega)$ on a neighborhood of the boundary. We then define the normal outward vector 
$$
n_y = n(y):=-{\grad \delta (y)\over \lv \grad \delta (y)\rv } \quad \text{for almost every } y\in \bar \Omega.
$$
We further define the boundary set $\Sigma = \partial \Omega\times \R^d$ and we distinguish between the sets of \emph{outgoing} ($\Sigma_+$), \emph{incoming} ($\Sigma_-$), and \emph{grazing} ($\Sigma_0$) velocities at the boundary by
\beqn
\Sigma_\pm:= \{(y,v)\in \Sigma,\, \pm n_y\cdot v>0 \}, \quad \text{ and } \quad  \Sigma_0:= \{ (y,v)\in \Sigma, \, n_y\cdot v =0\}.
\eeqn
Furthermore, we denote $\Gamma:= (0,\infty)\times\Sigma$ and accordingly $\Gamma_{\pm}:= (0,\infty)\times \Sigma_\pm$. We define $\gamma F$ as the trace function associated with $F$ over $\Gamma$ and $\gamma_\pm F:= \Ind_{\Gamma_{\pm}} \gamma F$. 
 
We then complement the Boltzmann equation \eqref{eq:BE} with the \emph{Maxwell boundary condition}
\be\label{eq:BEBC}
\gamma_-F(\tau, y, v)=\RRR\gamma_+F (\tau, y, v):= (1-\iota(y)) \SSS\gamma_+F(\tau, y, v) + \iota(y) \DDD\gamma_+F(\tau, y, v) \quad  \text{ on }\Gamma_-,
\ee
where we have defined the \emph{accommodation coefficient} $\iota:\partial \Omega\to [0,1]$, the \emph{specular reflection} operator
\begin{equation*}
\SSS \gamma_+F(\tau, y,v):= \gamma_+F(\tau , y,\VV_y v) \quad \text{ with } \quad \VV_y v = v - 2(n_y\cdot v)n_y,
\end{equation*}
and the \emph{diffusive reflection} operator
\begin{equation*}
\DDD \gamma_+F(\tau , y,  v):=  \MMM(v) \widetilde{\gamma_+F}  \quad \text{where}   \quad \widetilde{\gamma_+F}=\int_{\R^d}\gamma_+F(\tau, y,u) (n(y)\cdot u)_+du .
\end{equation*}
We have also defined the \emph{Maxwellian} distributions
$$
\MMM:= \sqrt{2\pi} \MM , \quad \text{ with } \quad \MM = \MM(v):= (2\pi)^{-3/2} e^{-{\lv v\rv^2 \over 2}},
$$ 
and it is worth remarking that $\widetilde \MMM=1$.

\medskip
We present now the two types of \emph{geometric assumptions} for our domain $\Omega$, and the respective choice for the accommodation coefficient in each case. 
\begin{enumerate}[leftmargin=*]
\item[(H1)]\label{item:A1} Assume $\Omega\subset\R^3$ is an open $C^2$ domain, and $\delta\in C^{2}(\R^3, \R)\cap W^{3,\infty}(\R^3, \R)$. Moreover, take $\iota\in C(\partial\Omega)$ and assume that there is $\iota_0\in (0,1]$ such that for every $y\in \partial \Omega$ there holds $\iota(y) \in [\iota_0, 1]$.
\item[(H2)]\label{item:A2} Assume $\Omega =  (-L, L) \times \Omega_0$, for some $L>0$ and where $\Omega_0\subset \R^2$ is the $2$-dimensional ball of radius $\RRRR>0$ centered at the origin. In this case we also define 
\begin{equation*}
    \Lambda_1:= \{-L\}\times\Omega_0,\quad
    \Lambda_2:= \{L\}\times\Omega_0, \quad
    \Lambda_3:= (-L,L)\times \partial \Omega_0 ,\label{eq:cylinderDefinition}
\end{equation*}
and $\Lambda:= \Lambda_1\cup \Lambda_2\cup \Lambda_3$. Furthermore, we impose  \emph{mixed} boundary conditions by taking $\iota= \Ind_{\Lambda_1 \cup \Lambda_2}$, i.e purely diffusive boundary condition on the bases of the cylinder ($\Lambda_1\cup\Lambda_2$), and specularity on the lateral surface ($\Lambda_3$).
\end{enumerate}

Finally, we complement Equation \eqref{eq:BE}-\eqref{eq:BEBC} with the initial condition
\be\label{eq:BEIC}
F (\tau = 0, \cdot ) = F_0 \quad \text{ in } \OO:= \Omega \times \R^3,
\ee
for some function $F_0$ satisfying $\lla F_0\rra_\OO := \displaystyle \int_{\OO} F_0   \, \d y\dv =1$.


\subsection{Main result and discussion}\label{sec:MainResult} 

In order to express our main result we need to introduce the set of the so-called \emph{admissible weight functions} given by:
\begin{itemize}[leftmargin=*]
\item\label{item:W1} \emph{Polynomial weights:} We consider $\omega(v)=(1+ \lvert v\rvert^2 )^{q/2}$ with  $q>q^\star_\iota$, for some explicit $q^\star_\iota >0$ defined in Remark~\ref{def:nu*}.
\item\label{item:W2} \emph{Stretched exponential weights:} We consider $\omega (v)=e^{\zeta\lvert v\rvert^s}$ with $s \in (0,2)$ and $\zeta>0$.
\item\label{item:W3}\emph{Inverse gaussian weights:} We consider $\omega (v)=e^{\zeta\lvert v\rvert^2}$ with $\zeta \in (0,1/2)$.
\end{itemize}

Moreover, for a given measure space $(Z,\ZZZ,\mu)$, a weight function $\rho: Z \to (0,\infty)$, and an exponent $p \in [1,\infty]$, we define the weighted Lebesgue spaces $L^p_\rho(Z)$ 
associated to the norm 
\be\label{def:Lebesgue_weighted_spaces}
\| g \|_{L^p_\rho(Z)} = \| \rho g \|_{L^p(Z)}. 
\ee

In this framework we have the following result for the Boltzmann equation.

\begin{theo}\label{theo:Main}
Consider either Assumption \ref{item:A1} or Assumption \ref{item:A2} to hold, and et $\omega$ be an admissible weight function. There exists  $\eps_0>0$ such that for every $\eps \in (0, \eps_0)$ there is $\eta(\eps)\in (0,1)$, satisfying $\eta(\eps)\to 0$ as $\eps\to 0$, such that for every $F_0\in L^\infty_{\omega}(\OO)$ satisfying
\beqn
 \lVert F_0-\MM\rVert_{L^{\infty}_{\omega}(\OO)}  \leq  (\eta(\eps))^2 ,
\eeqn
there exists a function $F\in  L^\infty_\omega(\UU)$ 
unique global solution to the Boltzmann equation~\eqref{eq:BE}--\eqref{eq:BEBC}--\eqref{eq:BEIC} in the distributional sense.
Furthermore, there is a constructive constant $\theta>0$ such that 
\beqn
  \lVert  F_\tau - \MM \rVert_{L^{\infty}_{\omega}( \OO)}   \leq e^{-\theta \tau} \, \eta(\eps)\label{eq:BEdecayFinal} \qquad  \forall \tau\geq 0.
\eeqn
\end{theo}

The precise sense of the solution given by Theorem~\ref{theo:Main} is constructed in Theorem~\ref{theo:MainRescaled}, after taking $f=F-\MM$, and performing the change of variables from Subsection~\ref{ssec:LinRes}.

\smallskip
We now briefly discuss the history of the progress on the study of the Boltzmann equation, and we pay particular attention to boundary value problems. This has been an active field of mathematical research since the groundbreaking works of L. Boltzmann in \cite{Boltzmann1905} and J.C. Maxwell in \cite{Maxwell1867}. 

\smallskip
We first mention the framework of renormalized solutions introduced by DiPerna and Lions in \cite{Diperna1989} for any arbitrary data,  
and in this setting for bounded domains we cite the works by S. Mischler in \cite{Mischler2000, Mischler2010}, and we also refer to the results presented in the review paper of C. Cercignani \cite{MR1359296}, and the references therein.

\smallskip
For other results concerning the Boltzmann equation (and variants), either on the torus or on the whole space, we refer to \cite{MR363332, MR1828983, MR2214953, MR3779780, MR2784329, MR1908664}. 

Regarding solely the long-time asymptotic behavior of (given) solutions we mention \cite{MR2116276} by L. Desvilletes and C. Villani, where they obtain relaxation to a Maxwellian global equilibrium for positivity and regular enough solutions.

\smallskip
On the existence of weak distributional solutions in bounded domains and their long-time assymptotics, we have the pioneering paper by Y. Guo \cite{MR2679358}, where the boundary conditions are either inflow, bounce-back, specular, or diffusive. It is worth mentioning though, that Guo assumes the spatial domain to be convex and analytic, which are quite demanding hypothesis in applications. 

\smallskip
In the last decade there have been several studies with the objective of relaxing their geometrical assumptions:
we cite \cite{MR3579575} where M. Briant studied a Boltzmann equations with diffusive reflexion along a $C^1$ domain---we note that he also studies the problem for specular reflections but under the extra assumption of having an analytical and convex boundary. We also refer to the work of C. Kim and D. Lee in \cite{MR3762275} for $C^3$ convex domains with specular reflections. 

We also mention \cite{chen2024boltzmannequationmixedboundary} where H. Chen and R. Duan consider a problem with mixed boundary conditions. It is worth remarking that such mixed boundary conditions are similar in objective as those of our Assumption~\ref{item:A2}, however instead of a cylinder they consider a $C^1$ domain. 

For results on the Boltzmann equation in irregular bounded domains we cite \cite{MR1629463} by L. Arkeryd and A. Heintz for a general conservative boundary condition, and, more particularly in the setting of cylindrical domains, we further cite \cite{MR3840911}, where C. Kim and D. Lee construct weak solutions for the Boltzmann equation complemented with specular boundary conditions. 

\smallskip
Regarding the problem under Maxwell boundary conditions we have the result from M. Briant and Y. Guo in \cite{BriantGuo16}, under the conditions of having an accommodation coefficient bounded below by $\sqrt{2/3}$, and for a $C^1$ bounded domain. 

For Boltzmann equations with long-range interaction (no Grad's cutofff assumption) we cite the more recent paper \cite{deng2025noncutoffboltzmannequationbounded} where D. Deng studies the problem within a $C^3$ bounded domain and complemented with a Maxwell boundary condition and a fixed accommodation coefficient $\iota \in (0, 1]$. 

Moreover, we also have the preprint from Y. Guo and F. Zhou \cite{GuoZhou18}, and their recently updated version with J. Jung \cite{GuoZhou2024}, where they study a Boltzmann equation near the hydrodynamic limit, in a $C^2$ domain, complemented with Maxwell boundary conditions with a constant accommodation coefficient within $[0,1]$. They construct a perturbative well-posedness theory for this problem and they also study its hydrodynamic limit.

\smallskip
The framework for our paper is motivated by, and therefore closely related to, the one developed in \cite{GuoZhou18, GuoZhou2024}. 
However, within the study of the Boltzmann equation, our result constitutes a genuine generalization of theirs in several key aspects. 
First, we craft a more delicate $L^2-L^\infty$ theory than the one presented in \cite{GuoZhou18, GuoZhou2024} for the linearized problem, allowing us to derive explicit and constructive decay rates to the equilibrium, going beyond the boundedness results previously obtained.
Second, we also carry out the analysis in cylindrical domains, introducing the presence of geometric irregularities to the obtention of the estimates, thereby generalizing the stretching method from \cite{GuoZhou18, GuoZhou2024} for the cylindrical setting.
Third, even though we do not achieve the full range of $[0,1]$ for the accommodation coefficient in smooth domains, we allow $\iota$ to be a spatially dependent continuous function, and within cylindrical domains we further have it to be discontinuous.

This last fact makes Theorem~\ref{theo:Main} also a generalization of \cite{BriantGuo16}, where we recall that the accommodation coefficient had an imposed lower bound of $\sqrt{2/3}$.
It is also worth mentioning that during the obtention of our main result we have extended, in Section~\ref{sec:Hypo}, the results from \cite{MR4581432} by constructing hypocoercivity estimates in cylindrical domains satisfying Assumption \ref{item:RH2}. Furthermore, we also provide well-posedness results for initial data with a wide range of decaying tail at infinity, including polynomial.

Notably, we emphasize that, to the best of our knowledge, this is the first well-posedness result for the Boltzmann equation in cylindrical domains with Maxwell boundary conditions and a discontinuous accommodation coefficient ranging over the entire interval $[0,1]$.


\subsection{Transformation of the problem} \label{ssec:LinRes}
To study the Boltzmann Equation~\eqref{eq:BE}-\eqref{eq:BEBC}-\eqref{eq:BEIC}, we define the function $\bar F$ such that 
$$
 \bar F(\tau, y, v) = F(\tau , y, v)-\MM(v),
$$ 
and we make the changes of variables $\tau = \eps^2 t $, $y= \eps x$, so we introduce $\Omega^{\varepsilon}:=\{\varepsilon^{-1} y, \, y\in\Omega\}$ and $\OO^\eps= \Omega^\eps\times \R^3$. Then we naturally define 
$$ 
f (t, x, v) := \bar F(\eps^2 t , \eps x ,v),
$$ 
so that $f$ satisfies the \emph{linearized rescaled Boltzmann equation}
\be
\partial_t f = -v\cdot \grad_x  f +\CCC f +  \QQ(f,f)  \quad \text{ in } \UU^\eps:= (0,+\infty)\times \OO^\eps, \label{eq:RBE}
\ee
where we have defined the \emph{linearized Boltzmann operator} $\CCC f:=\QQ(\MM,f)+\QQ(f,\MM)$. We decompose now $\CCC f = Kf-\nu f$, 
where, on the one hand, we have the non-local operator
\beqn
Kf = Kf(\cdot , v) := \int_{\R^3}\int_{\mathbb{S}^{2}} \BB \left[\MM(v'_*)f(v') + \MM(v') f(v'_*) -\MM(v) f(v_*)\right]d\sigma dv_* = \int_{\R^3} k(v,v_*) f (v_*)dv_*,
\eeqn
with the kernel function
\begin{multline*}
k = k(v,v_*):=  \sqrt{2\over \pi} \lvert v-v_* \rvert^{-1} e^{ -{1\over 8} {(\lvert v_*\rvert^2 - \lvert v\rvert^2)^2\over \lvert v-v_*\rvert^2} - {1\over 8} \lvert v-v_*\rvert^2  -{\lvert v\rvert^2 \over 4} + {\lvert v_*\rvert^2 \over 4}   }
 -{1\over 2} \lvert v-v_*\rvert e^{ -{\lvert v\rvert^2 \over 2} }, \label{def:Kkernel} 
\end{multline*}
see for instance \cite[Theorem 7.2.1]{MR1307620} for a derivation of $k$, up to a conjugate change of scale. On the other hand, we have 
$$
\nu = \nu(v):= \int_{\R^3}\int_{\mathbb{S}^{2}} B \,  \MM(v_*)d\sigma dv_* ,
$$ 
and there are constants $\nu_0, \nu_1 >0$ such that 
\be\label{eq:Controlnu}
\nu_0\leq \nu_0\la v\ra \leq\nu(v)\leq \nu_1\la v\ra,
\ee
where we define $\la v\ra:= \sqrt{1+\lvert v\rvert^2}$, and we refer to \cite[Section 4]{MR3779780} for a derivation of this. 
\begin{rem}\label{def:nu*}
Furthermore, also from \cite[Section 4]{MR3779780}, we note that possible choices for these constants are $\nu_0 = 4\pi \sqrt{2/( e\pi)}$ and $\nu_1 = 16\pi $. This motivates the definition of \hbox{$\nu_*:= \nu_1/\nu_0 = 2\sqrt{2e\pi}$} and subsequently
 \begin{equation} \label{def:Poly_degree}
q^\star_\iota =\left\{\begin{array}{ll} 
\displaystyle {5\iota_0+8\nu_* + \sqrt{ 128\pi \CCCC_0\iota_0\nu_* +( 8 \nu_* - 3\iota_0)^2 }\over 2\iota_0} & \text{ if \ref{item:A1} holds} \\
\\
\displaystyle {5+16\nu_* + \sqrt{9 +  160\nu_*+ 256\nu_*\left( 1+\pi \CCCC_0\right)   + 256\nu_*^2 }\over 2}  & \text{ if \ref{item:A2} holds},
\end{array}\right.
\end{equation}
where the constant $\CCCC_0>0$ is such that $\omega \MMM\leq \CCCC_0$. 
\end{rem}

We translate now our framework towards this new setting. We recall that for any $x\in \partial \Omega^\eps$ there holds $y = \eps x\in \partial \Omega$, thus we define $\delta^\eps:\Omega^\eps\to \R$, $\delta^\eps(x):= \delta(y)$, and we observe that there holds almost everywhere
\be\label{eq:Equality_normal_eps}
n(y) = -{\grad_y \delta (y)\over \lv \grad_y \delta (y)\rv } = -{\grad_y \delta (\eps x)\over \lv \grad_y \delta (\eps x)\rv } = -{\grad_x [\delta (\eps x)]\over \lv \grad_x [\delta (\eps x)] \rv } = -{\grad_x \delta^\eps (x)\over \lv \grad \delta^\eps (x)\rv } =: n(x) = n_x,
\ee
which is nothing but saying that the normal vector on a rescaled point of the boundary set $\partial\Omega^\eps$ coincides with the one of the corresponding point on the original boundary set $\partial\Omega$. \\

We define then $\Sigma^\eps:= \partial \Omega^\eps\times \R^3$ and we define accordingly the sets 
$$
\Sigma^\eps_\pm:= \{(x,v)\in \Sigma^\eps, \, \pm n_x \cdot v >0\}, \quad \Sigma_0^\eps:= \{(x,v)\in \Sigma^\eps, \,  n_x \cdot v =0\},
$$
and $\Gamma_\pm^\eps:= (0,\infty) \times \Sigma^\eps_\pm$. We introduce the rescaled accommodation coefficient $\iota^\eps:\partial\Omega^\eps\to [0,1]$ defined as $\iota^\eps(x):= \iota(\eps x)$. We have then that the associated boundary conditions for the rescaled Boltzmann Equation~\eqref{eq:RBE} read
\be \label{eq:RBEBC}
\gamma_-  f =\RRR \gamma_+ f = (1-\iota^\eps) \SSS \gamma_+f + \iota^\eps \DDD \gamma_+f  \quad \text{ on }\Gamma^\eps_- ,
\ee
by abusing notation on the fact that we will maintain unchanged the symbols of the boundary reflection operators. \\

We translate now the geometrical assumptions \ref{item:A1} and \ref{item:A2} into the rescaled setting. 
\begin{enumerate}[leftmargin=*]
\item[(RH1)]\label{item:RH1} $\Omega^\eps\subset\R^3$ is an open $C^2$ domain, and $\delta^\eps \in C^{2}(\R^3, \R)\cup W^{3,\infty}$. Moreover, $\iota^\eps\in C(\partial\Omega)$ and such that for every $x\in \partial \Omega^\eps$, $\iota^\eps(x) \in [\iota_0, 1]$ with $\iota_0 \in (0,1]$.

\item[(RH2)]\label{item:RH2} $\Omega^\eps =  (-L^\eps, L^\eps) \times \Omega_0^\eps$, with $L^\eps:= \eps^{-1} L$ and $\Omega_0^\eps:= \eps^{-1} \Omega_0$, i.e is the 2-dimensional ball of radius $\eps^{-1}\RRRR$ centered at the origin. We also define 
\begin{equation*}
    \Lambda_1^\eps:= \{-L^\eps\}\times\Omega_0^\eps,\quad
    \Lambda_2^\eps:= \{L^\eps\}\times\Omega_0^\eps, \quad
    \Lambda_3^\eps:= (-L^\eps,L^\eps)\times \partial \Omega_0^\eps ,\label{eq:cylinderDefinition}
\end{equation*}
and $\Lambda^\eps:= \Lambda_1^\eps\cup \Lambda_2^\eps\cup \Lambda_3^\eps$. Moreover, we take $\iota^\eps= \Ind_{\Lambda_1^\eps \cup \Lambda_2^\eps}$.
\end{enumerate}

Finally we complement Equation~\eqref{eq:RBE}-\eqref{eq:RBEBC} with the initial condition
\be\label{eq:RBEIC}
f (t=0,x, v) =  f_0 (x, v) := F_0 (\eps x,v) - \MM(v)  \quad \text{ in } \OO^\eps,
\ee
and we remark that due to the hypothesis on $F_0$ we have that $\lla f_0 \rra_{\OO^\eps} =0$.\\

We state now an equivalent version of Theorem~\ref{theo:Main} in the rescaled setting, and this will be the result we will prove during this paper. 

\begin{theo}\label{theo:MainRescaled}
Consider either Assumption \ref{item:RH1} or Assumption \ref{item:RH2} to hold, and let $\omega$ be an admissible weight function. There exists  $\eps_0>0$ such that for every $\eps \in (0, \eps_0)$ there is $\eta(\eps)\in (0,1)$, satisfying $\eta(\eps) \to 0$ as $\eps\to 0$, such that for every $f_0\in L^\infty_\omega(\OO^\eps)$ satisfying
\beqn
\lVert f_0\rVert_{L^{\infty}_{\omega}(\OO^\eps)} \leq  (\eta (\eps))^2,
\eeqn
there exists  $f\in  L^\infty_\omega(\UU^\eps)$, unique global solution to the linearized rescaled Boltzmann equation \eqref{eq:RBE}--\eqref{eq:RBEBC}--\eqref{eq:RBEIC}. Furthermore, there is a constructive constant $\theta>0$ such that 
\beqn
  \lVert  f_t \rVert_{L^{\infty}_{\omega}( \OO)}   \leq e^{-\theta \eps^2 t}  \, \eta(\eps) \label{eq:BEdecayFinal} \qquad  \forall t\geq0.
\eeqn
\end{theo}

In order to explain the precise sense of the solutions given by Theorem~\ref{theo:MainRescaled} we introduce the subset of admissible weight functions, that we will call \emph{strongly confining}, as follows
\be\label{def:WWWW_1}
\WWWW_1 = \{\omega:\R^3\to \R,\, \omega(v) = e^{\zeta\lv v\rv^2}, \, \text{ with } \zeta \in (1/4, 1/2) \}.
\ee
In consequence, we also define the space of \emph{weakly confining} admissible weights by  
\be\label{def:WeightSpace0}
\WWWW_0 =\{\omega:\R^3\to \R, \text{ $\omega$ is an admissible weight function, with } \omega \notin \WWWW_1 \}.
\ee
We then remark that the proof of Theorem~\ref{theo:MainRescaled} is given by Theorem~\ref{theo:existenceWWWWinfty} for strongly confining admissible weight functions and Theorem~\ref{theo:existenceWWWW0} for weakly confining ones.

\subsection{Strategy for the proof of the main result}\label{ssec:Strategy}
We explain in this subsection the main ideas used for the proof of Theorem \ref{theo:MainRescaled}.

%
%
%
%

\medskip\noindent
\ding{172} \emph{Hypocoercivity decay:} We consider the problem
\begin{equation}
	\left\{\begin{array}{llll}
		 \partial_t f &=& \LLL f:= -v\cdot \grad_x f + \CCC f  &\text{ in }\UU^\eps\\
		\gamma_-f&=&\RRR\gamma_+f &\text{ on }\Gamma^\eps_-\\
		f_{t=0}&=&f_0 &\text{ in }\OO^\eps.
	\end{array}\right.\label{eq:LRBE}
\end{equation}

On the one hand we observe that, for any nice enough functions $G,H, \varphi: \R^3 \to \R$, the Boltzmann collision operator classically satisfies 
\be\label{eq:pre_ConservationLaws}
\int_{\R^3} \QQ(G,H) \,  \varphi 
= \frac 18 \int_{\R^3}\int_{\R^3} \int_{\mathbb{S}^2} \BB  \left( G_*' H' +H_*'G' -G_*H - G H_* \right) \left( \varphi+\varphi_* - \varphi'-\varphi'_*\right) , 
\ee
where we have used the shorthands $\phi=\phi(v)$, $\phi_* = \phi(v_*)$, $\phi' = \phi(v')$, $\phi'_* = \phi(v'_*)$,
and we recall that $v'$ and $v'_*$ are given in Subsection \ref{sec:Framework}. The interested reader can consult the derivation of \eqref{eq:pre_ConservationLaws} in \cite[Section 3.1]{MR1307620}.

In particular, if we set $\R^3 \ni v = (v_1, v_2, v_3)$, then \eqref{eq:pre_ConservationLaws} implies that choosing $\varphi = \varphi(v)$ to be either $1, v_1, v_2, v_3$ or $\lv v\rv^2$ there holds
\be\label{eq:ConservationLaws}
\int_{\R^3} \QQ(G,H) (v)\,  \varphi (v) \, \dv = 0.
\ee

Therefore, \eqref{eq:ConservationLaws} implies that for the previous choices of $\varphi$ there holds
\be\label{eq:ConservationLawsCCC}
\int_{\R^3} (\CCC f)(v)\,  \varphi(v) \, \dv = 0. 
\ee

We now take as a momentary framework the Hilbert space $L^2_{\MM^{-1/2}}(\R^3)$ endowed with the scalar product 
$$
\la g, h \ra_{L^2_{\MM^{-1/2}}(\R^3)}:= \int_{\R^3} g(v)\, h(v)\, \MM^{-1}(v) \, \dv,
$$
and its associated norm as defined in \eqref{def:Lebesgue_weighted_spaces}. We observe that \cite[Theorem 7.2.4]{MR1307620} implies that we can set ${\rm Dom}(\CCC):= L^2_{\MM^{-1/2}}(\R^3)$, that $\CCC$ is a closed operator on its domain, and \eqref{eq:ConservationLawsCCC} further gives that
$$
{\rm ker} \, (\CCC) = {\rm span}\, \{ \MM, \, v_1 \MM, \, v_2 \MM, \, v_3 \MM, \, \lv v\rv^2 \MM\}.
$$
This motivates the definition of $\Pi f$ as the projection of $f\in {\rm Dom}(\CCC) $ onto $\mathrm{ker} (\CCC)$ given by 
\beqn\label{def:pi}
\Pi f = \left( \int_{\R^d} f (w) \, \d w \right) \MM
+ \left( \int_{\R^d} w f (w) \, \d w   \right) \cdot v \MM 
+ \left( \int_{\R^d} \frac{|w|^2-3}{\sqrt{6}} \, f (w) \, \d w   \right) \frac{|v|^2-3}{\sqrt{6}} \, \MM.
\eeqn
We also note that $\CCC$ is self-adjoint on its domain and negative, so that its spectrum is included in $\R_{-}$, and \eqref{eq:ConservationLawsCCC} holds true for any { $f \in \mathrm{Dom} (\CCC)$}. 
Furthermore, $\CCC$ satisfies a \emph{microscopic coercivity} estimate, more precisely  \cite[Theorem 1.1]{MR2231011} gives that there is $\kappa_0 >0$ such that for any $f \in \mathrm{Dom} (\CCC)$ one has
\be\label{eq:MicroCoercivity}
\la-\CCC f , f \ra_{L^2_{\MM^{-1/2}}(\R^3)} \ge \kappa_0 \|  f^\perp \|_{L^2_{\MM^{-1/2}} (\R^3)}^2,
\ee
where $f^\perp:= f - \Pi f$. Finally, we observe that for any polynomial function $\phi=\phi(v): \R^3 \to \R$ of degree less or equal to $4$, there holds  $\MM \phi \in \hbox{Dom}(\CCC)$, and using again \cite[Theorem 7.2.4]{MR1307620} we deduce that there is a constant $C_\phi >0$ such that
$$
\bigl\| \CCC (\phi \MM) \bigr\|_{L^2_{\MM^{-1/2}}(\R^3)}  \le \bigl\| \phi \MM \bigr\|_{L^2_{\MM^{-1/2}}(\R^3)}  \leq C_\phi.
$$
Finally, we remark that the conservations laws \eqref{eq:ConservationLawsCCC} together with the fact that 
$$
\lvv \RRR \rvv_{L^1 (\Gamma_+^\eps)\to L^1(\Gamma_-^\eps)} = 1,
$$
imply that, at least formally, Equation~\eqref{eq:LRBE} conserves mass, i.e for any solution $f$ of Equation~\eqref{eq:LRBE} there holds $\lla f_t\rra_{\OO^\eps} = \lla f_0\rra_{\OO^\eps} = 0$ for every $t\geq 0$. 

\medskip
Altogether, the previous analysis implies that $\CCC$ satisfies the structural assumptions made on \cite{MR4581432}. We thus define the weighted Hilbert space  $\HH:=  L^2_{\MM^{-1/2}}(\OO^\eps)$ equipped with the scalar product 
$$
\la g, h\ra_\HH := \int_\OO g(x,v)\, h (x,v) \MM^{-1} (v) \dv\dx,
$$
and we observe that \cite[Theorem 1.1]{MR4581432} and  \cite[Theorem 5.1]{MR4581432} imply that under Assumption \ref{item:RH1} we have that there is $\kappa >0$ such that any solution $f$ of Equation~\eqref{eq:LRBE} satisfies
\be\label{eq:Hypo_Intro}
\lvv f_t \rvv_\HH \leq C e^{\kappa \eps^{-2} t} \lvv f_0 \rvv_\HH \qquad \forall t\geq 0,
\ee
for a constant $C>0$, independent of $\eps$. We remark that the constant $\eps^{-2}$ on the rate of the exponential decay comes from the change of variables made in Subsection~\ref{ssec:LinRes}, see for instance \cite[Proposition A.1]{carrapatoso2025navierstokeslimitkineticequations} for a similar result.

We then dedicate Section \ref{sec:Hypo} to prove \eqref{eq:Hypo_Intro} under the geometrical Assumption \ref{item:RH2}. The main goal is to obtain $H^2$ regularity estimates for the solutions of certain elliptic equations introduced in \cite[Section 2]{MR4581432}, for cylindrically shaped domains. It is worth remarking that this is not straightforward because such gain of regularity, even though classical in smooth domains, doesn't necessarily hold for irregular domains, see for instance \cite{Grisvard11}. 

Once obtained these regularity estimates, the rest of the computations leading to \eqref{eq:Hypo_Intro} follow as an exact repetition of the main ideas and computations from \cite{MR4581432}.

\medskip\noindent
\ding{173} \emph{$L^2-L^\infty$ method for strongly confining weights:} 
We remark that $\LLL = \TTT + K$, where we have defined the \emph{free transport operator}
\be\label{eq:DefTransportOp}
\TTT h(t, x, v):= -v\cdot \grad_x h(t,x,v) -\nu(v) h(t, x,v) ,
\ee
for any nice enough function $h: \UU^\eps\to \R$. We consider then a function $G: \UU^\eps \to \R$, and we study the following perturbed evolution equation
\be
	\left\{\begin{array}{llll}
		\partial_{t} f &=& \TTT f + K f + G&\text{ in }\UU^\eps\\
		\gamma_- f&=&\RRR \gamma_+f &\text{ on }\Gamma_{-}^\eps\\
		 f_{t=0}&=& f _0  &\text{ in }\OO^\eps.
	\end{array}\right.\label{eq:PLRBE}
\ee

Our main interest is to provide weighted $L^2-L^\infty$ estimates for the solutions of Equation~\eqref{eq:PLRBE} by employing the \emph{stretching method} developed in \cite{GuoZhou18}, see also \cite{GuoZhou2024}.
To do this we define the \emph{free transport semigroup}, classically generated by $\TTT$ (see for instance \cite[Section 8]{sanchez2024kreinrutmantheorem}), as  
\be\label{def:Transport_semigroup}
S_\TTT(t) h(s, x,v):= e^{-\nu(v) t}h(s, x-vt, v) \qquad \text{ for any } s,t\in \R_+, \text{ and any } (x,v)\in \OO^\eps.
\ee
Furthermore, for any sufficiently nice operator $\AA$, and any $\sigma \in [0, t]$, we define the convolution operation
\be\label{def:convolution_sigma}
 (S_\TTT *_{\sigma}  \AA) \, h (t,x,v):= \int_{\sigma}^t S_\TTT(t-s) \, \AA\,  h(s, x, v) ds = \int_\sigma^t e^{-\nu(v)(t-s)} \AA f(s, x-(t-s)v, v) \ds ,
\ee
and in particular we will classically denote $S_\TTT*\AA:= S_\TTT *_{0}\,  \AA$.
We now remark that $S_\TTT$  will satisfy two important properties.
\begin{itemize}[leftmargin=*]
\item On the one hand, $S_\TTT * K$ generates an integrability gain of the form $L^2\to L^\infty$ in the sense provided by Proposition~\ref{prop:Kregularization} and Proposition~\ref{prop:KregularizationCylinder}, for the cases of smooth and cylindrical domains respectively.

\item On the other hand, and as described in \ref{item:K3}, weighted $L^\infty$ estimates on $S_\TTT*G$ generate a gain of weight of $\nu^{-1}$, which is crucial to control the Boltzmann collision operator $\QQ$ (see Lemma~\ref{lem:NonlinearGaussianWeightEstimate}).
\end{itemize}

Exploiting this, the strategy for the $L^2-L^\infty$ decay transfer will focus on the repeated use of the Duhamel formula on the semigroup $S_\TTT$ for the solutions of Equation~\eqref{eq:PLRBE}, and using the stretching method we will use the effect of the parameter $\eps$ on the domain to make small the \emph{measure} of the \textquotedblleft set of singular trajectories". 

To explain how this works we first observe that the characteristics of Equation~\eqref{eq:PLRBE} are given by 
\be
 X (s; t,x,v):= x - v(t - s)  \quad \text{ and } \quad V(s; t,x,v):= v. \label{eq:carachteristics}
 \ee
Therefore, for any fixed particle with coordinates in $ \UU^\eps$, we can characterize the coordinates of the last collision against the boundary of $\Omega^\eps$. Indeed, let $(t_0,x_0,v_0) \in \UU^\eps$ be the coordinates of a particle, we define the time of collision along this trajectory ($t_b$), the time of life of the particle prior to such collision ($t_1$), and the position $(x_1)$ and velocity ($v_1$) at the boundary during this collision, as follows
\begin{equation}\label{eq:BackwardsTrajectory}
\begin{array}{rcl}
    t_b(x_0,v_0) &=& \inf\{s >0;\ X(-s, 0, x, v)\notin\Omega^\eps\},\\
    t_1(t_0,x_0,v_0)&=&t_0-t_b(x_0,v_0),\\
    x_1(t_0, x_0,v_0)&=& X(t_1; t, x, v) = x -v(t - t_1),\\
    v_1(t_0, x_0,v_0)&=&\left\{\begin{array}{cl}
         \VV_{x_1}(v_0)& \text{ during specular reflection,} \\
         v_0^*& \text{ during diffuse reflection,}
    \end{array}\right.
\end{array}
\end{equation}
where $v_0^*$ stands for an independent variable.

We then note that the smaller $\eps$ is chosen, the \emph{flatter} the boundary $\partial \Omega^\eps$ becomes. In particular, we observe that if the boundary is flat then any particle following the \emph{backwards trajectories} given by \eqref{eq:BackwardsTrajectory}, doesn't collide more than once against the boundary.

Therefore, the objective of the stretching method is to choose $\eps$ small enough such that the set of paths that collide more than once against the boundary is \emph{small} in an appropriate sense.  

This will be done for smooth domains in Section \ref{sec:Linfty}, mainly by following the arguments presented in \cite{GuoZhou18}. Afterwards during Section~\ref{sec:LinftyCylinder}, and following a similar methodology, we will generalize this technique for cylindrical domains. 

It is worth remarking that such a generalization is not straightforward and requires a more delicate control of the backwards trajectories due to the presence of irregularities at the boundary. In particular, we will need to account for trajectories where the particle changes between the boundary subsets $\Lambda_1^\eps\cup\Lambda_2^\eps$ and $\Lambda_3^\eps$.

Finally, we emphasize that by combining in a very delicate way the $L^2$ hypocoercivity decay and the $L^\infty$ control given by the stretching method, we craft a $L^2-L^\infty$ control that provides a control on the long-time behavior for the solutions of Equation \eqref{eq:PLRBE} in a weighted $L^\infty$ space.

\medskip\noindent
\ding{174} \emph{Proof of Theorem~\ref{theo:MainRescaled} for strongly confining weight functions:} 
The a priori estimates obtained by using the above $L^2-L^\infty$ method are enough then to construct a well-posedness theory for Equation~\eqref{eq:PLRBE}. Using all these elements we will prove the well-posedness of Equation~\eqref{eq:RBE}-\eqref{eq:RBEBC}-\eqref{eq:RBEIC} by using a Banach fixed point argument.

\medskip\noindent
\ding{175} 
 \emph{Proof of Theorem~\ref{theo:MainRescaled} for weakly confining weights:} 
 In order to obtain the well-posedness for the linearized Boltzmann equation for weakly confining weights, we need to split the problem into a system of equations. 

Before doing this we need to define some operators. 
Consider a parameter $\delta>0$, and define the sets
\bean
E_1^\delta &=&\{ (v,v_*,\theta) \in \R^3\times \R^3 \times \mathbb S^2; \, \lvert v\rvert \leq \delta^{-1},\, 2\delta \leq \lvert v-v_*\rvert \leq \delta^{-1}, \, \lvert \cos(\theta)\rvert \leq 1-2\delta \},\\
E_2^\delta &=&\{ (v,v_*,\theta) \in \R^3\times \R^3 \times \mathbb S^2; \,  \lvert v\rvert \leq 2\delta^{-1},\, \delta \leq \lvert v-v_*\rvert \leq 2\delta^{-1}, \, \lvert \cos(\theta)\rvert \leq 1-\delta \},
\eean
and the cut-off function $\chi_\delta\in C^\infty_c(\R^3\times \R^3 \times\mathbb{S}^{2}, \R)$, such that $\Ind_{E^\delta_1}\leq \chi_\delta\leq \Ind_{E^\delta_2}$. We define then the operators
\be\label{eq:defSplittingLL}
 \AA_\delta f= K(  (1-\chi_\delta)f)\quad \text{ and }\quad \KK_\delta f =  K ( \chi_\delta f),
\ee
and we remark that $K = \AA_\delta + \KK_\delta$. 

This choice for the splitting of the operator $K$ is taken from \cite{MR3779780} and motivated by the ideas developed in \cite[Section 6]{BriantGuo16}. Its main interest is to exploit \cite[Lemma 4.12]{MR3779780}, providing dissipative estimates on $\AA_\delta$.

Moreover, we remark that the nature of the operator $\KK_\delta$ will make it enjoy nice boundedness properties (see for instance Lemma~\ref{lem:KKcontrol}). 

\medskip
We then observe that if $f_1, f_2$ are solutions of the system of equations
 \be
	\left\{\begin{array}{llll}
		\partial_{t} f_1 &=& \TTT f_1 + \AA_\delta f_1 + \QQ(f_1 + f_2 , f_1 +f_2) &\text{ in }\UU^\eps \\
		\gamma_- f_1&=&\RRR \gamma_+f_1 &\text{ on }\Gamma_{-}^\eps\\
		 f_{1,t=0}&=& f _0 &\text{ in }\OO^\eps,
	\end{array}\right.\label{eq:PLRBEDiss}
\ee
and
 \be
	\left\{\begin{array}{llll}
		\partial_{t} f_2 &=& \LLL f_2 + \KK_\delta f_1 &\text{ in }\UU^\eps \\
		\gamma_- f_2&=&\RRR \gamma_+f_2 &\text{ on }\Gamma_{-}^\eps \\
		 f_{2,t=0}&=& 0 &\text{ in }\OO^\eps,
	\end{array}\right.\label{eq:PLRBEReg}
\ee
respectively then, at least formally, $f = f_1+f_2$ is a solution of Equation~\eqref{eq:RBE}-\eqref{eq:RBEBC}-\eqref{eq:RBEIC}.
We proceed to use the stretching method as during Sections \ref{sec:Linfty} and \ref{sec:LinftyCylinder} to provide analogue a priori estimates for these problems and we construct their solutions by following the ideas from \cite[Theorems 6.1 and 6.9]{BriantGuo16}.

\subsection{Organisation of the paper}

This paper is structured as follows. 

In Section~\ref{sec:Hypo} we generalize the hypocoercivity theory developed in \cite{MR4581432} to cylindrical domains, and as a consequence we deduce a decay estimate for the solutions of Equation~\eqref{eq:LRBE} in the Hilbert space $\HH$.

During Section~\ref{sec:Linfty} we consider Assumption~\ref{item:RH1} to hold, and we use the stretching method developed in \cite{GuoZhou18} to obtain a $L^2-L^\infty$ type of decay estimate with strongly confining weights, for the solutions of Equation~\eqref{eq:PLRBE}. 

In Section~\ref{sec:LinftyCylinder} we consider Assumption~\ref{item:RH2} to hold, and still working with strongly confining weights, we generalize the ideas and methods from Section~\ref{sec:Linfty} to cylindrical domains.

We devote Section~\ref{sec:AprioriWeakConf} to study the linearized version of Equation~\eqref{eq:PLRBEDiss} with weakly confining weight functions. We use the stretching method as during Sections \ref{sec:Linfty} and \ref{sec:LinftyCylinder}, and the dissipative properties of the operator $\AA_\delta$ to obtain weighted $L^\infty$ decay estimates.

Section~\ref{sec:Toolbox} then addresses the well-posedness of transport equations with non-local terms. The focus is particularly in obtaining the well-posedness of Equation~\eqref{eq:PLRBE}, the linearized version of Equation~\eqref{eq:PLRBEDiss}, and Equation~\eqref{eq:PLRBEReg}.

At last, we dedicate Sections \ref{sec:SolutionsWWWWinfty} and \ref{sec:SolutionsWWWW0} to prove Theorem~\ref{theo:MainRescaled} for strongly confining and weakly confining weight functions respectively.

\subsection{Notation} \label{ssec:Notations}
During this paper we consider $\mu_L$ to be the Lebesgue measure, and when it is clear we will denote it as $\lv \cdot \rv = \mu_L$.  
We define $d\sigma_x$ as the Lebesgue measure on the boundary set $\partial \Omega^\eps$. A  $3-$dimensional ball of radius $r>0$ will be denoted as $B_r$ and $D:=\lv B_1\rv$. Moreover, for a point $z\in \R^3$ we denote its component as $z_i$ or $z^i$, with $i\in \{1,2,3\}$. $H^k(Z)$ is the Sobolev space for $L^2(Z)$ functions with $k>0$ weak derivatives in $L^2(Z)$. Furthermore, we define the tensor product between two vector $a,b\in \R^3$ as $a\otimes b := (a_ib_j)_{1\leq i,j\leq 3} \in \M_3$, the space of square matrices of size 3. At last, or $a,b>0$ we say that $a\lesssim b$ when there is a constant $c>0$ such that $a\leq cb$, and throughout this paper we will denote as $C$, a constant that---unless said otherwise---might change from one line to the other.

\section{Hypocoercivity}\label{sec:Hypo}
We dedicate this section to prove the following hypocoercivity result.  

\begin{theo}\label{theo:Hypo}
Consider either Assumption \ref{item:RH1} or \ref{item:RH2} to hold. There are constructive constants $\kappa>0$ and $C\geq 1$ such that for any solution $f$ of Equation~\eqref{eq:LRBE} there holds
 \begin{equation}
    \lVert f_t \rVert_{ \HH}\leq Ce^{-\kappa\eps^2 t}\lVert f_0\rVert_{ \HH} \qquad \forall t\geq0. \label{eq:Hypo}
\end{equation}
Furthermore, there is a norm $\lvvv \cdot \rvvv$ equivalent to the usual norm of $\HH$ uniformly in $\eps$, i.e there is a constant $c>0$ independent of $\eps$ such that
\be\label{eq:ClassicHypoEquivalence}
c^{-1} \lvv f\rvv_{\HH} \leq \lvvv f\rvvv \leq c \lvv f\rvv_{\HH},
\ee
and a constant $\kappa^\star>0$ for which there holds
 \begin{equation}
 \lvvv f_t \rvvv\leq e^{-\kappa^\star \eps^2 t}\lvvv f_0\rvvv \qquad \forall t\geq0. \label{eq:HypoEquiv}
\end{equation}
\end{theo}

\begin{rem}\label{rem:HypoRH1}
Theorem~\ref{theo:Hypo} under the assumption \ref{item:RH1} is nothing but \cite[Theorem 1.1]{MR4581432} and  \cite[Theorem 5.1]{MR4581432} up to the change of variables performed in Subsection \ref{ssec:Strategy}. Therefore our proof will focus on proving the hypocoercivity of Equation~\eqref{eq:LRBE} under Assumption \ref{item:RH2}. 
\end{rem}
\subsection{Poisson equation in the cylinder}\label{ssec:Poisson}
Consider $\Omega^\eps$ satisfying the geometrical Assumption \ref{item:RH2} and let $\xi\in L^2(\Omega^\eps)$. During this subsection we study the following Poisson equation
\begin{equation}\label{eq:PoissonEquation}
\left\{\begin{array}{rcll}
-\Delta u  &= &\xi & \text{ in }\Omega^\eps  \\
(2-\alpha(x)) \partial_n u + \alpha(x) u &=& 0& \text{ on } \partial\Omega^\eps,
\end{array}\right.
\end{equation}
where $\alpha$ is chosen satisfying one of the following conditions:
\begin{enumerate}
    \item[(P1)]\label{item:P1} either $\alpha = \Ind_{\Lambda^\eps_1\cup\Lambda^\eps_2}$,
\item[(P2)]\label{item:P2} or we assume $ \lla \xi \rra_{\Omega^\eps} =0$, and we take $\alpha\equiv 0$. 
\end{enumerate}
We consider then the functional spaces
\begin{equation*}
    V_1:= H^1(\Omega^\eps) \qquad\text{ and }\qquad V_2:=\left\{ u\in H^1(\Omega^\eps), \, \lla u \rra_{\Omega^\eps} =0\right\} ,
\end{equation*}
from where we define
\begin{equation*}
        V_k:=\left\{\begin{array}{cc}
        V_1 & \text{if \ref{item:P1} holds,} \\
        V_2 & \text{if \ref{item:P2} holds.}
    \end{array}\right.
\end{equation*}
and we will dedicate the rest of this subsection to prove the following well-posedness and regularity result.

\begin{theo}\label{theo:PoissonRegularity}
Let $\Omega^\eps$ be a cylindrical domain as defined in \ref{item:RH2} and let either \ref{item:P1} or \ref{item:P2} boundary conditions hold. For any $\xi\in L^2(\Omega^\eps)$ there exists a unique variational solution $w \in H^2(\Omega^\eps)\cap V_k$ to the Poisson Equation~\eqref{eq:PoissonEquation}, i.e there holds
\be\label{eq:PoissonVariationalSln}
\int_{\Omega^\eps} \grad w(x) \cdot \grad v(x) \, \dx + \int_{\partial\Omega^\eps} {\alpha(x) \over 2-\alpha(x) } w(x)v(x) \, \d\sigma_x  = \int_{\Omega^\eps} \xi(x) \, v(x) \dx \qquad \forall v\in V_k.
\ee
Furthermore, there holds
\beqn
	\lVert w\rVert_{H^2(\Omega^\eps)}\leq C \eps^{-2} \lVert \xi\rVert_{L^2(\Omega^\eps)},\label{H2Schr}
\eeqn
for some constructive constant $C>0$, independent of $\eps$.
\end{theo}

\begin{proof} We split the proof into three steps.

\medskip\noindent
\emph{Step 1. (Existence and uniqueness of variational solutions)} The analysis for this step is classical and similar to the arguments from the Step 1 of the proof of \cite[Theorem 2.2]{MR4581432}, so we will only sketch it. We introduce the bilinear form
\begin{equation*}
    a(u,v):=\int_{\Omega^\eps}\nabla u(x)\nabla v(x) \dx + \int_{\partial\Omega^\eps} {\alpha(x) \over 2-\alpha(x) } u(x)v(x) \d\sigma_x, 
\end{equation*}
and we observe that $a$ is continuous in $V_k\times V_k$. On the other hand, 
by using the Poincaré-Wirtinger's inequality (see for instance \cite[Proposition 2.1]{MR4581432}) together with classical scaling arguments, we deduce that
\be\label{eq:coercivityPoisson}
a(u,u)\gtrsim \eps^2 \lVert u\rVert^2_{H^1(\Omega^\eps)},
\ee
The bilinear operator $a$ is thus coercive and the Lax-Milgram theorem implies the existence and uniqueness of a variational solution $w\in V_k$ for the Poisson Equation~\eqref{eq:PoissonEquation} in the sense of \eqref{eq:PoissonVariationalSln}. 
Furthermore, from the variational formulation \eqref{eq:PoissonVariationalSln} and the previous coercivity estimate \eqref{eq:coercivityPoisson}, we further deduce that
\beqn\label{eq:H1Poisson}
	\lVert w\rVert_{H^1(\Omega^\eps)}\lesssim \eps^{-2} \lVert \xi\rVert_{L^2(\Omega^\eps)}.
\eeqn
In particular, we classically deduce that there is a trace function $w \restrict  {\partial \Omega^\eps} \in L^2(\partial\Omega^\eps)$ and there holds
 $$
 \lvv w \restrict  {\partial \Omega^\eps} \rvv_{L^2(\partial \Omega^\eps)}\lesssim  \eps^{-2} \lvv \xi\rvv_{L^2(\Omega^\eps)}.
 $$
This concludes the well-posedness of Equation~\eqref{eq:PoissonEquation}.

\medskip\noindent
\emph{Step 2. (Domain reflection)} We consider now the extended domain $\widehat{\Omega}^\eps := (-2L^\eps, 2L^\eps)\times \Omega^\eps_0$ and we define 
$$
    \widehat{\Lambda^\eps_1} := \{-2L^\eps\}\times {\Omega_{0}^\eps}, \qquad
    \widehat{\Lambda^\eps_2} := \{2L^\eps\}\times {\Omega_{0}^\eps}, \qquad
    \text{and} \qquad
    \widehat{\Lambda^\eps_3} := (-2L^\eps, 2L^\eps)\times \partial{\Omega_{0}^\eps}.
$$
Now for any $x\in \R^3$, we write it as $x = (x_1, \bar x)\in \R\times \R^{2}$ and for any function $\phi$ with domain in $\Omega^\eps$ we define a new function $\widehat \phi$ with domain in $\widehat{\Omega^\eps}$ as follows
\begin{equation}
	\widehat{\phi}(x)= \left\{\begin{array}{ll}
		\phi(-x_1 + 2L^\eps, \bar x)& \text{ if } x_1\in (-2L^\eps, -L^\eps)\\
		\phi(x_1, \bar x)& \text{ if } x_1\in (-L^\eps, L^\eps)\\
		\phi(-x_1- 2L^\eps, \bar x)& \text{ if } x_1\in (L^\eps, 2L^\eps).\\
	\end{array}\right.\label{eq7.3:ReflectionDef}
\end{equation}
It is easy to check that, defined this way, $\widehat{w}$ is a variational solution of the extended Poisson problem
\begin{equation}\left\{
\begin{array}{rcll}\label{eq:PoissonExtended}
-\Delta \widehat u  &= &\widehat \xi & \text{ in }\widehat{\Omega^\eps}  \\
(2-\widehat \alpha(x)) \partial_n \widehat u + \widehat \alpha(x)  \widehat u &=& 0& \text{ on } \partial\widehat{\Omega^\eps}.
\end{array}\right.
\end{equation}
As a consequence, repeating the arguments of Step 1, we deduce that $\widehat w$ satisfies the following estimates
 \begin{equation}\label{eq:H1PoissonExtended}
	\lVert \widehat w\rVert_{H^1(\widehat{\Omega^\eps})}\lesssim \eps^{-2} \lVert\widehat \xi\rVert_{L^2(\widehat{\Omega^\eps})} \quad \text{ and } \quad  \lvv \widehat w \restrict  {\partial \Omega^\eps} \rvv_{L^2(\partial \widehat\Omega^\eps)}\lesssim  \eps^{-2} \lvv \xi\rvv_{L^2(\widehat \Omega^\eps)}.
\end{equation}

\medskip\noindent
\emph{Step 3. (Regularity)} Using the partition of unity (see for instance \cite[Lemma 9.3]{Brezis}) we obtain the existence of a collection of open balls of radius $r>0$, $(B^j_r)_{j=1}^N\subset \R^3$ covering $ \partial \widehat{\Omega^\eps}$, and a family of functions $\varrho_j\in C^\infty(\R^3,\R)$ with $j=0,\ldots, N$, such that 
\begin{enumerate}
\item[(i)] $0\leq \varrho_j\leq 1$ for every $j =0,\ldots, N$, and $\sum_{j=0}^N\varrho_j = 1$ on $\R^3$, 
\item[(ii)] moreover, $ \supp \varrho_0\subset \R^3\setminus \partial \widehat{\Omega^\eps}$ and for every $j=1,\ldots, N$, there holds that $\varrho_j\in C^\infty_c(B^j_r)$.
\end{enumerate}
In particular, we choose $r>0$ small enough such that the balls $B_r^j$ covering $\Lambda_3^\eps$ do not intersect $\widehat{\Lambda^\eps_1}$ or $\widehat{\Lambda^\eps_2}$, and we call $\JJ$ the sets of such indexes plus $0$, i.e
$$
\JJ=\left\{j \in \{1,\ldots,N\} \left\rvert \,   \bigcup B^j_r \supset  \Lambda^\eps_3 \right.  \, \left\lvert \,   B^j_r \cap \widehat{\Lambda^\eps_1}  =\emptyset \right.\, \left\lvert\,    B^j_r \cap \widehat{\Lambda^\eps_2}  =\emptyset \right. \right\}\cup\{0\}.
$$
Then we classically have, see for instance \cite[Theorem 9.26 Steps $C_1$ and $C_2$]{Brezis}, that for every $j\in \JJ$ there holds $\widehat w\varrho_j\in H^2(\widehat \Omega^\eps)$ and
\be\label{eq:H2localPoissonExtended}
\lVert \Delta(\widehat w\varrho_j) \lVert_{H^2(\widehat\Omega^\eps)} \leq \widehat C \eps^{-2} \lVert \widehat \xi\lVert_{L^2(\widehat\Omega^\eps)},
\ee
where $\widehat C>0$ is a constant independent of $\eps$. We deduce then that for every $\varphi \in C^\infty_c(\R^3)$ we have
\begin{multline*}
\int_{\cup_{j\in \JJ} B^j_r\cap\widehat{\Omega^\eps}} \grad \widehat w \cdot  \grad \varphi = \int_{\cup_{j\in \JJ} B^j_r\cap \widehat{\Omega^\eps}} \grad \widehat w \cdot  \grad \varphi \, \varrho_j =  \sum_{j\in \JJ} \left( \int_{B^j_r\cap \widehat{\Omega^\eps}} \grad (\widehat w \varrho_j) \cdot  \grad \varphi   -  \int_{B^j_r\cap \widehat{\Omega^\eps}} \widehat w \grad  \varrho_j\cdot  \grad \varphi \right) \\
= \sum_{j\in \JJ} \left( \int_{ B_r^j\cap \partial\widehat{\Omega^\eps}} \varphi \grad(\widehat w \varrho_j) \cdot n_x  -\int_{B^j_r\cap \widehat{\Omega^\eps}}  \Delta (\widehat w \varrho_j) \varphi  -\int_{ B_r^j\cap \partial\widehat{\Omega^\eps}} \widehat w \varphi \grad \varrho_j  \cdot n_x + \int_{B^j_r\cap \widehat{\Omega^\eps}} \Div(\widehat w \grad  \varrho_j)  \varphi \right) 
\end{multline*}
where we have used that $\widehat w \varrho_j\in H^2(\widehat \Omega^\eps)$, $\widehat w\grad \varrho_j\in H^1(\widehat\Omega^\eps)$  and we have performed an integration by parts on the third line. We then compute
$$
\left\lvert \int_{\partial\widehat{\Omega^\eps}} \varphi \grad(\widehat w \varrho_j) \cdot n_x \right\rvert= \left\lvert \int_{\partial\widehat{\Omega^\eps}} \varphi \grad(\widehat w) \cdot n_x \varrho_j \right\rvert+  \left\lvert \int_{\partial\widehat{\Omega^\eps}} \varphi \grad(\varrho_j) \cdot n_x \widehat w \right\rvert \leq \int_{\partial\widehat{\Omega^\eps}}  \lvert \varphi \widehat w \rvert \left\lvert {\alpha \over 2-\alpha}\varrho_j + \grad \varrho_j\cdot n_x \right\rvert
$$
where we have used the boundary conditions on the second line. Then by using the Cauchy-Schwartz inequality we have
\bean
\int_{\cup_{j\in \JJ} B^j_r\cap\widehat{\Omega^\eps}} \grad \widehat w \cdot  \grad \varphi&\lesssim & \lVert \varphi\rVert_{L^2(\R^3)}  \sum_{j\in \JJ} \left(  \lVert \Delta (\widehat w \varrho_j)\rVert_{L^2(B^j_r\cap \widehat{\Omega^\eps})}    + \lVert \grad \widehat w  \rVert_{L^2(B^j_r\cap \widehat{\Omega^\eps})}      +\lVert \widehat w \rVert_{L^2(B^j_r\cap \partial \widehat{\Omega^\eps})}         \right)    \\
&\lesssim & \lVert  \varphi \rVert_{L^2(\R^3)}   \eps^{-2} \lVert \widehat \xi\rVert_{L^2( \widehat{\Omega^\eps})}
\eean
where we have used \eqref{eq:H1PoissonExtended} and \eqref{eq:H2localPoissonExtended} to obtain the second inequality. From the previous result we deduce that $\widehat w\in H^2(\cup_{j\in \JJ}  B^j_r\cap \widehat{\Omega^\eps})$ and 
$$
\lVert \Delta \widehat w\rVert_{L^2(\cup_{j\in \JJ} B^j_r\cap\widehat{\Omega^\eps})}\lesssim \eps^{-2} \lVert \widehat \xi\rVert_{L^2(\widehat{\Omega^\eps})},
$$
which implies in particular that $w\in H^2(\Omega^\eps)$. Moreover, we remark that the above estimate together with \eqref{eq:H1PoissonExtended} implies that
$$
\lVert  w\rVert_{H^2(\Omega^\eps)}  \leq  \lVert \widehat w\rVert_{H^2(\cup_{j\in \JJ} B^j_r\cap\widehat{\Omega^\eps})} \lesssim \eps^{-2} \lVert \widehat \xi\rVert_{L^2(\widehat{\Omega^\eps})} \lesssim \eps^{-2} \lVert  \xi\rVert_{L^2(\Omega^\eps)},
$$
and this concludes the proof.
\end{proof}

\subsection{Lamé system in the cylinder}\label{ssec:LameSystem}
Consider $\Omega^\eps$ satisfying Assumption \ref{item:RH2}, and consider $\Xi\in L^2(\Omega^\eps)$. We dedicate this subsection to study the following Lamé system of elliptic equations
\begin{equation}
	\left\{\begin{array}{rlll}
		-\Div (\nabla^s U) &=& \Xi &\text{ in }\Omega^\eps\\
		U\cdot n(x)&=&0 &\text{ on }\partial\Omega^\eps\\
		(2-\iota^\eps(x))[\nabla^sU \cdot n(x) - (\nabla^sU: n(x)\otimes n(x))n(x)]+\iota^\eps(x) U&=&0 &\text{ on }\partial\Omega^\eps.
	\end{array}\right. \label{eq:LameSystem}
\end{equation}
We consider the functional space
\be
\UUU(\Omega^\eps):=\{ U\in H^1(\Omega^\eps), \, U(x) \cdot n(x) = 0 \,\text{ on } \partial\Omega^\eps \}, 
\ee 
and we will prove the following well-posedness and regularity result.
\begin{theo}\label{thm:LameRegularity}
Let $\Omega^\eps$ be a cylindrical domain as described in \ref{item:RH2}. For any $\Xi\in L^2(\Omega^\eps)$ there is a unique variational solution $W \in \UUU(\Omega^\eps) \cap  H^2(\Omega^\eps)$ to the Lamé system \eqref{eq:LameSystem}, i.e for any $V\in \UUU(\Omega^\eps)$ there holds
\be\label{eq:WeakFormulationLame}
\int_{\Omega^\eps}\nabla^sW (x):\nabla^sV(x) \dx + \int_{\partial\Omega^\eps} {\iota^\eps(x) \over 2-\iota^\eps(x) }\, W(x)\cdot V(x)   \d\sigma_x = \int_{\Omega^\eps} \Xi (x) \cdot V(x) \dx.
\ee
Furthermore, there holds
\begin{equation}
    \lVert W\rVert_{H^2(\Omega^\eps)}\leq C\eps^{-2}\lVert \Xi\rVert_{L^2(\Omega^\eps)},
\end{equation}
for some constructive constant $C>0$, independent of $\eps$.
\end{theo}

\subsubsection{Preliminary inequalities}\label{sssec:Korn} Before proving Theorem~\ref{thm:LameRegularity} we will prove some preliminary results in order to provide the coercivity for Equation~\eqref{eq:LameSystem}.

One of the main inequalities to study this problem will be Korn's inequality, and we refer to \cite[ Theorem 3.2, Chapter 3]{MR0521262}, \cite{CiarletCiarlet} and \cite{DesvillettesVillani02} for more on the history and applications of this type of inequality.

\begin{lem}\label{lem:KornInequlities}
Consider a vector-field $U\in H^1(\Omega^\eps)$, the following statements hold.
\begin{enumerate}
\item[(L1)]\label{item:L1} \emph{(Poincaré's inequality)} If $U(x)\cdot n(x)=0$ for \emph{almost} every $x\in\partial\Omega^\eps$ we have	that
\begin{equation*}
	\eps^2\lVert U\rVert^2_{L^2(\Omega^\eps)} \lesssim\lVert \grad U\rVert^2_{L^2(\Omega^\eps)}.
\end{equation*} 

\item[(L2)]\label{item:L2}  \emph{(Korn's inequality)} There is a constant $C>0$, independent of $\eps$, for which there holds
$$
\lVert\grad U\rVert_{L^2(\Omega^\eps)}^2\lesssim_C \lVert\grad^s U\rVert^2_{L^2(\Omega^\eps)} + \left\lVert \sqrt{\iota^\eps\over 2-\iota^\eps} U\right\rVert^2_{L^2(\partial\Omega^\eps)}.
$$	
\end{enumerate}
\end{lem}

\begin{proof}
The proof of Poincaré's inequality given by \ref{item:L1} is nothing but a repetition of the compactness arguments used during the proof of \cite[Lemma 2.7]{MR4581432}, and we deduce the dependence on $\eps$ by a standard scaling argument.

Moreover, the proof of the Korn inequality given in \ref{item:L2} follows by using the exact same arguments as during its proof in \cite[Lemma 2.5]{MR4581432}, with a few extra considerations. First we deduce the validity of \cite[Lemma 2.3]{MR4581432} for cylindrical domains by arguing as during \cite[Section 5]{DesvillettesVillani02}. Secondly, we repeat the proof of \cite[Lemma 2.4]{MR4581432} by taking into account not to choose a point in the singular set of the boundary of the cylinder and we follow the same ideas.  
\end{proof}

\subsubsection{Proof of Theorem~\ref{thm:LameRegularity}}\label{sssec:ProofLame} 
We divide the proof into two steps.

\medskip\noindent
\emph{Step 1. (Existence and uniqueness of variational solutions)} We consider the bilinear operator $A:\UUU(\Omega^\eps)\times \UUU(\Omega^\eps) \to \R$ defined as
\begin{equation*}
    A(U,V)=\int_{\Omega^\eps}\nabla^sU (x):\nabla^sV(x) \dx + \int_{\partial\Omega^\eps} {\iota^\eps(x) \over 2-\iota^\eps(x)}\, U(x) \cdot V(x) \d\sigma_x.
\end{equation*}	
We observe that $A$ is continuous on $\UUU(\Omega^\eps)\times \UUU(\Omega^\eps)$ and \ref{item:L1} and \ref{item:L2} from Lemma~\ref{lem:KornInequlities} together imply that $A$ is also coercive.
We then may apply the Lax-Milgram theorem and we get that for every $\Xi\in L^2(\Omega^\eps)$ there is a unique variational solution $W\in \UUU(\Omega^\eps)$ of the Lamé system \eqref{eq:LameSystem}. Furthermore, using the coercivity of the operator $A$ and the weak formulation \eqref{eq:WeakFormulationLame} we deduce that
\begin{equation}
	\lVert W\rVert_{H^1(\Omega^\eps)}\leq C_L \, \eps^{-2}\lVert \Xi\rVert_{L^2(\Omega^\eps)}.\label{eq8.5:H1Lame}
\end{equation}
for some constant $C_L>0$ independent of $\eps$.

\medskip\noindent
\emph{Step 2. (Domain reflection and regularity)} Following the same reflection method as for the Poisson equation, we extend the  functions component-wise as described in \eqref{eq7.3:ReflectionDef}. We observe that $\widehat\Xi \in L^2(\widehat{\Omega^\eps})$ and that $\widehat W$ is a variational solution of the problem 
\begin{equation}
	\left\{\begin{array}{rccl}
		-\Div(\nabla^s\widehat{W}) &=& \widehat{\Xi} &\text{ in }\widehat{\Omega^\eps}\\
		\widehat{W}\cdot n(x)&=&0 &\text{ on }\partial\widehat{\Omega^\eps}\\
		\left(2-\widehat {\iota^\eps}\right)[\nabla^s\widehat{W} \cdot n - (\nabla^s\widehat{W}:n\otimes n)n]+\widehat {\iota^\eps} \, \widehat{W}&=&0 &\text{ on }\partial\widehat{\Omega^\eps}.
	\end{array}\right.\label{eq:LameSystemExtended}
\end{equation}
Moreover, by arguing as in the Step 1 we have that
\begin{equation}\
	\lVert \widehat {W }\rVert_{H^1(\widehat{\Omega^\eps})}\lesssim \eps^{-2}\rVert \Xi \rVert_{L^2(\widehat{\Omega^\eps})}.\label{eq:H1LameExtended}
\end{equation}
Repeating the arguments from \cite[Theorem 2.11, Steps 3 and 4]{MR4581432} together with \eqref{eq:H1LameExtended}, we deduce that for every small enough open set $\GG \subset (-3L^\eps/2, 3L^\eps/2) \times \Omega_0^\eps$ there holds 
$$
\lVert  \widehat W \lVert_{H^2(\widehat\Omega^\eps\cap \GG)} \lesssim \eps^{-2} \lVert \widehat \xi\lVert_{L^2(\widehat\Omega^\eps)}.
$$
We then choose a finite family of open sets $(\GG_j)_{j=1}^N$ with $\GG_j  \subset (-3L^\eps/2, 3L^\eps/2) \times \Omega_0^\eps$ and such that 
$\bigcup_{j=1}^N \GG_j \supset (-L^\eps, L^\eps) \times \Omega_0$. Using then the above estimate we deduce that $W\in H^2(\Omega^\eps)$ and there holds
$$
\lVert  W\rVert_{H^2(\Omega^\eps)}  \leq  \lVert \widehat W\rVert_{H^2(\bigcup_{j=1}^N \GG_j)} \lesssim \eps^{-2} \lVert \widehat \Xi\rVert_{L^2(\widehat{\Omega^\eps})} \leq 3\eps^{-2} \lVert  \Xi\rVert_{L^2(\Omega^\eps)}.
$$
This completes the proof. \qed

\subsection{Proof of Theorem~\ref{theo:Hypo}}\label{ssec:ProofHypo}
As discussed during Remark \ref{rem:HypoRH1}, Theorem~\ref{theo:Hypo} under hypothesis \ref{item:RH1} is nothing but \cite[Theorem 1.1]{MR4581432} and \cite[Theorem 5.1]{MR4581432}. Moreover, the existence of the norm $\lvvv \cdot \rvvv$, and the fact that it is equivalent to $\lvv \cdot \rvv_\HH$ uniformly in $\eps$, is obtained by arguing exactly as during the proof of \cite[Theorem 1.1]{MR4581432}, but we also refer to \cite[Theorem A.1]{carrapatoso2025navierstokeslimitkineticequations} together with the comments and references on  \cite[Appendix A]{carrapatoso2025navierstokeslimitkineticequations}.

Finally, the proof of Theorem~\ref{theo:Hypo} under the assumptions \ref{item:RH2} is mainly a repetition of the proof of \cite[Theorem 1.1]{MR4581432} using instead the regularity results from Theorems \ref{theo:PoissonRegularity} and \ref{thm:LameRegularity}. 
\qed

\section{Stretching method for strongly confining weights in smooth domains}\label{sec:Linfty}

During this section we assume there to hold assumptions \ref{item:RH1}, and we study Equation~\eqref{eq:PLRBE}. We dedicate this section to prove the following result.
\begin{prop}\label{prop:LinftyEstimatePerturbedFiniteT} 
Consider Assumption \ref{item:RH1} to hold, $\omega_1\in\WWWW_1$ a strongly confining admissible weight function, $G:\UU^\eps\to \R$ satisfying $\lla G_t\rra_{\OO^\eps} = 0$ for every $t\geq 0$, and let $f$ be a solution of Equation~\eqref{eq:PLRBE}. There are constructive constants $\eps_1, \theta >0$ such that for every $\eps\in (0,\eps_1)$ there holds 
\be
\lVert  f_t \rVert_{L^{\infty}_{\omega_1}(\bar\OO^\eps)} \leq  C e^{-\theta \eps^2 t} \left(   \lVert f_0\rVert_{L^{\infty}_{\omega_1}(\OO^\eps)} + \underset{s\in[0,t]}{\sup}\left[ e^{\theta \eps^2 s} \lVert G_s \rVert_{L^{\infty}_{\omega_1\nu^{-1}}(\OO^\eps)}\right]\right)  \quad \forall t\geq 0, \label{eq:LinftyPerturbedFiniteTimeDecay}
\ee
for some constant $C = C(\eps)>0$ such that $C(\eps) \to \infty$ as $\eps \to 0$. 
\end{prop}

\begin{rem}
Throughout this paper (at least when working at the level of a priori estimates) we will use the notation
\be\label{def:NormBarOOeps}
\lVert  f_t \rVert_{L^{\infty}_{\omega}(\bar \OO^\eps)} := \lVert  f_t \rVert_{L^{\infty}_{\omega}(\OO^\eps)} + \lVert \gamma f_t \rVert_{L^{\infty}_{\omega}(\Sigma^\eps)}, 
\ee
for any admissible weight function $\omega$, and we recall that $L^{\infty}_{\omega}(\Sigma^\eps) :=L^{\infty}(\Sigma^\eps, \omega(v) \, \dv\d\sigma_x )$.
\end{rem}

\begin{rem}
The computations leading to the proof of Proposition~\ref{prop:LinftyEstimatePerturbedFiniteT} use the stretching method developed in \cite{GuoZhou2024} (see also the previous version \cite{GuoZhou18}). 
Nonetheless, we point out that the exponential decay obtained in Proposition~\ref{prop:LinftyEstimatePerturbedFiniteT} is actually an improvement of \cite[Proposition 4.1]{GuoZhou2024}. This is done by using a more delicate way of combining the hypocoercivity exponential decay from Theorem~\ref{theo:Hypo} and the weighted $L^\infty$ estimate obtained by using the stretching method. Such type of way of combining norms is reminiscent of \cite[Proposition 3.6]{CM_Landau},  \cite[Proposition 3.2]{Carrapatoso_2016}, \cite[Proposition 4.1]{MR3591133}, and \cite[Proposition 7.2]{evans2025existencestabilitynonequilibriumsteady}.
\end{rem}

\subsection{Auxiliary problem in finite time and backwards trajectories}\label{ssec:Trajectories}
Here and below, during this section, we consider an arbitrary $T>0$, we take $G$ as defined in Subsection \ref{ssec:Strategy}, and we study the following evolution equation
\be
	\left\{\begin{array}{llll}
		\partial_{t} f &=& \LLL f + G&\text{ in }\UU^\eps_T:= (0,T)\times \OO^\eps\\
		\gamma_- f&=&\RRR \gamma_+f  &\text{ on }\Gamma_{-, T}^\eps:= (0,T) \times \Sigma_-^\eps\\
		 f_{t=0}&=& f _0  &\text{ in }\OO^\eps.
	\end{array}\right.\label{eq:PLRBE_T}
\ee
For any fixed $(t, x, v) \in  \UU_T^\eps$, we recall the backwards trajectories defined in \eqref{eq:BackwardsTrajectory}, and we define the singular set along the backwards trajectory
\be\label{def:Sx}
 S_{x}:=\{v\in\R^3;\ \exists \, t\in [0,T], \, n(X(t_b(x,v),t,x,v))\cdot v=0\}.
\ee

\subsection{Stretching lemma} We have now the following lemma, which is mainly an application of  \cite[Lemma 4.3]{GuoZhou2024}. 

\begin{lem}
\label{lemma3.2:GeneralRefl}
Assume \ref{item:RH1} to hold. Let $(t,x,v)\in[0,T]\times \partial\Omega^{\varepsilon}\times\{ \lvert v\rvert\leq M, \lvert n_x\cdot v\rvert>\eta\} $ for some $M,\eta>0$ and such that $(x,v) \in \Sigma^\eps_+$. There exists $\eps_R = \varepsilon_{R} (\eta, M, T) \lesssim \eta M^{-2} T^{-1}$ such that for every $\varepsilon \in (0, \varepsilon_R(\eta,M, T))$ there is at most one bounce against the boundary $\partial\Omega_{\varepsilon}$ along the backwards trajectory.
\end{lem}

\begin{rem}
The proof of Lemma~\ref{lemma3.2:GeneralRefl} follows mainly using the computations from \cite[Lemma 4.3]{GuoZhou2024} in a very close way to that of  \cite[Lemmas 4.4 and 4.5]{GuoZhou2024}, and we remark that the arguments are, in a certain sense, local along trajectories.
\end{rem}

\subsection{Preliminary lemmas}
We present now several results which will be useful during the rest of this paper. 

\begin{lem}\label{lem:Kproperties}
Let $h:\UU^\eps\to \R$, we consider $\omega$ to be an admissible weight function, and we define $\varsigma(v) = e^{\zeta \lv v\rv^2}$ for $\zeta \in (0,1/2)$. The following statements hold. 
\begin{enumerate}[leftmargin=*]
\item[(K1)]\label{item:K1} For every $N>0$ there is $m= m(N)\geq 1$ such that 
$$
\varsigma(v) \lvert K h (v)\rvert \leq {1\over N}\,  \lVert  h\rVert_{L^{\infty}_{\varsigma}(\R^3)} +  K_m \left(\lvert  \varsigma(v) h (v) \rvert \right)
$$
where $\displaystyle K_m h:= \int_{\R^3} k_m(v,v_*) h(v_*) dv_*$ and
\be\label{eq:defkm}
k_m(v,v_*):=  c_k \mathbf{1}_{\lvert v_*\rvert \leq m}\mathbf{1}_{\lvert v\rvert \leq m}\mathbf{1}_{\lvert v_*-v\rvert \geq {1/ m}} \, \left(\lvert v-v_*\rvert + \lvert v-v_*\rvert^{-1}\right),
\ee
for some constructive constant $c_k>0$. It is also worth remarking that $k_m\leq C_k \,m$, where $C_k = 3c_k$.


\item[(K2)]\label{item:K3} For every $\nu_2\in (0, \nu_0)$, every $(t,x,v)\in \UU^\eps$, and every $\sigma \in [0,t]$ there holds
$$
\omega(v) \left\lvert  S_{\TTT} *_\sigma h (t, x, v) \right\rvert  \leq {\nu_1\over \nu_0 - \nu_2} e^{-\nu_2 t} \underset{s\in [0,t]}{\sup}\left[e^{\nu_2 s}\lVert  h_s\rVert_{L^{\infty}_{\omega\nu^{-1}}(\OO^\eps)}\right] .
$$
where we remark that $S_\TTT *_\sigma h$ is in the sense of \eqref{def:convolution_sigma} taking $\AA = Id$,

\item[(K3)] \label{item:K4} Let $0\leq r\leq s\leq t$, $x\in\Omega^\eps$, $u, v_*\in\R^3$ and consider a function $\XX = \XX(t,s, r ,x,u, v_*)$ such that 
$$
\lvert\mathrm{det} \,\grad_{v_*}\XX (v_*)\rvert\geq C_0\lvert s-r\rvert^3 ,
$$ 
for some constant $C_0>0$. Then for any $m_1, m_2>0$, and every $\alpha>0$ there holds 
\begin{multline*}
\int_0^{t} \int_{\lvert v_*\rvert \leq m_1} \int_0^s e^{-\nu_0(t-r)} \int_{\lvert u\rvert\leq  m_2} \omega (u) \lv  h  (r,\XX,u) \rv \, \d u \d r \d v_*\d s 
    \leq D^2\, m_1^3\, m_2^3\, \alpha\,  t\,  e^{-\nu_0 t}\underset{s\in[0,t]}{\sup}\left[ e^{\nu_0s}\lVert  h_s\rVert_{L^{\infty}_{\omega}(\OO^{\eps})}\right]\\
    + Dm_1^{3/2}m_2^{3/2} \, \omega(m_2) \, t^2 \alpha^{-3/2}C_0^{-1/2} t^2e^{-\nu_0 t}\underset{s\in[0,t]}{\sup}\left[ e^{\nu_0 s}\lVert  h_s\rVert_{ L^2(\OO^\eps)}\right].
\end{multline*}
where we recall that $D>0$ is the volume of the $3-$dimensional ball of radius 1. 

\item[(K4)]\label{item:K5} For every $\alpha>0$ there are $M, \eta >0$ such that 
\bean
\int_{\R^3}  \lvert  h(t, x, u) \rvert (n_x\cdot u)_+ \d u &\leq& \alpha \lVert  h_t\rVert_{L^{\infty}_{\omega}(\OO^\eps)}+ \int_{\{\lvert u\rvert \leq M,\, \lvert n_x\cdot u\rvert >\eta\}}  \lvert  h(t, x, u)\rvert  (n_x\cdot u) \d u
\eean
\end{enumerate}
\end{lem}

\begin{rem}
It is worth emphasizing that the results from Lemma~\ref{lem:Kproperties} are independent of the geometry of the domain, so they hold under either Assumption \ref{item:RH1} or Assumption \ref{item:RH2}. 
\end{rem}

\begin{proof}[Proof of Lemma \ref{lem:Kproperties}] 
We proof each of the statements separately. 

\smallskip\noindent
\emph{Proof of \ref{item:K1}.} 
We observe that, using \cite[Lemma 3]{Guo09}, \cite{Grad1962} or \cite[Equation (2.17)]{MR1307620}, we have that 
\be\label{eq:Kcontrol3}
k(v,v_*) = \widetilde k(v,v_*) \exp\left(-{\lvert v\rvert^2 \over 4} + {\lvert v_*\rvert^2\over 4}\right) \quad \text{ with } \quad \widetilde k(v,v_*) \leq c_k \bar k(v,v_*),
\ee
for some  constant $c_k>0$, and where we have defined
$$
\bar k(v,v_*) = (\lvert v-v_*\rvert + \lvert v-v_*\rvert^{-1})   \exp\left( -{1\over 8} {(\lvert v_*\rvert^2 - \lvert v\rvert^2)^2\over \lvert v-v_*\rvert^2} - {1\over 8} \lvert v-v_*\rvert^2   \right).
$$
Furthermore, repeating the exact arguments used during the proof of \cite[Lemma 3]{Guo09} we have that for any $\zeta\in (0, 1/2)$, the quadratic form
$$
\QQQ: = -{1\over 8} {(\lvert v_*\rvert^2 - \lvert v\rvert^2)^2\over \lvert v-v_*\rvert^2} - {1\over 8} \lvert v-v_*\rvert^2  -{\lvert v\rvert^2 \over 4} + {\lvert v_*\rvert^2\over 4} + \zeta\lvert v\rvert^2 - \zeta \lvert v_*\rvert^2 ,
$$
is definite negative.
Thus there is $b>0$ such that
\be\label{eq:KdecayPart}
\int_{\R^3} \bar k(v,v_*) e^{-{\lvert v\rvert^2 \over 4} + {\lvert v_*\rvert^2\over 4}} \exp\left( {b \over 8} {(\lvert v_*\rvert^2 - \lvert v\rvert^2)^2\over \lvert v-v_*\rvert^2} + {b \over 8} \lvert v-v_*\rvert^2   \right) {\varsigma(v)\over \varsigma(v_*)} dv_* \leq {C_1\over 1+ \lvert v\rvert} ,
\ee
for some positive constant $C_1= C_1(b)$. 
We remark now that the above estimate implies that
\be
\bar k(v,v_*)e^{-{\lvert v\rvert^2 \over 4} + {\lvert v_*\rvert^2\over 4}}{\varsigma(v)\over \varsigma(v_*)} =  (\lvert v-v_*\rvert + \lvert v-v_*\rvert^{-1})   e^\QQQ \leq    (\lvert v-v_*\rvert + \lvert v-v_*\rvert^{-1}) \label{eq:Kcontrol5},
\ee
thus setting $A_m = \{\lvert v_*\rvert \leq m, \, \lvert v\rvert \leq m, \, \lvert v_*-v\rvert \geq {1/ m}\}$ for some $m>0$ to be chosen later, we compute
\bean
\varsigma(v) \lvert Kh(v) \rvert 
&\leq & c_k\lVert h\rVert_{L^\infty_{\varsigma}(\OO^\eps)} \int_{A^c_m} \bar k(v,v_*)e^{-{\lvert v\rvert^2 \over 4} + {\lvert v_*\rvert^2\over 4}}  {\varsigma(v)\over \varsigma(v_*)} dv_* + c_k\int_{A_m} (\lvert v-u\rvert + \lvert v-u\rvert^{-1}) \varsigma(v_*) \lvert h(v_*) \rvert dv_*\\
&\leq & {c_kC_1\over 1+m}\,  \lVert h\rVert_{L^\infty_{\varsigma}(\OO^\eps)} + \int_{\R^3} k_m(v,v_*) \varsigma(v_*) \lvert h(v_*) \rvert dv_* \leq  {1\over N} \lVert h\rVert_{L^\infty_{\varsigma}(\OO^\eps)} + K_m (\lvert  \varsigma(v) h(v)\rvert ) 
\eean
where we have used \eqref{eq:Kcontrol3} on the first inequality, \eqref{eq:Kcontrol5} to obtain the second, we have used \eqref{eq:KdecayPart} and the definition of $k_m$ on the third inequality, and we have chosen $m= \max( Nc_kC_1, 1)$ on the last line.

\medskip\noindent
\emph{Proof of \ref{item:K3}:}  We compute 
\bean
\left\lv \int_0^t S_{\TTT}(t-s)    h(s, x,v)   \omega (v)  \d s\right\rv 
 &\leq& \underset{s\in [0,t]}{\sup}\left[e^{\nu_2 s}\lVert  h_s \rVert_{L^{\infty}_{\omega\nu^{-1}}(\OO^\eps)}\right]   \int_0^t \nu(v)  e^{-\nu(v) (t-s) - \nu_2 \, s } \d s\\
   &\leq&{\nu(v) \over \nu(v) - \nu_2} e^{-\nu_2 \, t} \underset{s\in [0,t]}{\sup}\left[e^{\nu_2 s}\lVert  h_s\rVert_{L^{\infty}_{\omega\nu^{-1}}(\OO^\eps)}\right].
 \eean
At last we remark that \eqref{eq:Controlnu} implies that ${\nu(v) / (\nu(v) - \nu_2)} \leq {\nu_1 / (\nu_0-\nu_2)}$.

\medskip\noindent
\emph{Proof of \ref{item:K4}:} We split the integral on the LHS into $I_1 + I_2$ where 
\bean
I_1 &:=&\int_0^{t} \int_{\lvert v_*\rvert \leq m_1} \int_{s-\alpha}^s e^{-\nu_0(t-r)} \int_{\lvert u\rvert\leq  m_2} \omega (u) \lv h (r,\XX,u) \rv  
\leq D^2 m_1^3\, m_2^3\, \alpha\,  t\,  e^{-\nu_0t}\underset{s\in[0,t]}{\sup} \left[ e^{\nu_0 s}\lVert g_s\rVert_{L^{\infty}_{\omega}(\OO^\eps)}\right],
\eean
and
$$
I_2 :=\int_0^{t} \int_{\lvert v_*\rvert \leq m_1} \int_0^{s-\alpha} e^{-\nu_0(t-r)} \int_{\lvert u\rvert\leq  m_2} \omega (u) \lv h (r,\XX,u) \rv \, \d u\d r\d v_* \d s.
$$
To control $I_2$ we perform the change of variables $\XX \to x$, and we remark that
$$
\lvert\mathrm{det} \grad_{v_*}\XX(v_*)\rvert \geq C_0\lvert s-r\rvert^3> C_0\alpha^3>0,
$$ 
thus we compute
\begin{eqnarray*}
I_2 
&\leq& D m_1^{3/2}m_2^{3/2} \omega (m_2) \, t^2 e^{-\nu_0 t} \underset{0\leq r\leq s-\alpha<s\leq t}{\sup}\left[ e^{\nu_0 r} \left(\int_{\lvert v_*\rvert \leq m_1}\int_{\lvert u\rvert \leq m_2}\lvert  h(r,\XX,u)\rvert^2  \, \d u\d v_*\right)^{1/2}\right] \\
\\
&\leq& D m_1^{3/2}m_2^{3/2} \omega (m_2) \, t^2 e^{-\nu_0 t}   \underset{0\leq r\leq s-\alpha<s\leq t}{\sup}\left[ e^{\nu_0 r} \left(\int_{ x  \in\Omega^\eps}\int_{\lvert u\rvert \leq m_2}\lvert  h(r,x,u)\rvert^2   {1\over \lvert\mathrm{det} \grad_{v_*}X(v_*)\rvert }   \d u\d x\right)^{1/2}\right] \\
\\
&\leq&  Dm_1^{3/2}m_2^{3/2}\omega (m_2) \, t^2 C_0^{-1/2}\alpha^{-3/2} e^{-\nu_0 t}\underset{s\in[0,t]}{\sup} \left[ e^{\nu_0 s}\lVert  h_s\rVert_{ L^2(\OO^\eps)}\right].
\eean
where we have used the Cauchy-Schwartz inequality in the first line.

\medskip\noindent
\emph{Proof of \ref{item:K5}:} We fix $M, \eta >0$ to be defined later and we split 
 \begin{equation*}
    \{n_x \cdot u>0\} = \{\lvert u\rvert > M,\ n_x\cdot u>0 \}\cup \{\lvert u\rvert \leq M,\ \lvert n_x\cdot u\rvert<\eta\}\cup \{\lvert u\rvert \leq M,\ \lvert n_x\cdot u\rvert\geq\eta\}.
\end{equation*}
On the one hand, we compute 
\bean
\int_{\{\lvert u\rvert > M,\ n_x\cdot u>0\}} \lvert  h(t, x,u)\rvert \lvert n_{x}\cdot u\rvert du &\leq& \lVert  g_t\rVert_{L^{\infty}_{\omega}(\OO^\eps)} \int_{\{\lvert u\rvert > M,\ n_x\cdot u>0\}} \omega^{-1}(u) \lvert n_{x}\cdot u\rvert du\\
&\leq& o(M) \, \lVert h_t \rVert_{L^{\infty}_{\omega}(\OO^{\eps})} , 
\eean
where $o(M) \to 0$ as $M\to \infty$. Similarly, on the set $\{\lvert u\rvert \leq M,\ \lvert n(x_1)\cdot u\rvert<\eta\}$ we compute
\bean
 \int_{\{\lvert u\rvert \leq M,\ \lvert n_x\cdot u\rvert<\eta\}} \lvert  h(t,x,u)\rvert \lvert n_x \cdot u\rvert du &\leq& \eta \, \lVert  h_t\rVert_{L^{\infty}_{\omega}(\OO^\eps)}  \int_{\lvert u\rvert \leq M} \omega^{-1}(u)du \\
 &\leq& \eta \, M^3 D   \lVert  h_t\rVert_{L^{\infty}_{\omega}(\OO^\eps)}.
\eean
We take then $M>0$ large enough such that $o(M) \leq \alpha/2$
and we choose $\displaystyle \eta = {\alpha / (2M^3D) }$. This concludes the proof of \ref{item:K5} and the proof of Lemma~\ref{lem:Kproperties}. 
\end{proof}

Moreover, we also have the following lemma in the spirit of \ref{item:K4}. 

\begin{lem}\label{lem:Specular_Regularization}
Let $0\leq r\leq s\leq t$, $x\in\Omega^\eps$, $u,v_*\in\R^3$ and consider a function $\XX = \XX(t,s, r ,x,u, v_*)$ such that 
$$
\mathrm{det} \,\grad_{v_*}\XX (v_*) = C_0 \lvert s-r\rvert^3 + C_1 \eps,
$$ 
for some constants $C_0, C_1 >0$. 
Then for any $m_1, m_2>0$, and every $\alpha>0$ there is a constructive $\eps_U = \eps_U (\alpha, C_0, C_1):= \lv C_0\rv\alpha^3 /(2\lv C_1\rv)>0$ such that for every $\eps\in (0,\eps_U)$ there holds
\begin{multline*}
\int_0^{t} \int_{\lvert v_*\rvert \leq m_1} \int_0^s e^{-\nu_0(t-r)} \int_{\lvert u\rvert\leq  m_2} \omega (u) \lv  h  (r,\XX,u) \rv \, \d u \d r \d v_*\d s \leq D^2\, m_1^3\, m_2^3\, \alpha\,  t\,  e^{-\nu_0 t}\underset{s\in[0,t]}{\sup}\left[ e^{\nu_0s}\lVert  h_s\rVert_{L^{\infty}_{\omega}(\OO^{\eps})}\right]\\
    + 2^{1/2} Dm_1^{3/2}m_2^{3/2} \, \omega(m_2) \, t^2 \alpha^{-3/2} \lv C_0\rv ^{-1/2} t^2e^{-\nu_0 t}\underset{s\in[0,t]}{\sup}\left[ e^{\nu_0 s}\lVert  h_s\rVert_{ L^2(\OO^\eps)}\right].
\end{multline*}
where we recall that $D>0$ is the volume of the $3-$dimensional ball of radius 1. 
\end{lem}

\begin{rem}
We also notice that Lemma~\ref{lem:Specular_Regularization} is independent of either geometrical Assumption \ref{item:RH1} or \ref{item:RH2}.  
\end{rem}

\begin{proof}[Proof of Lemma~\ref{lem:Specular_Regularization}]
The proof follows the main ideas of that of \ref{item:K4}, thus we only sketch it. We split the integral on the LHS into $I_1 + I_2$ as defined during the proof of \ref{item:K4}. To control $I_2$ we remark that if $s-r \geq \alpha$ and $\eps\in (0, \eps_U)$, with $\eps_U$ as defined above, there holds
$$
\lv \mathrm{det} \,\grad_{v_*}\XX (v_*) \rv \geq \lv C_0\rv  \alpha ^3 - \lv C_1\rv  \eps\geq { \lv C_0\rv \over 2} \alpha ^3 + \left( { \lv C_0\rv \over 2} \alpha ^3 -\lv C_1\rv  \eps\right) > { \lv C_0\rv \over 2} \alpha ^3.
$$
The conclusion follows by repeating the arguments used on the proof of \ref{item:K4}. 
\end{proof}

\subsection{Regularizing effect of $K$}\label{ssec:RegularK} We will now use the previous lemmas to prove the regularizing character of the interplay between the free transport semigroup $S_\TTT$ and the operator $K$. 

\begin{prop} \label{prop:Kregularization}
Consider Assumption \ref{item:RH1} to hold, $\omega_1\in \WWWW_1$ a strongly confining admissible weight function, and let $ f$ be a solution of Equation~\eqref{eq:PLRBE_T}. For every $\lambda>0$ there is $\eps_2 = \eps_2(\lambda, T) $ such that for every $\eps\in (0, \eps_3)$ and every $\nu_2\in (0, \nu_0)$ we have that for every $(t, x, v)\in  \UU_T^\eps$ with $v\notin S_x$ there holds 
\begin{multline*}
\omega_1 (v) \lvert S_{\TTT} *_\sigma K  f (t, x,v) \rvert  \leq \lambda (t + t^2)  e^{-\nu_0t}\underset{s\in[0,t]}{\sup}\left[ e^{\nu_0s}\lVert  f_s\rVert_{L^{\infty}_{\omega_1}( \OO^\eps)}\right] 
+  Ct e^{-\nu_0 t} \lVert  f_0\rVert_{L^{\infty}_{\omega_1}(\OO^\eps)}  \\
+  Ct^2 e^{-\nu_0 t}\underset{s\in[0,t]}{\sup}\left[ e^{\nu_0 s}\lVert  f_s\rVert_{ \HH}\right] 
 + C (1+t) e^{-\nu_2 t}\underset{s\in [0,t]}{\sup}\left[e^{\nu_2 s}\lVert  G_s\rVert_{L^{\infty}_{\omega_1\nu^{-1}}(\OO^\eps)}\right]  ,
\end{multline*}
for a constant $C= C(\lambda)>0$, and any $\sigma\in [0,t]$ such that $x-v(t-\sigma )\in \bar\Omega^\eps$. Furthermore, there holds  $C(\lambda) \lesssim \lambda^{-p}$ for some constant $p>0$.
\end{prop}

\begin{proof}[Proof of Proposition~\ref{prop:Kregularization}] We fix $(t, x, v)\in  \UU_T^\eps$ with $v\notin S_x$, we set 
$$
\IIII:=  \omega_1 (v) \lvert S_{\TTT} *_\sigma K  f (t, x,v) \rvert  = \omega_1 (v) \left\lvert \int_{\sigma}^t S_{\TTT}(t-s) K  f (s, x,v) \d s \right\rvert,
$$
and we split the proof into four steps.

\medskip\noindent
\emph{Step 1. (A first control on $\IIII$)} By using \ref{item:K1} we deduce that for every $N_1>0$ there is $m_1(N_1)>0$ such that
\be \label{eq:STKN1}
\IIII \leq {1\over N_1} t e^{-\nu_0t}\underset{s\in [0, t]}{\sup}\left[e^{\nu_0 s}\lVert  f_s\rVert_{L^{\infty}_{\omega_1}(\OO^\eps)} \right]+ S_{\TTT}*_\sigma K_{m_1}\left(  \omega_1(v)  \lvert f(t, x,v)\rvert \right) ,
\ee
where we recall that $K_{m_1}$ has been defined in \ref{item:K1}.
Moreover, for each $s\in [\sigma, t]$ we denote $x_s:= y-v(t -s)$, and for any constant $\alpha_0>0$ we define $k_{0,m_1} = \Ind_{\{ \lvert n(x_s)\cdot v_*\rvert  \geq \alpha_0\}}k_{m_1}$ and the modified non-local operator 
\be\label{eq:DefKm1}
K_{0,m_1} h (v)= \int_{\R^3} k_{0,m_1} (v, v_*) h(v_*) \d v_*.
\ee
We then have that $S_{\TTT}*_\sigma K_{m_1} \left( \omega_1(v) \lvert f(t,x,v) \rvert\right)
= \IIII_0 + \widetilde \IIII $ where
 $$
\IIII_0:= 
 \int_\sigma^t e^{-\nu(v)(t-s)}\int_{\R^3} k_{0, m_1}(v, v_*)\omega_1(v_*) \lvert  f(s, x_s, v_*) \rvert  \d v_*\ds,
$$
and
\bean
\widetilde \IIII &:=& \int_0^t e^{-\nu (v) (t-s)} \int_{\lvert v_*\rvert \leq m_1, \lvert n_{x_s} \cdot v_*\rvert \leq \alpha_0} k_{m_1} (v,v_*)\omega_1 (v) \lvert  f (s, x_s ,v_*) \rvert \d v_* \ds \\
    &\leq &  C_km_1 \int_0^{t} e^{-\nu_0(t-s)}\lVert  f_s\rVert_{L^{\infty}_{\omega_1}(\OO^\eps)}\d s \int_{-\alpha_0}^{\alpha_0}\d v_{*}^{par}\int_{\lv v_*^{\perp}\rv \leq m_1}\d v_{*}^\perp\\
    &\leq&  2C_kDm_1^{4}\alpha_0 t e^{-\nu_0t} \underset{s\in[0,t]}{\sup}\left[ e^{\nu_0 s}\lVert  f_s\rVert_{L^{\infty}_{\omega_1}(\OO^\eps)}\right]
\end{eqnarray*}
where we have used the change of variables $v_{*}^{par} = [n(x_s)\cdot v_*]n(x_s)$ and $v_{*}^\perp = v_*-v_*^{par}$ to obtain the second line together with \eqref{eq:defkm}. The above computations together with \eqref{eq:STKN1} give that
$$
\IIII \leq \left( {1\over N_1}   + 2C_kDm_1^{4}\alpha_0  \right)t e^{-\nu_0t}\underset{s\in [0, t]}{\sup}\left[e^{\nu_0 s}\lVert  f_s\rVert_{L^{\infty}_{\omega_1}(\OO^\eps)} \right] + \IIII_0.
$$

We remark now that \cite[Lemma 17]{Guo09} gives that $\Sigma_0^\eps$ and $S_{x_s}$ have Lebesgue measure zero. Due to this fact, we may rewrite $\IIII_0$ being integrated over the set $v_* \in \R^3 \setminus S_{x_s}$. This implies that, since $v\notin S_x$ and $v_*\notin S_{x_s}$, we may define $(s_1, x_1, v_1)$, with $(x_1, v)\notin \Sigma_0^\eps$, as the point of the first bounce against the boundary through the backwards trajectory starting at $(s,x_s, v_*)$, and which is given by \eqref{eq:BackwardsTrajectory}. In particular, we remark that  $x_1 = x_s-v_*(s- s_1)$. 

Moreover, it is worth remarking that when $s_1\leq 0$ we reach the initial time $t=0$ along the backwards trajectory, on the contrary if $s_1 >0$ then along the backwards trajectory we reach the boundary at the time $s-s_1$. 
Using the Duhamel formula we deduce then that $\IIII_0 \leq  \IIII_1 + \IIII_2 + \IIII_3 +\IIII_4$,
where we have defined 
\bean
\IIII_1 &:=& \int_\sigma^t e^{-\nu(t-s)}\int_{\R^3} k_{0, m_1}  (v,v_*) \,  \omega_1 (v_*)\left\lvert   S_{\TTT}(s)    f_0 (x_s, v_*)    \right\rvert  \d v_*\d s,\\
\IIII_2 &:=& \int_\sigma^t e^{-\nu(t-s)}\int_{\R^3} k_{0, m_1} (v,v_*) \, \omega_1(v_*) \left\lvert    S_{\TTT}*_{\max (0, s_1)} K   f(s, x_s, v_*) \right\rvert  \d v_*\d s,\\
\IIII_3 &:=& \int_\sigma^t e^{-\nu(t-s)}\int_{\R^3} k_{0, m_1}(v,v_*) \, \omega_1(v_*)  \left\lvert    S_{\TTT}*_{\max(0,s_1)}  G(s, x_s, v_*) \right\rvert  \d v_*\d s,\\
\IIII_4 &:=&  \int_\sigma^t e^{-\nu(t-s)}\int_{\R^3} k_{0, m_1}(v,v_*)  \,  e^{-\nu(v_*)(s-s_1)} \, \omega_1(v_*) \left\lvert     f(s_1, x_1, v_*) \right\rvert  \d v_*\d s,
\eean
and we will control now each term separately. Using the very definition of $k_{m_1}$ given by \eqref{eq:defkm} together with \ref{item:K3} we deduce that for every $\nu_2\in (0, \nu_0)$ we have
\beqn
\IIII_1 + \IIII_3 \leq  C_k D m_1^{4} t e^{-\nu_0 t} \lVert   f_0\rVert_{L^{\infty}_{\omega_1}(\OO^\eps)} + C_k m_1^{4}D {\nu_1 t \over \nu_0 - \nu_2} e^{-\nu_2t}  \underset{s\in [0,t]}{\sup}\left[e^{\nu_2 s}\lVert   G_s\rVert_{L^{\infty}_{\omega_1\nu^{-1}}(\OO^\eps)}\right] .
\eeqn
Using now \ref{item:K1} we deduce that for any $N_2>0$ there is $m_2(N_2)>0$ such that
\beqn
\IIII_2  \leq   C_k m_1^{4} D{ t\over N_2} e^{-\nu_0 t}\underset{s\in [0, t]}{\sup}\left[e^{\nu_0 s}\lVert   f_s\rVert_{L^{\infty}_{\omega_1}(\OO^\eps)} \right]  + \IIII^K,
\eeqn
where we have defined
\beqn
\IIII^K :=  (C_k)^2m_1m_2 \int_0^{t} \underset{\lvert v_*\rvert \leq m_1} \int \int_0^s e^{-\nu_0(t-r)}  \underset{\lvert u\rvert\leq  m_2}\int     \omega_1(u) \, \lvert f (r,x_s-(s-r)u ,u) \rvert \, \d u\d r\d v_*\d s,
\eeqn
and since 
$$
\lvert\mathrm{det} \grad_{v_*} (x_s-(s-r)u)\rvert  = \lvert\mathrm{det} \grad_{v_*} (x-(t-s)v_*-(s-r)u)\rvert = \lvert  t - s\rvert^3
$$ 
we can apply \ref{item:K4} and it yields that for every $\alpha_1 >0 $ there holds
\begin{multline*}
\IIII^K\leq  C_k^2 D^2 m_1^{4}m_2^{4}\alpha_1 t e^{-\nu_0t}\underset{s\in[0,t]}{\sup}\left[ e^{\nu_0s}\lVert   f_s\rVert_{L^{\infty}_{\omega_1}(\OO^\eps)}\right]\\
 + C_k^2Dm_1^{1+3/2}m_2^{1+3/2}\, \omega_1(m_2) \alpha_1^{-3/2} t^2e^{-\nu_0 t}\underset{s\in[0,t]}{\sup}\left[ e^{\nu_0 s}\lVert   f_s\rVert_{{\HH}}\right].
\end{multline*}
Altogether we have obtained that
\bean
\IIII&\leq& \left( {1\over N_1}  + 2C_kDm_1^{4}\alpha_0+  {C_kD m_1^{4}\over N_2}+ C_k^2D^2m_1^{4}m_2^{4} \alpha_1 \right)   t e^{-\nu_0t}\underset{s\in [0, t]}{\sup}\left[e^{\nu_0 s}\lVert   f_s\rVert_{L^{\infty}_{\omega_1}(\OO^\eps)} \right] \\
&&+ C_k D m_1^{4} t e^{-\nu t} \lVert   f_0\rVert_{L^{\infty}_{\omega_1}(\OO^\eps)} 
 + C_kDm_1^{4}{\nu_1t \over \nu_0 - \nu_2}  e^{-\nu_2 t}\underset{s\in [0,t]}{\sup}\left[e^{\nu_2 s}\lVert   G_s\rVert_{L^{\infty}_{\omega_1\nu^{-1}}(\OO^\eps)}\right] \\
&&+C_k^2Dm_1^{1+3/2}m_2^{1+3/2} \omega_1(m_2) \alpha_1^{-3/2} t^2 e^{-\nu_0 t}\underset{s\in[0,t]}{\sup}\left[ e^{\nu_0 s}\lVert   f_s\rVert_{{\HH}}\right]   +\IIII_4,
\eean
where there is only left to control the boundary term $\IIII_4$. We remark that
\beqn
\IIII_4  \leq \int_\sigma^t  \underset{\vert v_*\rvert \leq m_1, \, \lvert n(x_s)\cdot v_* \rvert\geq \alpha_0}\int  e^{-\nu_0 (t-s_1)}  k_{0,m_1} (v,v_*) \omega_1(v_*) \lvert   f(s_1, x_1, v_*)\rvert \, \d v_*\d s .
\eeqn
We recall that $ x_s - (s-s_1)v_* = x_1$, and since by definition $(x_1, v_*)\in \Sigma_-^\eps $, the Maxwell boundary condition applies thus $f(s_1, x_1, v_*) = (1-\iota^\eps(x_1)) f(s_1, x_1, \VV_{x_1}v_*) + \iota^\eps(x_1) \MMM(v_*) \widetilde f(s_1, x_1)$.
We then have that  $\IIII_4 \leq  \IIII^S + \IIII^D$,
where 
\bear
\IIII^S &:=& \int_\sigma^t \int_{\R^3} k_{0, m_1} (v,v_*)  e^{-\nu_0(t-s_1)} \omega_1(v_*) (1-\iota^\eps(x_1)) \left\lvert   f(s_1, x_1, \VV_{x_1}v_*) \right\rvert  \d v_*\d s , \label{eq:IIII_4_S}\\
\IIII^D &:=& \int_\sigma^t \int_{\R^3} k_{0, m_1} (v, v_*)  e^{-\nu_0(t-s_1)}  \omega_1(v_*)   \iota^\eps(x_1) \MMM(v_*) \left\lvert \widetilde f(s_1, x_1) \right\rvert  \d v_*\d s .\label{eq:IIII_4_D}
\eear

\medskip\noindent
\emph{Step 2. (Control of $\IIII^S$)} For any small parameter $\alpha_{0,1}>0$ we define the non-local operator
$$
K_{1,m_1} h (v)= \int_{\R^3} k_{1,m_1} (v, v_*) h(v_*) \d v_*,
$$
where $k_{1,m_1} (v, v_*)= \Ind_{\{ \lvert n(x_1)\cdot v_*\rvert  \geq \alpha_{0,1}\}}k_{0, m_1} (v,v_*)$, and we also define
$$
\IIII^S_0 =  \int_\sigma^t \int_{\R^3} k_{1, m_1} (v, v_*)  e^{-\nu_0(t-s_1)} \omega_1(v_*) \left\lvert   f(s_1, x_1, \VV_{x_1}v_*) \right\rvert  \d v_*\d s.
$$
Proceeding then as during the Step 1 of this proof we have that
\beqn
\IIII^S \leq \IIII^S_0 + 2C_kDm_1^{4}\alpha_{0,1} t e^{-\nu_0t} \underset{s\in[0,t]}{\sup}\left[ e^{\nu_0 s}\lVert  f_s\rVert_{L^{\infty}_{\omega_1}(\OO^\eps)}\right], 
\eeqn
We remark now that  $\lvert \VV_{x_1} v_*\rvert  = \lvert v_*\rvert$, $\lvert n(x_1) \cdot \VV_{x_1} v_*\rvert = \lvert n(x_1)\cdot v_*\rvert$, 
and we recall that \cite[Lemma 17]{Guo09} gives that the singular set $S_{x_1}$ has Lebesgue measure zero. Therefore, we deduce that we can integrate the expression of $\IIII^S_0$ on the set $\{\lvert  \VV_{x_1} v_*\rvert \leq m_1, \lvert n(x_1)\cdot  \VV_{x_1} v_*\rvert \geq \alpha_{0,1} \}  \setminus S_{x_1}$.
We can then apply Lemma~\ref{lemma3.2:GeneralRefl} and that there is $\eps_2^1 = \eps_R (\alpha_0, M, T)$ such that for every $\eps \in (0, \eps_2^1)$ there is only one bounce against the backwards trajectory, hence applying the Duhamel formula as before we have that $\IIII^S_0 \leq \IIII^S_1 + \IIII^S_2+ \IIII^S_3$,
where
\bean
\IIII_1^S &:=& \int_\sigma^t \int_{\R^3} k_{1,m_1} (v, v_*) e^{-\nu_0(t-s_1)}  \omega_1(v_*)\lvert  S_{\TTT}(s_1)  f_0(x_1, \VV_{x_1}v_*)\rvert \d v_*\ds,\\
\IIII_2^S &:=& \int_\sigma^t  \int_{\R^3} k_{1,m_1} (v, v_*) e^{-\nu_0(t-s_1)} \omega_1 (v_*)\lvert  S_{\TTT}* K  f (s_1, x_1, \VV_{x_1}v_*) \rvert \d v_*\ds,\\
\IIII_3^S &:=& \int_\sigma^t  \int_{\R^3} k_{1,m_1} (v,v_*) e^{-\nu_0(t-s_1)} \omega_1(v_*)\lvert  S_{\TTT}*   G(s_1, x_1, \VV_{x_1}v_*) \rvert \d v_*\ds,
\eean
and we proceed to control each of these terms separately. First we remark that  $\omega_1(\VV_{x_1} v) = \omega_1(v)$, and using \ref{item:K3}, we deduce that for every $\nu_2\in (0, \nu_0)$ there holds
$$
\IIII_1^S + \IIII_3^S \leq C_k D m_1^{4} t e^{-\nu_0 t} \lVert  f_0\rVert_{L^{\infty}_{\omega_1}(\OO^\eps)} +  C_k D m_1^{4}{\nu_1t \over \nu_0 - \nu_2}  e^{-\nu_2t} \underset{s\in [0,t]}{\sup}\left[e^{\nu_2 s}\lVert   G_s\rVert_{L^{\infty}_{\omega_1\nu^{-1}}(\OO^\eps)}\right] .
$$
Furthermore, to control $\IIII_2^S$ we use \ref{item:K1} so that for every $N_3>0$ there is $m_3 (N_3)>0$ satisfying
$$
\IIII_2^S \leq C_k D m_1^{4}{ t \over N_3} e^{-\nu_0 t}\underset{s\in [0, t]}{\sup}\left[e^{\nu_0 s}\lVert   f_s\rVert_{L^{\infty}_{\omega_1}(\OO^\eps)} \right] +\IIII_{2,0}^S,
$$
where we have defined
$$
\IIII_{2,0}^S =  \int_\sigma^t \int_{\R^3} k_{0,m_1}(v, v_*)  e^{-\nu_0(t-s_1)} \, S_{\TTT}* K_{m_3} \left( \omega_1(v_*) \lvert  f (s_1, x_1, \VV_{x_1}v_*) \rvert\right) \d v_*\ds.
$$
To control then $\IIII_{2,0}^S$ we first use the expression of $k_{m_1}$ given by \eqref{eq:defkm} and we have that 
$$
\IIII_{2,0}^S \leq C_k^2 m_1m_3 \int_0^t \underset{\lv v_*\rv\leq m_1}\int  \int_0^{s_1} e^{-\nu_0(t-r)}  \underset {\lv u_*\rv \leq m_3}\int  \omega_1(u_*)\lvert  f (x_1 - (\VV_{x_1}v_*)(s_1 - r)), u_*) \rvert \d u_*\d r\d v_*\ds. 
$$
We then have the following lemma, by arguing as \cite[Equations (2.141) and (2.142)]{GuoZhou18}.
\begin{lem}\label{lem:Specular_Jacobian}
Assume $\lv v_*\rv \leq m_1$ for some $m_1>0$, and we define $\XX(v_*):= x_1 - \VV_{x_1}v_*(s_1- r)$. There holds
\beqn
 \det \grad_{v_*}(\XX(v_*)) =  (s-r)^3 + \OOO(\eps).
\eeqn
\end{lem}

We then deduce that Lemma~\ref{lem:Specular_Jacobian} together with Lemma~\ref{lem:Specular_Regularization}, imply that for every $\alpha_2>0$ there is $\eps_2^2 = \eps_U(\alpha_2)>0$ such that for every $\eps\in(0, \min(\eps_2^1, \eps_2^2))$ there holds
\begin{multline*}
\IIII_{2,0}^S \leq C_k^2\, D^2\, m_1^{4}\, m_3^{4}\, \alpha_2 \, t \, e^{-\nu_0 t}\underset{s\in[0,t]}{\sup}\left[ e^{\nu_0s}\lVert  f_s\rVert_{L^{\infty}_{\omega_1}(\OO^\eps)}\right] \\
+  2^{1/2}\omega_1(m_3) \, C_k^2\, Dm_1^{1+3/2}\, m_3^{1+3/2}\, \alpha_2^{-3/2}\,  t^2\, e^{-\nu_0 t}\underset{s\in[0,t]}{\sup}\left[ e^{\nu_0 s}\lVert  f_s\rVert_{{\HH}}\right].
\end{multline*}
We conclude this step by putting together the estimates to control the specular reflection and we obtain that 
\bean
\IIII^S &\leq& \left( 2C_kDm_1^{4}\alpha_{0,1} + C_k Dm_1^{4} { t \over N_3} + C_k^2D^2m_1^{4}m_3^{4}\alpha_2 t \right) e^{-\nu_0 t}\underset{s\in[0, t]}{\sup}\left[ e^{\nu_0s}\lVert  f_s\rVert_{L^{\infty}_{\omega_1}(\OO^\eps)}\right] \\
&&+C_k D m_1^{4} t e^{-\nu_0 t} \lVert  f_0\rVert_{L^{\infty}_{\omega_1}(\OO^\eps)} + C_k D m_1^{4}{\nu_1t \over \nu_0 - \nu_2}   e^{-\nu_2 t} \underset{s\in [0,t]}{\sup}\left[e^{\nu_2 s}\lVert  G_s\rVert_{L^{\infty}_{\omega_1\nu^{-1}}(\OO^\eps)}\right]  \\
&& + 2^{1/2}\omega_1(m_3) C_k^2 m_1^{1+3/2}m_3^{1+3/2} \alpha_2^{-3/2} t^2e^{-\nu_0 t}\underset{s\in[0,t]}{\sup}\left[ e^{\nu_0 s}\lVert  f_s\rVert_{{\HH}}\right] ,
\eean
for every $\nu_2\in (0, \nu_0)$.

\medskip\noindent
\emph{Step 3. (Control of $\IIII^D$)}
 From the very definition of $\IIII^D$ we first observe the following elementary inequality
$$
\IIII^D \leq \int_\sigma^t  \int_{\R^3} k_{0,m_1}(v, v_*) e^{-\nu_0(t-s_1)} \int_{\R^3} \left\lvert  f(s_1, x_1, u_*)\right\rvert (n(x_1)\cdot u_*)_+\,  \d u_* \d v_*\ds,
$$
where we have used the fact that $ \MMM\omega_1\leq 1$. We then apply \ref{item:K5} so that for every $\lambda_1>0$ there are $M_1, \eta>0$ such that
\beqn
\IIII^D \leq \lambda_1 C_k D m_1^{4} t e^{-\nu_0 t}\underset{s\in[0,t]}{\sup}\left[ e^{\nu_0s}\lVert  f_s\rVert_{L^{\infty}_{\omega_1}(\OO^\eps)}\right]  + \IIII^D_0 
\eeqn
where we have defined
$$
\IIII^D_0:=  \int_\sigma^t  \int_{\R^3}k_{m_1}(v, v_*) e^{-\nu_0(t-s_1)} \int_{\lvert u_*\rvert \leq M_1,\, \lvert n_{x_1}\cdot u\rvert >\eta} \left\lvert  f(s_1, x_1, u_*)\right\rvert (n_{x_1}\cdot u_*)_+ \, \d u_* \d v_*\ds.
$$
Arguing as during the first Duhamel decomposition of this proof we deduce that we may rewrite the previous integral as
being integrated over the set $ \{\lvert u_*\rvert \leq M_1, \lvert n_{x_1}\cdot u_*\rvert \geq \eta \}\setminus S_{x_1}$,
thus we may apply Lemma~\ref{lemma3.2:GeneralRefl} and we deduce that there is $\eps_2^3= \eps_R(\eta, M_1, T)$ such that for every $\eps \in (0, \min (\eps_2^1, \eps_2^2, \eps_2^3))$ there is only one bounce against the backwards trajectory. Hence applying Duhamel's principle as in the Step 2 we get that $\IIII^D_0 \leq \IIII^D_1 + \IIII^D_2 + \IIII^D_3$, 
with
\bean
\IIII_1^D &:=& \int_\sigma^t  \int_{\R^3} k_{m_1}(v, v_*) e^{-\nu_0(t-s_1)} \underset{\lvert u_*\rvert \leq M_1,\, \lvert n_{x_1}\cdot u\rvert >\eta}\int \left\lvert S_{\TTT}(s_1)  f_0(x_1, u_*)\right\rvert (n_{x_1}\cdot u_*)_+ \d u_* \d v_*\ds, \\
\IIII_2^D &:=& \int_\sigma^t \int_{\R^3} k_{m_1}(v, v_*)  e^{-\nu_0(t-s_1)} \underset{\lvert u_*\rvert \leq M_1,\, \lvert n_{x_1}\cdot u\rvert >\eta}\int \left\lvert S_{\TTT}* K  f (s_1, x_1, u_*) \right\rvert (n_{x_1}\cdot u_*)_+ \d u_* \d v_*\ds, \\
\IIII_3^D &:=& \int_\sigma^t \int_{\R^3} k_{m_1}(v, v_*)  e^{-\nu_0(t-s_1)} \underset{\lvert u_*\rvert \leq M_1,\, \lvert n_{x_1}\cdot u\rvert >\eta}\int \left\lvert S_{\TTT}*   G(s_1, x_1, u_*) \right\rvert (n_{x_1}\cdot u_*)_+ \d u_* \d v_*\ds,
\eean
 and we will control each of these terms separately. On the one hand, using again the form of $k_{m_1}$ given by \eqref{eq:defkm}, and  \ref{item:K3}, there holds
$$
\IIII_1^D + \IIII_3^D \leq C_k D^2 m_1^{4} M_1^{4} t e^{-\nu_0 t} \lVert  f_0\rVert_{L^{\infty}_{\omega_1}(\OO^\eps)} + {\nu_1 C_k D^2m_1^{4} M_1^{4} t \over \nu_0 - \nu_2} e^{-\nu_2 t}  \underset{s\in [0,t]}{\sup}\left[e^{\nu_2 s}\lVert   G_s\rVert_{L^{\infty}_{\omega_1\nu^{-1}}(\OO^\eps)}\right] .
$$
To analyze $\IIII_2^D$ we use \ref{item:K1} so that for every $N_4>0$ there is $m_4(N_4)>0$ for which there holds
  \bean
\IIII_2^D
&\leq &  {t^2 C_k^2 D^2 m_1^{4} M_1^{4} \over N_4} e^{-\nu_0 t}\underset{s\in[0, t]}{\sup}\left[ e^{\nu_0s}\lVert   f_s\rVert_{L^{\infty}_{\omega_1}(\OO^\eps)}\right]  + C_k^2m_1m_4M_1\,   \IIII^D_{2,0},
\eean
where 
we have defined
$$
\IIII_{2,0}^D = \int_\sigma^t  \int_{\lvert v_*\rvert \leq m_1}  \int_{\lvert u_*\rvert \leq M_1} \int_0^{s_1} e^{-\nu_0(t-r)} \int_{\lvert u'\rvert \leq m_4} \omega (u')\lvert  f (r, x_1-u_*(s_1 - r), u') \rvert \d u'\dr \d u_* \d v_*\ds.
$$
Furthermore, we remark that
\beqn
\lvert \det \grad_{u_*}(x_1 - u_*(s_1- r))\rvert  = \vert s_1-r\rvert^3,
\eeqn
thus using \ref{item:K4} we have that for every $\alpha_4>0$ there holds
\begin{multline*}
\IIII_{2,0}^D \leq 
 D^3\, m_1^3\, M_1^3\, m_4^3\, \alpha_4 \, t\,  e^{-\nu_0 t}\underset{s\in[0, t]}{\sup}\left[ e^{\nu_0s}\lVert  f_s\rVert_{L^{\infty}_{\omega_1}(\OO^\eps)}\right] \\
+ D^2m_1^3M_1^{3/2}m_4^{3/2} \, \omega_1(m_4) \, \alpha_4^{-3/2} \, t^2\, e^{-\nu_0 t}\underset{s\in[0,t]}{\sup}\left[ e^{\nu_0 s}\lVert  f_s\rVert_{{\HH}}\right].
\end{multline*}
Putting together the above estimates controlling the diffusion term $\IIII^D$ we have that 
\bean
\IIII^D  &\leq & \left(  \lambda_1 C_k D m_1^{4} t +   {t^2 C_k^2 D^2 m_1^{4} M_1^{4} \over N_4}  + C_k^2D^3m_1^{4} M_1^{4}m_4^{4}\alpha_4 t  \right)   e^{-\nu_0 t}\underset{s\in[0,t]}{\sup}\left[ e^{\nu_0s}\lVert   f_s\rVert_{L^{\infty}_{\omega_1}(\OO^\eps)}\right] \\
&&+ C_k D^2 m_1^{4} M_1^{4} t e^{-\nu_0 t} \lVert   f_0\rVert_{L^{\infty}_{\omega_1}(\OO^\eps)} +   {\nu_1 C_k D^2m_1^{4} M_1^{4} t \over \nu_0 - \nu_2}  e^{-\nu_2 t}  \underset{s\in [0,t]}{\sup}\left[e^{\nu_2 s}\lVert   G_s\rVert_{L^{\infty}_{\omega_1\nu^{-1}}(\OO^\eps)}\right] \\
 && + C_k^2D^2m_1^{4}m_4^{1+3/2} M_1^{1+3/2}\omega(m_4) \alpha_4^{-3/2} t^2e^{-\nu_0 t}\underset{s\in[0,t]}{\sup}\left[ e^{\nu_0 s}\lVert  f_s\rVert_{{\HH}}\right]  ,
  \eean
  for every $\nu_2\in (0, \nu_0)$.

\medskip\noindent
\emph{Step 4. (Choice of the parameters)} Putting together the estimates from Steps 1, 2 and 3 we get that for every $\nu_2\in (0, \nu_0)$ there holds
\begin{multline*}
\IIII\leq  C_1\left(t+t^2\right) e^{-\nu_0 t}\underset{s\in[0,t]}{\sup}\left[ e^{\nu_0s}\lVert  f_s\rVert_{L^{\infty}_{\omega_1}(\OO^\eps)}\right] + C_2 t e^{-\nu_0t} \lVert  f_0\rVert_{L^{\infty}_{\omega_1}(\OO^\eps)} \\
 + C_3t^2e^{-\nu_0 t}\underset{s\in[0,t]}{\sup}\left[ e^{\nu_0 s}\lVert  f_s\rVert_{{\HH}}\right] + C_4 t e^{-\nu_2 t} \underset{s\in [0,t]}{\sup}\left[e^{\nu_2 s}\lVert   G_s\rVert_{L^{\infty}_{\omega_1\nu^{-1}}(\OO^\eps)}\right],
\end{multline*}
with
\bean
 C_1 &=& {1\over N_1}  + 2C_kDm_1^{4} (\alpha_0 + \alpha_{0,1}) +  {C_kD m_1^{4}\over N_2}+ C_k^2D^2m_1^{4}m_2^{4} \alpha_1  +  C_k Dm_1^{4} { 1 \over N_3} \\
&&+ C_k^2D^2m_1^{4}m_3^{4}\alpha_2  +  \lambda_1 C_k D m_1^{4}  +   {C_k^2 D^2 m_1^{4} M_1^{4} \over N_4}  + C_k^2D^3m_1^{4} M_1^{4}m_4^{4}\alpha_4 ,  \\
C_2 &=&C_k D m_1^{4}  + C_k D m_1^{4} + C_k D^2 m_1^{4} M_1^{4} ,\\
C_3 &=&  9C_k^2Dm_1^{1+3/2}m_2^{1+3/2} \omega_1(m_2) \alpha_1^{-3/2} +   2^{3/2}\omega_1(m_3) C_k^2 m_1^{1+3/2}m_3^{1+3/2} \alpha_2^{-3/2} \\
&&+   C_k^2D^2m_1^{4}m_4^{1+3/2} M_1^{1+3/2}\omega_1(m_4) \alpha_4^{-3/2}  , \\
C_4 &=&  {\nu_1 \over \nu_0 - \nu_2}\left( C_kDm_1^{4} +  C_k D m_1^{4} + C_k D^2m_1^{4} M_1^{4}  \right).
\eean

We then set 
\begin{multicols}{2}
\begin{itemize}
\item $N_1=  9\lambda^{-1}$ and this fixes $m_1$,
\item $\alpha_0 + \alpha_{0,1} = \lambda ( 18C_k D m_1^{4})^{-1}$,
\item $\displaystyle N_2 = 9 C_k Dm_1^{4} \lambda^{-1}$ fixing $m_2$, 
\item $\alpha_1= \lambda (9C_k^2 D^2 m_1^{4}m_2^{4})^{-1}$,
\item $N_3 = 9C_kDm_1^{4} \lambda^{-1}$ and this fixes $m_3$
\item $\alpha_2= \lambda (9 C_k^2 D^2m_1^{4}m_3^{4})^{-1}$,
\item $\lambda _1 = \lambda (9C_kDm_1^{4})^{-1}$ fixing $M_1$,
\item $N_4 = 9C_k^2 D^2 m_1^{4} M_1^{4} \lambda^{-1}$ and this fixes $m_4$,
\item  $\alpha_4= \lambda (9C_k^2D^3m_1^{4}m_4^{4}M_1^{4})^{-1}$.
\end{itemize}
\end{multicols}
This implies that $C_1\leq \lambda$, we then define $C = \max(C_2, C_3, C_4)$ and we observe that, since the constants that define $C$ come from \ref{item:K1}, \ref{item:K4} and \ref{item:K5}, there is a constant $p>0$ such that $ C\lesssim \lambda^{-p}$, and we conclude by setting $\eps_2= \min(\eps_2^1, \eps_2^2, \eps_2^3)$. 
\end{proof}

\subsection{Pointwise estimate on the trajectories}\label{ssec:TrajectoryLinfty} We now use the regularization estimate given by Proposition~\ref{prop:Kregularization} to obtain a pointwise control on the solutions of Equation~\eqref{eq:PLRBE_T}. 
\begin{prop} \label{prop:GainLinftyL2LB1}
Consider Assumption \ref{item:RH1} to hold, $\omega_1\in \WWWW_1$ a strongly confining weight function, and let $f$ be a solution of Equation~\eqref{eq:PLRBE_T}. For every $\lambda>0$ there is a constructive $ \eps_3 = \eps_3(\lambda, T)\geq 0 $ such that for every $\eps\in (0,\eps_3)$, every $t\in [0, T]$, and every $\nu_2\in (0, \nu_0)$, we have that for every $(x, v)\in  \OO^\eps$, $v\notin S_{x}(v)$ there holds
\begin{multline} \label{eq:GainLinftyL2LB1}
\omega_1(v)\lvert f  (t, x,v) \rvert \leq \left[1-\iota_0 + \lambda (1+ t + t^2)\right] e^{-\nu_0 t}\underset{s\in[0,t]}{\sup}\left[ e^{\nu_0s}\lVert  f_s\rVert_{L^{\infty}_{\omega_1}( \bar\OO^\eps)}\right] 
+ C(1+t) e^{-\nu_0 t} \lVert f_0\rVert_{L^{\infty}_{\omega_1}(\OO^\eps)}   \\
+ C t^2 e^{-\nu_0 t}\underset{s\in[0,t]}{\sup}\left[ e^{\nu_0 s}\lVert  f_s\rVert_{\HH}\right]  
+ C(1+t) e^{-\nu_2t}\underset{s\in [0,t]}{\sup}\left[e^{\nu_2 s}\lVert G_s\rVert_{L^{\infty}_{\omega_1\nu^{-1}}(\OO^\eps)}\right] ,
\end{multline}
for some constant $C = C(\lambda)>0$. Moreover, there is $p>0$ such that $C \lesssim \lambda^{-p}$.
\end{prop} 

\begin{proof}
We split the proof into four steps.

\medskip\noindent
\emph{Step 1. (Duhamel decomposition)} We denote the point $(t_1, x_1, v_1)$ as the first collision through the backwards trajectory as defined in \eqref{eq:BackwardsTrajectory}, and we remark that we have $x_1 = x-v(t- t_1)$. We observe that if $t_1\leq 0$, then along the backwards trajectory we reach the initial time $t=0$, on the contrary if $t_1 >0$ then along the backward trajectory we reach the boundary at the time $t-t_1$. Using Duhamel's formula we get that
\bean
    \lvert \omega_1(v) f  (t, x,v) \rvert &\leq& e^{-\nu(v) t}  \omega_1(v) \lv f_0 (x,v)\rv   + \omega_1(v)\lvert S_{\TTT}*_{\sigma}  K f (t,x,v)\rvert  + \omega_1(v) \lvert S_{\TTT}*_{\sigma} G (t,x,v)\rvert \\
    && + e^{-\nu(v)(t-t_1)} \omega_1(v)\lvert  f(t_1, x_1, v)\rvert\\
    &=:& \III_1 + \III_2 + \III_3 + \III_4 ,
\eean
where we have defined $\sigma = \max(0,t_1)$ and we will control each of these terms separately. 
Using then \ref{item:K3} and Proposition~\ref{prop:Kregularization} we have that for every $\lambda_1>0$ there is $\eps_3^1:= \eps_2 (\lambda_1, T) >0$ such that for every $\eps \in (0, \eps_3^1)$ and every $\nu_2\in (0, \nu_0)$ there is a constant $C_1=C_1( \lambda_1)>0$ such that 
\bean
\III_1 + \III_2 + \III_3 &\leq& \lambda_1 (t+t^2)  e^{-\nu_0t}\underset{s\in[0,t]}{\sup}\left[ e^{\nu_0s}\lVert f_s\rVert_{L^{\infty}_{\omega_1}(\OO^\eps)}\right] +  (1+ C_1t) e^{-\nu_0t} \lVert f_0\rVert_{L^{\infty}_{\omega_1}(\OO^\eps)} \\
&& +  C_1 t^2 e^{-\nu_0 t}\underset{s\in[0,t]}{\sup}\left[ e^{\nu_0 s}\lVert f_s\rVert_{\HH}\right] + (1+C_1)(1+t) e^{-\nu_2 t} \underset{s\in [0,t]}{\sup}\left[e^{\nu_2 s}\lVert G_s\rVert_{L^{\infty}_{\omega_1\nu^{-1}}(\OO^\eps)}\right],
\eean
furthermore there are $c_1, p_1>0$ such that $C_1\leq  c_1 \lambda_1^{-p_1}$.

To control now the boundary term $\III_4$ we have that $(x_1, v)\in \Sigma^\eps_-$, thus the Maxwell boundary condition gives that $f(t_1, x_1, v) = (1-\iota^\eps(x_1)) f(t_1, x_1, \VV_{x_1}v) + \iota^\eps(x_1) \MMM(v)\widetilde f(t_1, x_1)$,
thus
\beqn
\III_4 \leq (1-\iota_0) \underset{s\in[0,t]}{\sup}\left[ e^{\nu_0s}\lVert f_s\rVert_{L^{\infty}_{\omega_1}(\bar \OO^\eps)}\right] + \III^D ,
\eeqn
where we have used that $\omega_1\, \MMM\leq 1$ and we have defined $\III^D:= e^{-\nu(v) (t-t_1)} \left\lvert \widetilde f(t_1, x_1)\right\rvert$.

\medskip\noindent
\emph{Step 2. (Control of the diffusive reflection)}  To control the diffusive boundary condition we first use \ref{item:K5} so that for every $\lambda_2>0$ there are $M, \eta>0$ such that 
\beqn
\III^D \leq \lambda_2  e^{-\nu_0 t}\underset{s\in[0,t]}{\sup}\left[ e^{\nu_0 s}\lVert f_s\rVert_{L^\infty_{\omega_1}(\OO^\eps)}\right] + \III_0^D \label{eq:SB1_diffusive},
\eeqn
where we have used that $\omega_1\geq 1$ and we have defined
$$
\III_0^D:= e^{-\nu(v)(t-t_1)} \int_{\lvert u\rvert \leq M,\, \lvert n(x_1)\cdot u\rvert >\eta}  \omega_1 (u)\lvert  f(t_1, x_1, u)\rvert  (n(x_1)\cdot u) \d u.
$$
Arguing as during the Step 3 of the proof of Proposition~\ref{prop:Kregularization} there is $\eps_3^2:= \eps_R (\eta, M, T)$, given by Lemma~\ref{lemma3.2:GeneralRefl},  such that for every $\eps \in (0, \min(\eps_3^1, \eps_3^2))$ there holds that for every $\lambda_3>0$ there is $\eps_3^3:= \eps_3 (\lambda_3, T)$, given by Propsotion~\ref{prop:Kregularization}, such that for every $\eps\in (0, \min(\eps_3^1, \eps_3^2, \eps_3^3))$ there holds that for every $\nu_2\in (0, \nu_0)$ there is a constant $C_3(\lambda_3)>0$ such that
\begin{multline*}
\III^D \leq \left( \lambda_2 + D  M^{4}\lambda_3 (t +t^2) \right)  e^{-\nu_0 t}\underset{s\in(0,t]}{\sup}\left[ e^{\nu_0 s}\lVert f_s\rVert_{L^\infty_{\omega_1}(\OO^\eps)}\right] + \left( D M^{4} + D  M^{4}C_3t^2 \right) e^{-\nu_0t} \lVert f_0\rVert_{L^{\infty}_{\omega_1}(\OO^\eps)} \\
+ \left( D  M^{4} +  D  M^{4}C_3 (1+t)\right){\nu_1\over \nu_1 - \nu_0} e^{-\nu_2t} \underset{s\in [0,t]}{\sup}\left[e^{\nu_2 s}\lVert G_s\rVert_{L^{\infty}_{\omega_1\nu^{-1}}(\OO^\eps)}\right] +  D  M^{4}C_3t^2 e^{-\nu_0 t}\underset{s\in[0,t]}{\sup}\left[ e^{\nu_0 s}\lVert f_s\rVert_{ \HH}\right],
\end{multline*}
furthermore there are $c_3, p>0$ such that $C\leq c_3 \lambda_3^{-p}$.

\medskip\noindent
\emph{Step 3. (Choice of the parameters)} 
Altogether we have obtained that
\bean
  \omega_1(v) \lvert f (t, x ,v) \rvert &\leq& (1-\iota_0) e^{-\nu_0t}\underset{s\in[0,t]}{\sup}\left[ e^{\nu_0s}\lVert f_s\rVert_{L^{\infty}_{\omega_1}(\bar \OO^\eps)}\right]  + \IIII_1 (1+t)e^{\nu_0t} \lVert f_0\rVert_{L^{\infty}_{\omega_1}(\OO^\eps)}   \\
 &&+ \left( \lambda_1 + \lambda_2 + D  M^{4} \lambda_3   \right) (1+ t+t^2)  e^{-\nu_0t}\underset{s\in[0,t]}{\sup}\left[ e^{\nu_0s}\lVert f_s\rVert_{L^{\infty}_{\omega_1}(\bar \OO^\eps)}\right]\\
 && + \IIII_2 t^2 e^{-\nu_0 t}\underset{s\in[0,t]}{\sup}\left[ e^{\nu_0 s}\lVert f_s\rVert_{ \HH}\right] + \IIII_3 (1+t)e^{-\nu_2 t} \underset{s\in [0,t]}{\sup}\left[e^{\nu_2 s}\lVert G_s\rVert_{L^{\infty}_{\omega_1\nu^{-1}}(\OO^\eps)}\right] ,
\eean
with the constants
$$
\IIII_1 = 1+ C_1  +  D M^{4} + D  M^{4}C_3 , \quad \IIII_2 =  C_1 + D  M^{4}C_3, \quad \text{ and } \quad  \IIII_3 =  \left( 1+    D  M^{4}\right) {\nu_1\over \nu_0 - \nu_2} + C_1 + D  M^{4}C_3.
$$
We choose then $\lambda_1=\lambda_2 = \lambda/3$, $\lambda_3 = \lambda/(3DM^{4})$, we take $\eps_3 = \min(\eps_3^1, \eps_3^2, \eps_3^3) $, we define $C(\lambda) = \max\left(\IIII_1, \IIII_2, \IIII_3 \right) $ thus $C(\lambda) \lesssim \lambda^{-p}$, for some $p>0$, and this concludes the proof.
\end{proof}

\subsection{Weighted $L^\infty$ control for solutions of Equation~\eqref{eq:PLRBE_T}}
In this subsection we use the estimate given by Proposition~\ref{prop:GainLinftyL2LB1} to deduce a weighted $L^\infty$ control on the solutions of Equation~\eqref{eq:PLRBE_T}.
\begin{prop}\label{prop:LinftyEstimatePerturbedFiniteT_0} 
Consider Assumption \ref{item:RH1} to hold, $\omega_1\in\WWWW_1$ a strongly confining admissible weight function, and let $f$ be a solution of Equation~\eqref{eq:PLRBE_T}. There is $\eps_4 = \eps_4 (T)>0$ such that for every $\eps\in (0,\eps_4)$ there holds 
\begin{multline*}
\lVert  f_t  \rVert_{L^{\infty}_{\omega_1}(\bar\OO^\eps)} \leq   C(1+T) (1+T+T^2)^pe^{-\nu_0 t} \lVert f_0\rVert_{L^{\infty}_{\omega_1}(\OO^\eps)} 
+  C (1+T+T^2)^p  T ^2 e^{-\nu_0  t }\underset{s\in[0, t]}{\sup}\left[ e^{\nu_0 s}\lVert  f_s\rVert_{\HH} \right]  \\
+  C (1+T+T^2)^p (1+T) e^{-\nu_2t}\underset{s\in [0,t]}{\sup}\left[e^{\nu_2 s}\lVert G_s\rVert_{L^{\infty}_{\omega_1\nu^{-1}}(\OO^\eps)}\right] .
\end{multline*}
for every $t\in [0, T]$, and some universal constants $C, p>0$.
\end{prop}

\begin{proof}
We consider $( t , x, v)\in  \UU_T^\eps$ such that $v\notin S_{x}$, then Proposition~\ref{prop:GainLinftyL2LB1} implies that for every $\lambda>0$ there is $\eps_4:=\eps_3( \lambda, T)$ such that for every $\eps \in (0, \eps_4)$  and every $\nu_2\in (0, \nu_0)$ there holds \eqref{eq:GainLinftyL2LB1}.

We then remark that \cite[Lemma 2, Item 4]{MR2679358} and \cite[Lemma 17]{MR2679358} imply that $\Sigma_0^\eps$ and $S_x$ respectively have Lebesgue measure zero. Therefore we may take the $L^{\infty}(\R^d)$ norm first and then the $L^{\infty}({\bar\Omega^\eps})$ norm on \eqref{eq:GainLinftyL2LB1} and we obtain
\begin{multline*}
\lVert  f_{t} \rVert_{L^{\infty}_{\omega_1}(\bar \OO^\eps)} \leq \left[ 1-\iota_0 + \lambda (1+  t  +  t ^2)\right] e^{-\nu_0  t }\underset{s\in[0, t]}{\sup}\left[ e^{\nu_0s}\lVert  f_s\rVert_{L^{\infty}_{\omega_1}( \bar \OO^\eps)}\right] 
+ c\,  (1+ t)  e^{-\nu_0  t } \lVert f_0\rVert_{L^{\infty}_{\omega_1}(\OO^\eps)} \\
+ c\,   t ^2 e^{-\nu_0  t }\underset{s\in[0, t]}{\sup}\left[ e^{\nu_0 s}\lVert  f_s\rVert_{\HH}\right] 
 + c\,  (1+t) e^{-\nu_2t}\underset{s\in [0,t]}{\sup}\left[e^{\nu_2 s}\lVert G_s\rVert_{L^{\infty}_{\omega_1\nu^{-1}}(\OO^\eps)}\right] ,
\end{multline*}
for every $t\in [0,T]$ and some constant $c>0$ such that $c \leq C_0 \lambda^{-p}$ for some constants $C_0, p>0$. We then choose $\lambda =\displaystyle \iota_0(2(1+T+T^2))^{-1}$, we absorb the small contributions, and we conclude by setting $C = 2C_0 (2-\iota_0)^{-1}$, and using the fact that $c\leq C_0\lambda^{-p}$.
\end{proof}

\subsection{Proof of Proposition~\ref{prop:LinftyEstimatePerturbedFiniteT}}
We split the proof into two steps.

\medskip\noindent
\emph{Step 1. (Choice of parameters and $L^2$ estimate)} We choose $T>0$ large enough such that 
$$
C_1 (1+T) (1+T+T^2)^p e^{-T\nu_0/2 } \leq 1,
$$ 
where $C_1>0$ is given by Proposition~\ref{prop:LinftyEstimatePerturbedFiniteT_0},
and we set $\eps_1^1 = \min(\eps_4(T), 1/2, \sqrt{\nu_0/\kappa^\star})$, where $\eps_4(T)>0$ is given by Proposition~\ref{prop:LinftyEstimatePerturbedFiniteT_0} and we recall that $\kappa^\star>0$ is given by Theorem~\ref{theo:Hypo}. 


We recall the hypocoercivity norm $\lvvv \cdot \rvvv$ given by Theorem~\ref{theo:Hypo} and the equivalency relation \eqref{eq:ClassicHypoEquivalence}. 
We now denote by $S_\LLL$ the semigroup generated by the solutions of Equation~\eqref{eq:LRBE}, which is given by Theorem~\ref{theo:ExistenceL2} and Remark~\ref{rem:ExistenceL2}, and using the Duhamel formula we have that
$$
f_t = S_\LLL(t) f_0 + \int_0^t S_\LLL (t-s) G_s \, \ds \qquad \forall t\geq 0.
$$
Using the fact that $\lla G_t\rra_{\OO^\eps}=0$ for every $t\geq 0$, the decay estimate \eqref{eq:HypoEquiv}, and the equivalency relation \eqref{eq:ClassicHypoEquivalence}, we then have that
$$
\lvvv f_t \rvvv \leq  e^{-\kappa^\star \eps^2 t} \lvvv f_0  \rvvv +  t \int_0^t  e^{-\kappa^\star \eps^2 (t-s)}  \lvvv G_s \rvvv \, \ds \leq e^{-\kappa^\star \eps^2 t} \lvvv f_0  \rvvv + ct \int_0^t  e^{-\kappa^\star \eps^2 (t-s)}  \lvv G_s \rvv_\HH \, \ds ,
$$
where the constant $c>0$ is given by Theorem~\ref{theo:Hypo}. 
Moreover, we observe that
\be\label{eq:ControlHHinfty} 
\lvv G\rvv_\HH \lesssim \left( \int_{\OO^\eps} \dx \right)^{1/2} \lvv G\rvv_{L^\infty_{\omega_1\nu^{-1}} (\OO^\eps)} \lesssim \eps^{-3/2} \lvv G\rvv_{L^\infty_{\omega_1\nu^{-1}} (\OO^\eps)} ,
\ee
thus we deduce that there is a constant $C_2>0$ such that
\be
 \lvvv f_t \rvvv \leq  e^{-\kappa^\star \eps^2 t} \lvvv f_0  \rvvv + C_2  \, t^2 \, \eps^{-3/2} \, e^{- \kappa^\star \eps^2t}\underset{s\in [0,t]}{\sup}\left[e^{ \kappa^\star \eps^2 s}\lVert G_s\rVert_{L^{\infty}_{\omega_1\nu^{-1}}(\OO^\eps)}\right],\label{eq:HypoPerturbed}
\ee
where $\kappa^\star>0$ is given by Theorem~\ref{theo:Hypo}. Putting together the estimate given by Proposition~\ref{prop:LinftyEstimatePerturbedFiniteT_0} and \eqref{eq:HypoPerturbed} we have that for every $\eps \in (0,\eps_1^1)$, every $\nu_2\in (0,\nu_0)$, and $\kappa_0 \in (0,\kappa^\star)$ there holds 
 \bean
\lVert  f_{t} \rVert_{L^{\infty}_{\omega_1}(\bar \OO^\eps)} &\leq&  C_1 (1+T) (1+T+T^2)^p e^{-\nu_0 t} \lVert f_0\rVert_{L^{\infty}_{\omega_1}( \OO^\eps)} + C_1 T^2 (1+T+T^2)^p e^{-\kappa^\star \eps^2 t} \lvvv f_0 \rvvv\nonumber \\
&& + C_1 (1+T+T^2)^p (1+T) e^{-\nu_2t} \underset{s\in [0,t]}{\sup}\left[e^{\nu_2 s}\lVert G_s\rVert_{L^{\infty}_{\omega_1\nu^{-1}}(\OO^\eps)}\right] \nonumber \\
&&+ C_1C_4 \, T^4 (1+T+T^2)^p \, \eps^{-3/2} \, e^{- \kappa_0 \eps^2t}\underset{s\in [0,t]}{\sup}\left[e^{ \kappa_0 \eps^2 s}\lVert G_s\rVert_{L^{\infty}_{\omega_1\nu^{-1}}(\OO^\eps)}\right],
\eean
where we have used the fact that $\nu_0-\kappa^\star(\eps)^2 \geq 0$ due to our choice of $\eps_1^1$.

We then set $\theta = \min(\nu_0, \kappa^{\star})/8$, $\nu_2= \eps^2\theta$ and $\kappa_0 = \theta$. 
We remark that these choices are possible due to the fact that $\eps\leq 1/2$ and $\theta\leq \nu_0/2$, thus $\eps^2\theta\in (0,\nu_0)$ and to the fact that $\theta \in (0, \kappa^{\star})$. 

Putting everything together we have that
 \begin{multline}\label{eq:LargeTimeTEstimateZeta1/4}
\lVert  f_T \rVert_{L^{\infty}_{\omega_1}(\bar \OO^\eps)} \leq  e^{-2 \theta \eps^2 T} \lVert f_0\rVert_{L^{\infty}_{\omega_1}(\OO^\eps)} + C_T e^{-2\theta \eps^2 T} \lvvv f_0\rvvv  \\
+\eps^{-3/2}  C_T e^{-\theta \eps^2 T} \underset{s\in [0,T]}{\sup}\left[e^{\theta\eps^2 s}\lVert G_s\rVert_{L^{\infty}_{\omega_1\nu^{-1}}(\OO^\eps)}\right] ,
\end{multline}
and 
\begin{multline}\label{eq:SmallTime_t_EstimateZeta1/4}
\lVert  f_t  \rVert_{L^{\infty}_{\omega_1}(\bar\OO^\eps)} \leq  C_T  e^{-2 \theta \eps^2 t}  \lVert f_0\rVert_{L^{\infty}_{\omega_1}(\OO^\eps)} + C_T e^{-2 \theta \eps^2 t}  \lvvv f_0\rvvv \\
+ \eps^{-3/2}  C_T e^{- \theta \eps^2 t} \underset{s\in [0,t]}{\sup}\left[e^{  \theta \eps^2 s}\lVert G_s\rVert_{L^{\infty}_{\omega_1\nu^{-1}}(\OO^\eps)}\right],
\end{multline}
for some constant $C_T>0$ and where \eqref{eq:SmallTime_t_EstimateZeta1/4} holds for all $t\in [0,T]$.

\medskip\noindent
\emph{Step 2. (Decay estimate)} We now set 
$$
X_t:= e^{\theta \eps^2 t} \lVert  f_t \rVert_{L^{\infty}_{\omega_1}(\bar \OO^\eps)} , \quad  Y_t:=e^{\theta \eps^2 t}   \lvvv  f_t \rvvv,  \quad  \text{ and } \quad \Phi_{t_0, t_1}:= \underset{s\in [t_0,t_1]}{\sup}\left[e^{ \theta \eps^2 s}\lVert G_s\rVert_{L^{\infty}_{\omega_1\nu^{-1}}(\OO^\eps)}\right].
$$
Translating \eqref{eq:HypoPerturbed} with $t=T$, and \eqref{eq:LargeTimeTEstimateZeta1/4} into this new notations we observe that
$$
Y_T \leq e^{- \theta \eps^2 T}  Y_0+ C_2 \, \eps^{-3/2}  T\,  \Phi_{0,T}  , \quad \text{ and } \quad X_T \leq  e^{- \theta \eps^2 T} X_0 + C_T e^{- \theta \eps^2 T} Y_0 +  \eps^{-3/2}  C_T  \Phi_{0,T}.
$$
We define $\eps_1 = \min(\eps_1^1, (\theta T)^{-1} \log 2)$ so that $e^{\theta\eps^2 T} -1\leq 1$, and we introduce a constant $\beta>0$ defined by 
$$
\beta =  {1\over 2C_T} \left( e^{\theta\eps^2 T} -1\right),
$$
so that, due to our choice of $\eps_1$, there holds $\beta C_T\leq 1/2$ and simultaneously
$$
\vartheta :=e^{- \theta \eps^2 T}(1+\beta C_T) =  \frac 12 \left( 1+ e^{-\theta\eps^2T} \right)  < 1, \quad \forall \eps >0.
$$
We then have that
\bear
Z_T:= Y_T + \beta X_T &\leq& (1+ \beta C_T) e^{- \theta \eps^2 T} ( Y_0 + \beta X_0) + \eps^{-3/2} (\beta C_T + C_2 T) \Phi_{0,T} \nonumber \\
&\leq& \vartheta Z_0 + \eps^{-3/2} \left( \frac 12 + C_2 T\right) \Phi_{0,T} \label{eq:LargeTimeTEstimateXY},
\eear
where we have used the fact that $\beta C_T   \leq 1/2$ due to the choice of $\beta$. On the other hand, using \eqref{eq:SmallTime_t_EstimateZeta1/4} and again our choice of $\beta$ we have that
\be\label{eq:SmallTime_t_EstimateXY}
Z_t  \leq\frac 32  e^{- \theta \eps^2 t} Z_0 + \eps^{-3/2}  \left( \frac 12 + C_2 T\right)  \Phi_{0, t} \qquad \forall t\in [0,T].
\ee
Then for any $\bar t \in \R$ there is $n\in \N$ such that $\bar t\in [nT, (n+1)T)$ and iterating first \eqref{eq:LargeTimeTEstimateXY} we have
\bean
Z_{\bar t} &\leq& \vartheta^n  Z_{\bar t-nT} +  \eps^{-3/2}  \left( \frac 12 + C_2 T\right) \Phi_{\bar t-nT, \bar t}  \left( \sum_{k=0}^{n-1} \vartheta^k\right) \\
&\leq & \frac 32 \vartheta^n Z_{0} +  \eps^{-3/2}  \left( \frac 12 + C_2 T\right) \left( \sum_{k=0}^{n-1} \vartheta^k\right) \Phi_{0, \bar t}   
\eean
where we have used \eqref{eq:SmallTime_t_EstimateXY} on the second line. Using the previous estimate we deduce that 
$$
\beta \, e^{\theta \eps^2 t} \lvv f_t \rvv_{L^{\infty}_{\omega_1}(\bar \OO^\eps)} \leq Z_t \leq \frac 32 Z_0 + \tilde C \underset{s\in [0,t]}{\sup}\left[e^{  \theta \eps^2 s}\lVert G_s\rVert_{L^{\infty}_{\omega_1\nu^{-1}}(\OO^\eps)}\right] \quad \forall t>0,
$$
where 
$$
\tilde C = \eps^{-3/2}  \left( \frac 12 + C_2 T\right) \left( \sum_{k=0}^{\infty} \vartheta^k\right) = \eps^{-3/2}  {1\over 1-\vartheta} \left( \frac 12 + C_2 T\right) = {\eps^{-3/2} \over 1- e^{-\theta \eps^2 T}}  \left( \frac 12 + C_2 T\right) .
$$
Finally, we observe that \eqref{eq:ClassicHypoEquivalence} and \eqref{eq:ControlHHinfty} imply together that
$$
Z_0 \lesssim \lvv f_0 \rvv_{L^{\infty}_{\omega_1}(\OO^\eps)} +  \lvv f_0\rvv_{\HH} \lesssim  \eps^{-3/2} \lvv f_0 \rvv_{L^{\infty}_{\omega_1}(\OO^\eps)},
$$
Altogether, using the definition of $\beta$ and $\vartheta$ we have that 
$$
 e^{\theta \eps^2 t} \lvv f_t \rvv_{L^{\infty}_{\omega_1}(\bar \OO^\eps)} \leq { \eps^{-3/2} C'\over e^{\theta\eps^2 T}-1} \lvv f_0 \rvv_{L^{\infty}_{\omega_1}(\OO^\eps)} +  {\eps^{-3/2} C' \over \left( e^{\theta\eps^2T} -1 \right) \left( 1-e^{-\theta \eps^2 T} \right) } \underset{s\in [0,t]}{\sup}\left[e^{  \theta \eps^2 s}\lVert G_s\rVert_{L^{\infty}_{\omega_1\nu^{-1}}(\OO^\eps)}\right] ,
$$
for some constant $C'>0$ independent of $\eps$. So we conclude by taking 
$$
C(\eps) =  {\eps^{-3/2} C' \over \left( e^{\theta\eps^2T} -1 \right) \left( 1-e^{-\theta \eps^2 T} \right) },
$$
and remarking that $C(\eps) \to \infty$ as $\eps\to 0$. 
\qed

\section{Stretching method for strongly confining weights in cylindrical domains}\label{sec:LinftyCylinder}
Throughout this section we assume Assumption \ref{item:RH2} to hold and we will prove the same result as in Proposition~\ref{prop:LinftyEstimatePerturbedFiniteT} within this framework. 

\begin{prop}\label{prop:LinftyEstimatePerturbedFiniteTCylinder} 
Consider Assumption \ref{item:RH2} to hold, $\omega_1\in\WWWW_1$ a strongly confining weight function, and let $f$ be a solution of Equation~\eqref{eq:PLRBE}. There are constructive constants $\eps_5, \theta  >0$ such that for every $\eps\in (0,\eps_5)$ there holds
\be
\lVert  f_t \rVert_{L^{\infty}_{\omega_1}(\bar \OO^\eps)} \leq C \, e^{-\theta \eps^2 t} \left(   \lVert f_0\rVert_{L^{\infty}_{\omega_1}(\OO^\eps)} + \underset{s\in[0,t]}{\sup}\left[ e^{\theta \eps^2 s} \lVert G_s \rVert_{L^{\infty}_{\omega_1\nu^{-1}}(\OO^\eps)}\right]\right) \qquad \forall t \geq 0 \label{eq:LinftyPerturbedFiniteTimeDecayCylinder} ,
\ee
for some constant $C = C(\eps)>0$ such that $C(\eps) \to \infty$ as $\eps \to 0$. 
\end{prop}

We remark that the ideas developed during Section \ref{sec:Linfty} do not immediately apply to cylindrical domains due to the presence of irregularities at the boundary. 
However, we can still provide a weighted $L^\infty$ control for the solutions of Equation~\eqref{eq:PLRBE} by using a more delicate control on the trajectories of particles within the cylinder. 

To do this we will exploit the geometrical properties of the domain described in Assumption \ref{item:RH2}, which will be expressed through a series of preliminary lemmas in the next subsection. Furthermore, we remark that the rest of this section is structured as Section \ref{sec:Linfty}.

\subsection{Preliminary lemmas}\label{ssec:TrajectoriesCylinder}
We present now the coordinates of multiple consecutive collisions along the boundary set presenting specular reflections, $\Lambda_3^\eps$. To be more precise, we consider $(t_0, x_0,v_0) \in \UU^\eps$ and we define the sequence $(t_N, x_N,v_N)$, as long as it makes sense, of specular collisions following the backwards trajectories as follows
\begin{equation}\label{eq:BackwardsSpecTrajectory}
\begin{array}{rcl}
    t_b(x_{N-1},V_{N-1}) &=& \inf\{s >0;\ X(-s, 0, x_{N-1}, v_{N-1})\notin\Omega^\eps\},\\
    t_N(t_{N-1}, x_{N-1},v_{N-1})&=&t_{N-1} -t_b(x_{N-1},v_{N-1}),\\
    x_N(t_{N-1}, x_{N-1},v_{N-1})&=& x_{N-1} - v_{N-1} (t_{N-1} - t_N),\\
    v_N(_{N-1}, x_{N-1},v_{N-1})&=& \VV_{x_N}(v_{N-1}).
\end{array}
\end{equation}
It is worth remarking that, when we say that we define $(t_N, x_N,v_N)$ as long as it makes sense, we mean that for every $j\in \llbracket 1, N \rrbracket$, there holds $x_j \in \Lambda^\eps_3$. 

During the rest of this section, we will be denoting the points $z\in \R^3$ as $z=(z^1, z^2, z^3)$ and we define $\widehat z = (z^2, z^3)$. We then remark that, in particular, if $x\in \Omega^\eps$ then $\widehat x\in \Omega_0^\eps$, the $2-$dimensional ball of radius $\eps^{-1}\RRRR$, and $x^1\in (-\eps^{-1} L, \eps^{-1} L)$. Moreover, we remark that if $x\in \Lambda^\eps_3$ then 
$$
n_x = \frac 1 {\lv \widehat x\rv } (0, x^2, x^3)=  \eps \RRRR^{-1} (0, x^2, x^3).
$$ 
Similarly, we explicitly have that the normal on the surface $\partial \Omega_0^\eps \subset \R^2$, at the point $\widehat x \in \partial \Omega_0^\eps$ is given by $\widehat n_{\widehat x}:= \eps  \RRRR^{-1} \widehat x $

We also define the set of non-smooth points of the boundary 
\be\label{def:SSSS}
\begin{aligned}
\SSSS^\eps&:=(\overline{\Lambda_1^\eps} \cap \overline{\Lambda^\eps_3}) \cup (\overline{\Lambda_2^\eps} \cap \overline{\Lambda^\eps_3}) \\
& \ = \left\{ x=(x^1, x^2, x^3)\in \R^3,\, x^1 =\pm L^\eps, \text{ and } (x^2)^2+(x^3)^2 = \eps^{-2} \RRRR^2 \right\},
\end{aligned}
\ee
and the set of singular velocities across multiple specular reflections through the backwards trajectories
$$
\begin{aligned}
 W_{t, x}&:=\{v\in\R^3;  \text{ such that starting at $(t,x,v)$ there is $N\in \N$ with }
 ((t_j, x_j , v_j))_{j=1}^{N}\\
 &\qquad  
 \text{ as defined in \eqref{eq:BackwardsSpecTrajectory}, with $t_j >0$ and $x_j\in \Lambda_3^\eps$ for every $j\in \llbracket 1, N\rrbracket$}, \\
 &\qquad 
 \text{ and such that if } (t_{N+1}, x_{N+1}, v_{N+1}) \text{ is the first collision against}\\
 &\qquad   
 \text{ the backwards trajectory starting at } (t_{N}, x_{N}, v_{N}) ,   \\
 &\qquad 
   \text{ given by \eqref{eq:BackwardsTrajectory} is such that } t_{N+1}>0 \text{ and }  x_{N+1} \in \SSSS^\eps\}.
\end{aligned}
$$
In words, for any fixed $t,x$ we have that $v\in W_{t,x}$, when starting from the point $(t,x,v)$ there are a certain number of consecutive purely specular collisions against $\Lambda_3^\eps$ followed by a collision against the singular set $\SSSS$.

We now prove a first result on the number of possible consecutive collisions through the specular reflection on $\Lambda_3^\eps$.

\begin{lem}\label{lem:CircunferenceSpecularNNNN}
For any fixed $M, T>0$ and any point $(t,x, v)\in  [0,T]\times \left( \Omega^\eps \cup \Lambda_3^\eps\right) \times \{ \lvert v\rvert\leq M\}$ there is a constant $\NNNN = \NNNN(T,x,M)$ such that there are not more than $\NNNN$ consecutive specular collisions against the boundary through the backwards trajectory starting from $(t,x,v)$, given by \eqref{eq:BackwardsSpecTrajectory}.
\end{lem}

This result is an inmediate consequence of the following geometrical lemma. 
\begin{lem}\label{lem:CircunferenceSpecular}
Consider a point $(t_0,x_0,v_0)\in (0,T) \times \Lambda^\eps_3\times \{ \lvert v\rvert\leq M\} $ such that $ n(x_0)\cdot v  = \eta$  for some constants $\eta, M>0$. Assume the two points of collision against the boundary through the backwards trajectories are $(t_1, x_1, v_1)\in (0,T) \times \Lambda^\eps_3\times \R^3$ and $(t_2, x_2, v_2)\in (0,T) \times \Lambda^\eps_3\times \R^3 $. Then $t-t_1 = t_1-t_2 \geq \TT$ for some $\TT = \TT ( M, \eta)>0$, $\lv v_0\rv =\lv v_1\rv = \lv v_2\rv$, and 
\be\label{eq:CircleSpecular}
n_{x_1}\cdot v_1 = n_{x_2}\cdot v_2 = n_{x_0} \cdot v_0 = \eta .
\ee
\end{lem}

\begin{proof}
We first observe that we elementary have that 
$$
\lv v_3\rv = \lv \VV_{x_3} v_2\rv = \lv v_2\rv = \lv \VV_{x_2} v_1 \rv = \lv \VV_{x_1} v_0 \rv = \lv v_0\rv.
$$
Moreover, there also holds
\be\label{eq:SpecularHatEquality}
\lv \widehat v_3\rv  = \lv \widehat \VV_{\widehat x_3} \widehat v_2 \rv = \lv \widehat v_2\rv = \lv \widehat \VV_{\widehat x_2} \widehat v_1 \rv = \lv \widehat\VV_{\widehat x_1} \widehat v_0 \rv = \lv \widehat v_0\rv,
\ee
where we have defined $\widehat \VV_{\widehat z} \widehat w = \widehat w -  (\widehat n_{\widehat z} \cdot \widehat w) \widehat w$, for any $z\in \partial \Omega_0^\eps$ and $w\in \R^3$.

We now denote now as $A_j$ the line perpendicular to $\widehat n_{\widehat x_j}$ passing through the point $\widehat x_j$. We also denote $\alpha := \angle(\widehat x_0, v_0) \in (0, \pi/2)$ and we remark that $\alpha \neq \pi/2$ due to the fact that there are collisions through the backwards trajectories, and $\alpha = \pi/2$ would imply that $\widehat v_0$ is a tangent velocity to the circumference in $\widehat x_0$. 

We then observe that, from the tangent theorem for the circle, the arc $\bigfrown{\widehat x_0 \widehat x_1} =  \pi_2 - \alpha =: \beta$. Moreover, using the same argument, we have that $\angle(A_1, v_0) = \beta$, thus $\angle (v_0, x_1) = \alpha + \pi/2$. We now bserve that from the very definition of the specular boundary conditions there holds $x_1\cdot v_0 = - x_1\cdot v_1$. This together with \eqref{eq:SpecularHatEquality} implies that $\pi - \angle(v_1, x_1) = \angle (v_0, x_1)$. We deduce then that $\angle (v_1, x_1) = \alpha$, and in particular that $\bigfrown{\widehat x_1 \widehat x_2} = \beta$. 
The fact that $\angle (v_0, x_0) = \angle (v_1, x_1)$ and \eqref{eq:SpecularHatEquality} imply \eqref{eq:CircleSpecular}. Moreover, the fact that $\bigfrown{\widehat x_0 \widehat x_1} =\bigfrown{\widehat x_1 \widehat x_2}$ imply that the $\overline {\widehat x_0 \widehat x_1}= \overline {\widehat x_1 \widehat x_2}$. This information together with \eqref{eq:SpecularHatEquality} and the very definition of the trajectories imply that $t-t_1 = t_1-t_2$. This concludes the proof. 
\end{proof}

We present now some results regarding the singular sets in the framework of cylindrical domains. 

\begin{lem}\label{lem:ZeroMeasureSetCylinderSx}
Let $x\in \Omega^\eps \cup \Lambda_3^\eps$, the set $S_x$, that we recall was defined in \eqref{def:Sx}, has Lebesgue measure zero. 
\end{lem}

\begin{proof}
We consider first $x\in \Omega^\eps$ and assume there is $v\in \R^d$ and $t>0$ such that $x_1 = x-tv\in \Lambda^\eps_3$ with $n_{x_1} \cdot v = 0$. However, we recall that $n_{x_1} = (0, x_1^2, x_1^3)/\lv \widehat x_1\rv$, therefore $n_{x_1} \cdot v = x_1^2 v^2+ x_1^3v^3$ has to be equal to zero. This then implies that $\widehat n_{\widehat x} \perp \widehat v$, and if we now go through the backwards trajectories we observe that $\widehat x_1 + t\widehat v\notin \Omega_0^\eps$ which is in contradiction with the fact that $x\in \Omega^\eps$. We thus deduce that $S_x=\emptyset$ for any $x\in \Omega^\eps$. 

Taking $x\in \Lambda^\eps_3$, we remark that $S_x$ has an explicit form and  we immediately observe that it has codimension 1 in $\R^3$, thus Lebesgue measure zero.
\end{proof}

\begin{lem}\label{lem:ZeroMeasureSetCylinder}
Consider some arbitrary $M, T>0$ and a point $(t,x)\in [0,T]\times (\Omega \cup \Lambda_3^\eps)$, the set $W_{t,x} \cap \{v\in \R^3, \, \lv v \rv \leq M\} \setminus S_x$ has Lebesgue measure zero in $\R^3$. 
\end{lem}

\begin{proof}
This result is classical in the study of the ergodic properties of \emph{dynamical billiards}, i.e particles moving by the dynamics \eqref{eq:carachteristics} and colliding against the boundary following specular (also called elastic in the framework of billiards) reflections. We refer the interest reader towards \cite{MR2229799, MR1874973, MR274721, MR2168892, MR357736} and the references therein for more information in this kind of systems and the current known results. However, even if it is a well accepted result within this field, we weren't able to find explicit references regarding the measure preservation of this flow, except in some particular two dimensional cases (see for instance in the previous references). Therefore, we provide now a sketch of a proof. 

We define the Hamiltonian energy function $H(v):= \lv v\rv^2/2$ and the flow $\Phi^t$ generated by the system
\be\label{eq:Hamiltonian}
\dot{x} = - {\partial H \over \partial v} (v) = -v ,\qquad \dot{v} = {\partial H \over \partial x} (v) = 0,
\ee
complemented with the specular boundary reflection, that we recall is given by 
\be\label{eq:Specular_reflection}
(x,v ) =  (x, \VV_x v) \quad \text{ for every } x\in  \Lambda_3^\eps, \text{ and } v\in \R^3 \text{ such that } n(x)\cdot v>0.
\ee
We will then prove that the flow $\Phi^t$ defined this way is measure preserving, that is for every Borel set $A\subset (\Omega^\eps \cup \bar \Lambda_3^\eps) \times \R^3$ there holds
\be\label{eq:Def_Measure_preservation}
\mu_L (\Phi^t (A)) = \mu_L (A) = \mu_L ((\Phi^t)^{-1} (A)),
\ee
for every $t\geq 0$, and where $\mu_L$ is the Lebesgue measure in phase space $\dx\dv$.

To do this, we remark first that during free flight, i.e between collisions with the boundary, the flow $\Phi^t$ corresponds to a Hamiltonian flow, therefore Liouville's theorem \cite[Part II, Chapter 3, Theorem 16.1]{MR1345386} implies that this map is measure preserving with respect to $\mu_L$. 
We prove then that the specular reflection map given by \eqref{eq:Specular_reflection} is a measure preserving map.

To prove this we observe that 
\be\label{eq:Measure_Preserv_Identity}
\text{a map }\psi \text{ is measure preserving iff } \int_{  \Lambda_3^\eps \times \R^3} F(x,v) \dx\dv = \int_{  \Lambda_3^\eps \times \R^3} F \circ \psi(x, v) \dx\dv,
\ee
for every integrable function $F\in L^1( \Lambda_3^\eps \times \R^3)$. 

In particular, we remark that if $\psi(x,v) = \VV_x v$, then \eqref{eq:Measure_Preserv_Identity} holds by using the change of variables $v\mapsto \VV_x v$.

We have then proved that the flow $\Phi^t$ is measure preserving and the conclusion follows by using this and the fact that $\mu_L(\SSSS) = 0$.
\end{proof}

We provide now a control on the angle of reflections against the normal on $\Lambda_3^\eps$ after a diffusive collision. To do this, we introduce the vector field $\nnnn:\R^3\to \R^3$ defined as $\nnnn(x) = (0,x^2,x^3) /\lv \widehat x\rv$, and we remark that $\nnnn(x) = n_x$ for every $x\in \Lambda_3^\eps$.

\begin{lem}\label{lem:SpecularAngleDiffusiveReflection}
Assume $t>0$, $x\in \Lambda_1^\eps\cup \Lambda_2^\eps$, $v\in \R^3$, such that $(x,v)\in \Sigma_+^\eps$, and $\lv n(x)\cdot v \rv >\eta_0$ for some $\eta_0>0$. Consider $(t_1, x_1, v_1)$ the point of collision against the boundary through the backwards trajectory starting at $(t,x,v)$ given by \eqref{eq:BackwardsTrajectory}, and assume $x_1\in \Lambda^\eps_3$. 

For every $\eta>0$ there is a constructive constant $A>0$ such that  if $\lv \nnnn(x) \cdot v\rv >\eta$ then $\lv n(x_1) \cdot v_1\rv \geq A$, uniformly in $x$.
\end{lem}

\begin{proof}
This result is a consequence of elementary geometrical properties of the circumference, therefore we will skip it.
\end{proof}

Finally, we present the following stretching lemmas for the cylindrical framework. 

\begin{lem}
\label{lemma3.2:GeneralReflCylinder}
Assume \ref{item:RH2} to hold, and take $(t,x,v)\in[0,T]\times \Lambda_3^\eps\times\{ \lvert v\rvert\leq M, \lvert n(x_1)\cdot v\rvert>\eta\} $ for some $M ,T,\eta>0$, and where $(x,v) \in \Sigma_+^\eps$. There exists $\varepsilon_S =\eps_S(\eta, M,T):=2\RRRR M^{-2} T^{-1} \eta$ such that for every $\varepsilon \in (0,\varepsilon_S)$ there is no more specular reflection along the backwards trajectory.
\end{lem}

\begin{proof}
We consider the point $(t_1, x_1, v_1)$ as the point of collision against the boundary through the backwards trajectory starting from $(t,x,v)$ and which is given by \eqref{eq:BackwardsTrajectory}. 
If we assume that $x\in \Lambda^\eps_3$, that means that $\widehat x_1 \in \partial \Omega_0^\eps$ and $x_1^1 \in (-\eps^{-1} L, \eps^{-1} L)$. In particular, we remark that this means that the time of the trajectory from $x$ towards $x_1$, $t-t_1$, is equal to the time of the backwards trajectory from $\widehat x$ towards $ \widehat x_1$ with velocity $\widehat v$, within $\Omega_0^\eps$. 

We call now $d:=  \lv \widehat x - \widehat x_1\rv$, the length of the chord formed by the points $\widehat x$ and $\widehat x_1$ within $\bar \Omega_0^\eps$. We call then $\OOO$ the center point of the circumference $\partial \Omega^\eps_0$, and we observe that the angle 
$$
\angle(\widehat x \OOO\widehat x_1) = \pi - 2\angle(\widehat n_{\widehat x}, \widehat v),
$$
where we remark that the first one is interpreted as the angle between the segments $\overline{\widehat x\OOO}$ and $\overline{\OOO\widehat x_1}$, and the second is interpreted as the angle between the two vectors $\widehat n_{\widehat x}$ and $\widehat v$.

Using then the law of cosines we elementary deduce that 
\bean
d^2& =& 2\RRRR^2 \eps^{-2} \left( 1-\cos \left( \angle(\widehat x, \OOO, \widehat x_1)\right)\right) =  2\RRRR^2 \eps^{-2} \left( 1-\cos \left( \pi - 2\angle\left(\widehat n_{\widehat x}, \widehat v\right) \right)\right)\\
&=& 2\RRRR^2\eps^{-2} \left( 1+ \cos \left( 2\angle\left(\widehat n_{\widehat x}, \widehat v\right) \right)\right) = 4\RRRR^2\eps^{-2} \left\lv\cos \left( \angle\left(\widehat n_{\widehat x}, \widehat v\right) \right)\right\rv^2 
\geq 4\RRRR^2\eps^{-2} { \eta^2\over M^2 }
\eean
Therefore 
$$
t-t_1 = {\lv \widehat x - \widehat x_1\rv\over \lv v\rv} \geq 2\RRRR \eps^{-1} { \eta\over M^2 },
$$
and we conclude by remarking that if $\eps \in (0, \eps_S)$, then $t-t_1 >T$ which is a contradiction with the fact that, from its very definition $0\geq t-t_1\leq T$. 
\end{proof}

\begin{lem}
\label{lemma3.2:GeneralReflCylinderDiff}
Assume \ref{item:RH2} to hold, and $(t,x,v)\in[0,T]\times (\Lambda_1^\eps\cup \Lambda_2^\eps) \times\{ \lvert v\rvert\leq M, \lvert n(x)\cdot v\rvert>\eta\} $ for some $M, T,\eta>0$, and where $(x,v) \in \Sigma_+^\eps$. There exists $\varepsilon_D = \eps_D(M,T):=2 L  M^{-1} T^{-1}$ such that for every $\varepsilon \in (0,\varepsilon_D)$ there is no more diffussive reflection along the backwards trajectory.
\end{lem}

\begin{proof}
We will follow the same ideas as those exposed during Lemma~\ref{lemma3.2:GeneralReflCylinder}. We assume, without loss of generality that $x\in \Lambda_1^\eps$, i.e $x^1=-L\eps^{-1}$.

We denote $(t_1, x_1, v_1)$ as the point of collision against the boundary through the backwards trajectory starting from $(t,x,v)$ and which is given by \eqref{eq:BackwardsTrajectory}. 

We recall now that for every $z\in \Lambda_3^\eps$ and $w\in \R^3$ there holds that 
$\lv \VV_z w\rv = \lv w\rv$ and $(\VV_z w)^1 = w^1$. Therefore, and because $\lvert n(x)\cdot v\rvert>\eta$, we deduce that if there is a diffusive reflection through the backwards trajectory it has to be because the particle eventually arrives at $\Lambda_2^\eps$,  even after several possible specular reflections against $\Lambda_3^\eps$ of the form \eqref{eq:BackwardsSpecTrajectory}. Moreover, we also observe in an elementary way that the trajectory that minimizes the time to arrive to $\Lambda_2^\eps$ is when $x_1^1=L\eps^{-1}$ and $x_1^2 = x^2$, $x_1^3 = x^3$. 

Due to this analysis, if we assume that $x_1^1=L\eps^{-1}$, $x_1^2 = x^2$, $x_1^3 = x^3$, and we prove that this collision cannot happen in the time interval $[0,T]$ then this will conclude the proof. 
Finally, we just need to observe that, for every $\eps\in (0,\eps_D)$, there holds
$$
t-t_1 = {\lv x^1 - x_1^1 \lv \over \lv v^1\lv }  \geq {2L\eps^{-1} \over M} >T,
$$ 
which concludes the proof.
\end{proof}

\subsection{Regularizing effect of $K$}\label{ssec:RegularKCylinder} We extend in this subsection the regularizing effect generated by the interplay between the free transport semigroup and the non-local operator $K$ to the cylindrical framework.

\begin{prop} \label{prop:KregularizationCylinder}
Consider Assumption \ref{item:RH2} to hold, $\omega_1\in \WWWW_1$ a strongly confining admissible weight function and let $ f$ be a solution of Equation~\eqref{eq:PLRBE_T}. For every $\lambda>0$ there is $\eps_6 = \eps_6(\lambda, T) $ such that for every $\eps\in (0, \eps_6)$ and every $\nu_2\in (0, \nu_0)$ we have that for every point $(t, x, v)\in  \UU_T^\eps$ with $v\notin S_x\cup W_x$ there holds 
\begin{multline*}
\omega_1 (v) \lvert S_{\TTT}*_\sigma K  f (t, x,v)\rvert  \leq \lambda (t + t^2)  e^{-\nu_0t}\underset{s\in[0,t]}{\sup}\left[ e^{\nu_0s}\lVert  f_s\rVert_{L^{\infty}_{\omega_1}( \OO^\eps)}\right] +  Ct e^{-\nu_0 t} \lVert  f_0\rVert_{L^{\infty}_{\omega_1}(\OO^\eps)} \\
 +  Ct^2 e^{-\nu_0 t}\underset{s\in[0,t]}{\sup}\left[ e^{\nu_0 s}\lVert  f_s\rVert_{ \HH}\right] + C (1+t) e^{-\nu_2 t}\underset{s\in [0,t]}{\sup}\left[e^{\nu_2 s}\lVert  G_s\rVert_{L^{\infty}_{\omega_1\nu^{-1}}(\OO^\eps)}\right] ,
\end{multline*}
for a constant $C= C(\lambda)>0$ and any $\sigma \in [0,T]$ such that $x-v(t-s)\in \bar\Omega^\eps$. Furthermore, there holds $C(\lambda) \lesssim \lambda^{-p}$ for some constant $p>0$.
\end{prop}

\begin{proof}
We define 
$$
\IIII:=  \omega_1(v) \lvert S_{\TTT}*_\sigma K  f(t, x,v)\rvert  = \omega_1 (v) \left\lvert \int_{\sigma}^t S_{\TTT}(t-s) K  f (s, x,v) \ds \right\rvert,
$$
and arguing as during the Step 1 of the proof of Proposition~\ref{prop:Kregularization} we have that by using \ref{item:K1} we deduce that for every $N_1>0$ there is $m_1(N_1)>0$ such that
\bear
\IIII &\leq& {1\over N_1} t e^{-\nu_0t}\underset{s\in [0, t]}{\sup}\left[e^{\nu_0 s}\lVert  f_s\rVert_{L^{\infty}_{\omega_1}(\OO^\eps)} \right]+ S_{\TTT}*_\sigma K_{m_1} \left( \omega_1(v)  \lvert f(t, x,v)\rvert \right) , \label{eq:STKN1_cyl}
\eear
where we recall that $K_{m_1}$ has been defined in \ref{item:K1}.

Using then Lemma~\ref{lem:ZeroMeasureSetCylinderSx} and Lemma  \ref{lem:ZeroMeasureSetCylinder} we deduce that we can take the previous integral over the set $\R^3\setminus (S_{x_s}\cup W_{t-s, x_s})$, therefore we can define $(s_1, x_1, v_1)$ as the point of the first bounce against the boundary through the backwards trajectory starting at $(s,x_s, v_*)$, with  $( x_1, v_1) \in \Sigma_+$, $x_1\notin \SSSS$, and whose formulas are given by \eqref{eq:BackwardsTrajectory}. In particular, we remark that  $x_1 = x_s - (s-s_1) v_*$.

For any $\alpha_0>0$ we define now $k_{0,m_1} = \Ind_{\{ \lvert n(x_1)\cdot v_*\rvert  \geq \alpha_0\}}k_{m_1}$, and the operator $K_{0,m_1} $ as in \eqref{eq:DefKm1}. Repeating the arguments from the Step 1 of the proof of Proposition~\ref{prop:Kregularization} we have that 
$$
\IIII \leq \left( {1\over N_1}   + 2C_kDm_1^{4}\alpha_0  \right)t e^{-\nu_0t}\underset{s\in [0, t]}{\sup}\left[e^{\nu_0 s}\lVert  f_s\rVert_{L^{\infty}_{\omega_1}(\OO^\eps)} \right] + \IIII_0,
$$
with
$$
\IIII_0:= S_{\TTT}*_\sigma K_{0,m_1}\left( \omega_1 (v) \lvert   f(t,x,v)\rvert\right) = \int_\sigma^t e^{-\nu(v)(t-s)}\int_{\R^3} k_{0, m_1}(v, v_*)\omega_1(v_*) \lvert  f(s, x_s, v_*) \rvert  \d v_*\ds.
$$

Using the Duhamel formula and repeating one more time the arguments from the Step 1 of the proof of the Proposition~\ref{prop:Kregularization}, we further deduce that for every $N_2, \alpha_2>0$ there is $m_2(N_2)>0$ such that for every $\nu_2\in (0,\nu_0)$ there holds
\bean
\IIII&\leq& \left( {1\over N_1}  + 2C_kDm_1^{4}\alpha_0+  {C_kD m_1^{4}\over N_2}+ C_k^2D^2m_1^{4}m_2^{4} \alpha_1 \right)   t e^{-\nu_0t}\underset{s\in [0, t]}{\sup}\left[e^{\nu_0 s}\lVert   f_s\rVert_{L^{\infty}_{\omega_1}(\OO^\eps)} \right] \\
&&+ C_k D m_1^{4} t e^{-\nu t} \lVert   f_0\rVert_{L^{\infty}_{\omega_1}(\OO^\eps)} + C_kDm_1^{4}{\nu_1t \over \nu_0 - \nu_2}  e^{-\nu_2 t}\underset{s\in [0,t]}{\sup}\left[e^{\nu_2 s}\lVert   G_s\rVert_{L^{\infty}_{\omega_1\nu^{-1}}(\OO^\eps)}\right] \\
&&+C_k^2Dm_1^{1+3/2}m_2^{1+3/2} \omega_1(m_2) \alpha_1^{-3/2} t^2 e^{-\nu_0 t}\underset{s\in[0,t]}{\sup}\left[ e^{\nu_0 s}\lVert   f_s\rVert_{{\HH}}\right]   +\IIII_4,
\eean
where we have defined
\beqn
\IIII_4 = \int_\sigma^t e^{-\nu(v)(t-s)} \int_{\vert v_*\rvert \leq m_1, \, \lvert n(x_1)\cdot v_* \rvert\geq \alpha_0}   k_{0,m_1} (v, v_*) \omega_1(v_*) \lvert   f(s_1, x_1, v_*)\rvert \, \d v_*\ds ,
\eeqn
Using the Maxwell boundary conditions yields 
$$
f(s_1, x_1, v_*) =(1-\iota^\eps(x_1)) f(s_1, x_1, \VV_{x_1}v_*) + \iota^\eps(x_1)  \MMM(v_*) \widetilde f(s_1, x_1),
$$
thus we deduce that $\IIII_4 \leq \IIII^S + \IIII^D$ where $\IIII^S$ and $\IIII^D$ are given by \eqref{eq:IIII_4_S}, \eqref{eq:IIII_4_D} respectively.

\medskip\noindent
\emph{Step 1. (Control of the diffusive term $\IIII^D$)}
 We observe first that, since $ \MMM\omega_1\leq 1$, there holds
$$
\IIII^D \leq \int_\sigma^t  \int_{\R^3} k_{0,m_1}(v, v_*) e^{-\nu_0(t-s_1)} \int_{\R^3} \left\lvert  f(s_1, x_1, u_*)\right\rvert (n(x_1)\cdot u_*)_+\,  \d u_* \d v_*\ds,
$$
Using then \ref{item:K5} we have that for every $\lambda_1>0$ there are $M_1, \eta>0$ such that
\beqn
\IIII^D 
\leq \lambda_1 C_k D m_1^{4} t e^{-\nu_0 t}\underset{s\in[0,t]}{\sup}\left[ e^{\nu_0s}\lVert  f_s\rVert_{L^{\infty}_{\omega_1}(\OO^\eps)}\right]  + \IIII^D_0 
\eeqn
where we have used the properties of $k_{m_1}$ as exposed in \ref{item:K1}, and we have defined
\be\label{eq:IIIID0def}
\IIII^D_0 =  \int_\sigma^t  \int_{\R^3}k_{m_1}(v, v_*) e^{-\nu_0(t-s_1)} \int_{\lvert u_*\rvert \leq M_1,\, \lvert n_{x_1}\cdot u_*\rvert >\eta} \left\lvert  f(s_1, x_1, u_*)\right\rvert (n(x_1)\cdot u_*)_+ \, \d u_* \d v_* \ds.
\ee
Arguing now as during the Step 1 we deduce that we may rewrite the previous integral as
being integrated over the set $ \{\lvert u_*\rvert \leq M_1, \lvert n(x_1)\cdot u_*\rvert > \eta \}\setminus (S_{x_1} \cup W_{t_1, x_1})$. 
Using then Lemma~\ref{lemma3.2:GeneralReflCylinderDiff} there is $\eps_6^1= \eps_D(\eta, M_1, T)$ such that for every $\eps\in (0,\eps_6^1)$ there is no more collision against the diffusive boundary through the backwards trajectory starting at $(s_1, x_1, u_*)$. We then denote $(s_2, x_2, u_2)$ as the point of collision against the boundary through the backwards trajectory starting at $(s_1, x_1, u_*) $, moreover we remark that $(x_2, u_*) \in \Sigma_-^\eps$ and due to the previous choice of $\eps$ we also have that
$$
x_2= x_1- u_*(s_1- s_2) \in \Lambda^\eps_3 .
$$
Then, for any $\eta_1>0$ small, we define $\UUU_1 = \{\lvert u_*\rvert \leq M_1,\, \lvert n_{x_1}\cdot u_*\rvert >\eta, \lv \nnnn(x_1) \cdot   u_* \rv \leq \eta_1\}$, $\UUU_2 = \{\lvert u_*\rvert \leq M_1,\, \lvert n_{x_1}\cdot u_*\rvert >\eta, \lv \nnnn( x_1) \cdot  u_* \rv > \eta_1\}$, and we have that $\IIII^D_0 \leq  \IIII^D_1 + \IIII^D_2$ with
\bean
\IIII^D_1 &:=&  \int_\sigma^t  \int_{\R^3}k_{m_1}(v, v_*) e^{-\nu_0(t-s_1)}   \int_{\UUU_1} \left\lvert  f(s_1, x_1, u_*)\right\rvert (n(x_1)\cdot u_*)_+ \, \d u_* \d v_* \ds , \\
\IIII^D_2 &:=& \int_\sigma^t  \int_{\R^3}k_{m_1}(v, v_*) e^{-\nu_0(t-s_1)}   \int_{\UUU_1} \left\lvert  f(s_1, x_1, u_*)\right\rvert (n(x_1)\cdot u_*)_+ \, \d u_* \d v_* \ds .
\eean
In order to control $\IIII^D_1$ we perform the change of variable $u_*^{\mathrm par}= (\nnnn(x_1) \cdot u_*)\nnnn(x_1)$ and its perpendicular direction $u_*^\perp = u_* - u_*^{\mathrm par}$ such that  $u_* = u_*^{\mathrm par} + u^\perp_*$, and we obtain that 
\bean
\IIII^D_1 &\leq& M_1 e^{-\nu_0 t} \underset{s\in [0,t]}{\sup} \left[e^{\nu_0 s} \lvv f_s\rvv_{L^\infty_{\omega_1}(\OO^\eps)}\right]  \int_\sigma^t  \int_{\R^3}k_{m_1}(v, v_*)   \int_{\lv u^\perp_* \rv \leq M_1 }  \d u_*^\perp \int_{ -\eta_1}^{ \eta_1} \d u_*^{\mathrm par}  \d v_* \ds \\
&\leq & 2 \eta_1 \, t M_1^3 D^2 C_k m_1^3 e^{-\nu_0 t} \underset{s\in [0,t]}{\sup} \left[e^{\nu_0 s} \lvv f_s\rvv_{L^\infty_{\omega_1}(\OO^\eps)}\right] ,
\eean
where we have used the bounds on $k_{m_1}$ as exposed during \ref{item:K1}.

\medskip\noindent
\emph{Step 1.1.} 
Using then the Duhamel decomposition and arguing as during the Step 3 of Proposition~\ref{prop:Kregularization} and we deduce that for every $N_3, \alpha_3>0$ there is $m_3(N_3)>0$ such that there holds 
\bean
\IIII^D_0  &\leq & \left(  \lambda_1 C_k D m_1^{4} t +   {t^2 C_k^2 D^2 m_1^{4} M_1^{4} \over N_3}  + C_k^2D^3m_1^{4} M_1^{4}m_3^{4}\alpha_3 t + 2 \eta_1 \, t M_1^3 D^2 C_k m_1^3  \right)   e^{-\nu_0 t}\underset{s\in[0,t]}{\sup}\left[ e^{\nu_0s}\lVert   f_s\rVert_{L^{\infty}_{\omega_1}(\OO^\eps)}\right] \\
&&+ C_k D^2 m_1^{4} M_1^{4} t e^{-\nu_0 t} \lVert   f_0\rVert_{L^{\infty}_{\omega_1}(\OO^\eps)} +  {\nu_1 C_k D^2m_1^{4} M_1^{4} t \over \nu_0 - \nu_2}  e^{-\nu_2 t}  \underset{s\in [0,t]}{\sup}\left[e^{\nu_2 s}\lVert   G_s\rVert_{L^{\infty}_{\omega_1\nu^{-1}}(\OO^\eps)}\right]   \\
 && + C_k^2D^2m_1^{4}m_4^{1+3/2} M_1^{1+3/2}\omega(m_3) \alpha_3^{-3/2} t^2e^{-\nu_0 t}\underset{s\in[0,t]}{\sup}\left[ e^{\nu_0 s}\lVert  f_s\rVert_{{\HH}}\right]  +\IIII^D_S
\eean
for every $\nu_2\in (0, \nu_0)$, and where we have defined
$$
\IIII^D_S = \int_\sigma^t  \int_{\R^3}k_{m_1}(v, v_*) e^{-\nu_0(t-s_2)} \int_{\UUU_2} \left\lvert  f(s_2, x_2, \VV_{x_2} u_*) \right\rvert (n(x_1)\cdot u_*)_+ \, \d u_* \d v_*\ds.
$$

Using then Lemma~\ref{lem:SpecularAngleDiffusiveReflection}, we deduce that there is a constructive constant $A_1=A_1(\eta)>0$ such that $\lv n(x_2) \cdot \VV_{x_2} \rv \geq A_1$ uniformly in $x$. Applying then Lemma~\ref{lemma3.2:GeneralReflCylinder} we have that there is $\eps_6^2 = \eps_S (A_1, M_1, T)$, such that for every $\eps\in (0, \min(\eps_6^1, \eps_6^2))$ there is no more bounce against the specular reflection. Therefore the Duhamel formula gives $\IIII^D_S \leq  \IIII^D_{S,1} + \IIII^D_{S,2} + \IIII^D_{S,3}$ with
\bean
\IIII^D_{S,1} &=& \int_\sigma^t  \int_{\R^3} k_{m_1}(v, v_*) e^{-\nu_0(t-s_2)} \int_{\UUU_2} \left\lvert S_{\TTT}(s_2)  f_0(x_2, \VV_{x_2} u_*)\right\rvert (n(x_1)\cdot u_*)_+, \d u_* \d v_*\ds, \\
\IIII^D_{S,2} &=& \int_\sigma^t \int_{\R^3} k_{m_1}(v, v_*)  e^{-\nu_0(t-s_2)} \int_{\UUU_2} \left\lvert S_{\TTT}* K  f (s_2, x_2, \VV_{x_2} u_*) \right\rvert (n(x_1)\cdot u_*)_+,\d u_* \d v_*\ds, \\
\IIII^D_{S,3} &=& \int_\sigma^t \int_{\R^3} k_{m_1}(v, v_*)  e^{-\nu_0(t-s_2)} \int_{\UUU_2} \left\lvert S_{\TTT}*   G(s_2 , x_2, \VV_{x_2}u_*) \right\rvert (n(x_1)\cdot u_*)_+,\d u_* \d v_*\ds,
\eean
 and we will control each of these terms separately. We compute then in a similar way as during the Step 3 of the proof of Proposition~\ref{prop:Kregularization}, and we have that, for every $\nu_2\in (0, \nu_0)$, and every $N_4>0$, there is $m_4(N_4)>0$ such that 
\begin{multline*}
\IIII^D_{S,1}  + \IIII_{S,2}^D + \IIII^D_{S,3}\leq C_k D^2 m_1^{4} M_1^{4} t e^{-\nu_0 t} \lVert  f_0\rVert_{L^{\infty}_{\omega_1}(\OO^\eps)} + {\nu_1 C_k D^2m_1^{4} M_1^{4} t\over \nu_0 - \nu_2} e^{-\nu_2 t}  \underset{s\in [0,t]}{\sup}\left[e^{\nu_2 s}\lVert   G_s\rVert_{L^{\infty}_{\omega_1\nu^{-1}}(\OO^\eps)}\right] \\
+  {t^2 C_k^2 D^2 m_1^{4} M_1^{4} \over N_4} e^{-\nu_0 t}\underset{s\in[0, t]}{\sup}\left[ e^{\nu_0s}\lVert   f_s\rVert_{L^{\infty}_{\omega_1}(\OO^\eps)}\right]  + C_k^2m_1m_4M_1\,   \IIII^D_{S,2,0},
\end{multline*}
where we have defined
$$
\IIII_{S,2,0}^D = \int_\sigma^t  \int_{\lvert v_*\rvert \leq m_1}  \int_{\{\lvert u_*\rvert \leq M_1\}} \int_0^{s_2} e^{\nu_0(t-r)} \int_{\lvert u'\rvert \leq m_4} \omega (u')\lvert  f (r, x_2- \VV_{x_2} u_*(s_1 - r), u') \rvert .
$$
Arguing then exactly as during the proof of Lemma~\ref{lem:Specular_Jacobian} and using Lemma~\ref{lem:Specular_Regularization} we have that for every $\alpha_4>0$ there is $\eps_6^3 = \eps_U(\alpha_4)>0$ such that for every $\eps\in (0,\min(\eps_6^1, \eps_6^2, \eps_6^3))$ there holds
\bean
\IIII_{S, 2,0}^D &\leq & D^3\, m_1^3\, M_1^3\, m_4^3\, \alpha_4 \, t\,  e^{-\nu_0 t}\underset{s\in[0, t]}{\sup}\left[ e^{-\nu_0s}\lVert  f_s\rVert_{L^{\infty}_{\omega_1}(\OO^\eps)}\right] \\
&&+ D^2 2^{1/2}m_1^3M_1^{3/2}m_4^{3/2} \, \omega_1(m_4) \, \alpha_4^{-3/2} \, t^2\, e^{-\nu_0 t}\underset{s\in[0,t]}{\sup}\left[ e^{\nu_0 s}\lVert  f_s\rVert_{{\HH}}\right].
  \eean
Putting together the above estimates controlling the diffusion term we have that 
\bean
\IIII^D  &\leq & \left(  \lambda_1 C_k D m_1^{4} t  +   {t^2 C_k^2 D^2 m_1^{4} M_1^{4} \over N_3}  + C_k^2D^3m_1^{4} M_1^{4}m_3^{4}\alpha_3 t + 2 \eta_1 \, t M_1^3 D^2 C_k m_1^3 \right.\\
&&\left.+   {t^2 C_k^2 D^2 m_1^{4} M_1^{4} \over N_4}  + C_k^2D^3m_1^{4} M_1^{4}m_4^{4}\alpha_4 t  \right)   e^{-\nu_0 t}\underset{s\in[0,t]}{\sup}\left[ e^{\nu_0s}\lVert   f_s\rVert_{L^{\infty}_{\omega_1}(\OO^\eps)}\right] \\
&&+ 2C_k D^2 m_1^{4} M_1^{4} t e^{-\nu_0 t} \lVert   f_0\rVert_{L^{\infty}_{\omega_1}(\OO^\eps)} +  2 {\nu_1 C_k D^2m_1^{4} M_1^{4} t\over \nu_0 - \nu_2}  e^{-\nu_2 t}  \underset{s\in [0,t]}{\sup}\left[e^{\nu_2 s}\lVert   G_s\rVert_{L^{\infty}_{\omega_1\nu^{-1}}(\OO^\eps)}\right] \\
&& + C_k^2D^2m_1^{4}m_3^{1+3/2} M_1^{1+3/2}\omega(m_3) \alpha_3^{-3/2} t^2e^{-\nu_0 t}\underset{s\in[0,t]}{\sup}\left[ e^{\nu_0 s}\lVert  f_s\rVert_{{\HH}}\right]\\
 && + C_k^2D^2 2^{1/2}m_1^{4}m_4^{1+3/2} M_1^{1+3/2}\omega(m_4) \alpha_4^{-3/2} t^2e^{-\nu_0 t}\underset{s\in[0,t]}{\sup}\left[ e^{\nu_0 s}\lVert  f_s\rVert_{{\HH}}\right] 
  \eean
  for every $\nu_2\in (0, \nu_0)$.

\medskip\noindent
\emph{Step 2. (Control of the specular term $\IIII^S$)} We recall now $\IIII^S$ defined in \eqref{eq:IIII_4_S} and we recall that $x_1\in \Lambda_3^\eps$. Using then the fact that $\lvert \VV_{x_1} v_*\rvert  = \lvert v_*\rvert$, $\lvert n(x_1) \cdot \VV_{x_1} v_*\rvert = \lvert n(x_1)\cdot v_*\rvert$,
and that we integrate on the velocity set $\{ \vert v_*\rvert \leq m_1, \, \lvert n(x_1)\cdot v_* \rvert\geq \alpha_0\}$ we deduce that it is equivalent to integrate on the set $\{\lvert  \VV_{x_1} v_*\rvert \leq m_1, \lvert n(x_1)\cdot  \VV_{x_1} v_*\rvert \geq \alpha_0 \}  \cap \{\VV_{x_1} v_* \notin S_{x_1} \} \cap \{ \VV_{x_1} v_* \notin W_{t_1, x_1} \}$,
by using that the sets $S_{x_1}$ and $W_{t_1, x_1} \cap \{\lv v_*\rv\leq m_1\}\setminus S_{x_1}$ have Lebesgue measure zero, due to Lemmas \ref{lem:ZeroMeasureSetCylinderSx} and \ref{lem:ZeroMeasureSetCylinder}. We denote $(s_2, x_2, v_2)$ the first bounce against the boundary through the backwards trajectory starting at $(s_1,x_1, \VV_{x_1} v_*)$ given by \eqref{eq:BackwardsTrajectory} and we remark that 
$$
x_2 = x_1- \VV_{x_1} v_*(s_1- s_2) \notin \SSSS \quad \text{ and } \quad ( x_2, \VV_{x_1} v_*) \in \Sigma_+^\eps.
$$
We define then the sets 
\bean
\AAA_S &:=& \{v^*\in \R^3, \,  \vert v_*\rvert \leq m_1, \, \lvert n(x_s)\cdot v_* \rvert\geq \alpha_0 \text{ and } x_2  \in \Lambda_3^\eps \},\\
\AAA_D &:=& \{v^*\in \R^3, \,  \vert v_*\rvert \leq m_1, \, \lvert n(x_s)\cdot v_* \rvert\geq \alpha_0 \text{ and } x_2 \in \Lambda_1^\eps\cup \Lambda_2^\eps \},
\eean
so we have that $\IIII^S \leq  \IIII^S_S + \IIII^S_D$, where
\bean
\IIII^S_S &=& \int_0^t e^{-\nu(v)(t-s)}\int_{\AAA_S} k_{0, m_1} (v,v_*)  e^{-\nu(v)(s-s_1)}   \omega_1 (v_*) (1-\iota^\eps(x_1)) \left\lvert f(s_1, x_1, \VV_{x_1}v_*) \right\rvert  \d v_*\ds ,\\
\IIII^S_D &=& \int_0^t e^{-\nu(v)(t-s)}\int_{\AAA_D} k_{0, m_1} (v,v_*)  e^{-\nu(v)(s-s_1)}  \omega_1 (v_*) (1-\iota^\eps(x_1))  \left\lvert  f(s_1, x_1, \VV_{x_1}v_*) \right\rvert  \d v_*\ds ,
\eean
and we study each one separately. 

By repeating the analysis performed on the Step 2 of the proof of the Proposition~\ref{prop:Kregularization}, using Lemma~\ref{lemma3.2:GeneralReflCylinder} instead of Lemma~\ref{lemma3.2:GeneralRefl}, we deduce that there is $\eps_6^4>0$ such that for every $\eps\in  (0, \min(\eps_6^1,\eps_6^2,\eps_6^3, \eps_6^4))$ and every $N_6,\alpha_6>0$ there is $m_6(N_6)>0$ such that there holds
\bean
\IIII^S_S &\leq& \left( C_k Dm_1^{4} { t \over N_6} + C_k^2D^2m_1^{4}m_6^{4}\alpha_6 t \right) e^{-\nu_0 t}\underset{s\in[0, t]}{\sup}\left[ e^{\nu_0s}\lVert  f_s\rVert_{L^{\infty}_{\omega_1}(\OO^\eps)}\right] \\
&&+C_k D m_1^{4} t e^{-\nu_0 t} \lVert  f_0\rVert_{L^{\infty}_{\omega_1}(\OO^\eps)} + C_k D m_1^{4}{\nu_1t \over \nu_0 - \nu_2}   e^{-\nu_2 t} \underset{s\in [0,t]}{\sup}\left[e^{\nu_2 s}\lVert  G_s\rVert_{L^{\infty}_{\omega_1\nu^{-1}}(\OO^\eps)}\right] \\
&& + 2^{1/2}\omega_1(m_6) D C_k^2 m_1^{1+3/2}m_6^{1+3/2} \alpha_6^{-3/2} t^2e^{-\nu_0 t}\underset{s\in[0,t]}{\sup}\left[ e^{\nu_0 s}\lVert  f_s\rVert_{{\HH}}\right] ,
\eean
for every $\nu_2\in (0, \nu_0)$. On the other hand, to control $\IIII^S_D$ we use again the Duhamel formula and we have that $\IIII_D^S \leq \IIII^S_{D,1} + \IIII^S_{D,2} + \IIII^S_{D,3} + \IIII^S_{D,4}$ with 
\bean
\IIII^S_{D,1} &=& \int_0^t \int_{\AAA_D} k_{0, m_1} (v,v_*)  e^{-\nu_0(t-s_1)}  \omega_1(v_*)  \left\lvert  S_{\TTT}(s_1)  f_0(x_1, \VV_{x_1}v_*)  \right\rvert  \d v_*\ds ,\\
\IIII^S_{D,2} &=& \int_0^t \int_{\AAA_D} k_{0, m_1} (v,v_*)  e^{-\nu_0(t-s_1)}  \omega_1(v_*)  \left\lvert  \int_{\max\{0,s_2\}}^{s_1}S_{\TTT}(s_1-r) K f (r,x_1,\VV_{x_1}v_*) \dr   \right\rvert  \d v_*\ds ,\\
\IIII^S_{D,3} &=& \int_0^t \int_{\AAA_D} k_{0, m_1} (v,v_*)  e^{-\nu_0(t-s_1)}  \omega_1(v_*)  \left\lvert  \int_{\max\{0,s_2\}}^{s_1}S_{\TTT}(s_1-r)  G(r, x_1, \VV_{x_1}v_*) \dr    \right\rvert  \d v_*\ds ,\\
\IIII^S_{D,4} &=& \int_0^t \int_{\AAA_D} k_{0, m_1} (v,v_*)  e^{-\nu_0(t-s_1)}  \omega_1(v_*)  \left\lvert  S_\TTT (s_1-s_2) f(s_2, x_2, \VV_{x_1}v_*)  \right\rvert  \d v_*\ds .
\eean  
We observe then that the first three terms can be controlled by repeating the same computations as during the Step 1 of the proof of the Proposition~\ref{prop:Kregularization}, using Lemma~\ref{lemma3.2:GeneralReflCylinder} instead of Lemma~\ref{lemma3.2:GeneralRefl}. Therefore, there is $\eps_6^5>0$ such that for every $\eps\in(0, \min(\eps_6^1,\eps_6^2,\eps_6^3, \eps_6^4, \eps_6^5))$ and every $N_6,\alpha_6>0$ there is $m_6(N_6)>0$ such that there holds
\bean
\IIII^S_D &\leq& \left( C_k Dm_1^{4} { t \over N_6} + C_k^2D^2m_1^{4}m_6^{4}\alpha_6 t \right) e^{-\nu_0 t}\underset{s\in[0, t]}{\sup}\left[ e^{\nu_0s}\lVert  f_s\rVert_{L^{\infty}_{\omega_1}(\OO^\eps)}\right] \\
&&+C_k D m_1^{4} t e^{-\nu_0 t} \lVert  f_0\rVert_{L^{\infty}_{\omega_1}(\OO^\eps)} + C_k D m_1^{4}{\nu_1t \over \nu_0 - \nu_2}   e^{-\nu_2 t} \underset{s\in [0,t]}{\sup}\left[e^{\nu_2 s}\lVert  G_s\rVert_{L^{\infty}_{\omega_1\nu^{-1}}(\OO^\eps)}\right]  \\
&& + 2^{3/2}\omega_1(m_6) D C_k^2 m_1^{1+3/2}m_6^{1+3/2} \alpha_6^{-3/2} t^2e^{-\nu_0 t}\underset{s\in[0,t]}{\sup}\left[ e^{\nu_0 s}\lVert  f_s\rVert_{{\HH}}\right] + \IIII^{S}_{D, 4}.
\eean
We proceed then to control the remaining term $\IIII^S_{D,4}$, and using the boundary conditions of Equation~\eqref{eq:PLRBE_T} we first have that
$$
\IIII^{S}_{D,4} \leq     \int_0^t \int_{\AAA_D} k_{0, m_1} (v,v_*)  e^{-\nu_0(t-s_1)}   \int_{\R^3} \left\lvert  f(s_2, x_2, u')\right\rvert (n(x_2)\cdot u')_+\,  \d u' \d v_*\ds.
$$
Repeating then exactly the same computations performed during the Step 1 of this proof we deduce that there are $\eps_6^6, \eps_6^7>0$ such that for every $\eps \in (0, \min(\eps_6^1,\eps_6^2,\eps_6^3, \eps_6^4, \eps_4^5, \eps_6^6, \eps_6^7))$, there holds that for every $\lambda_2, \eta_2,  N_7, N_8, \alpha_7, \alpha_8>0$, there are constants $m_6(N_6), m_7(N_7)>0$ such that 
\bean
\IIII^{S}_{D,4}  &\leq & \left(  \lambda_2 C_k D m_1^{4} t  +   {t^2 C_k^2 D^2 m_1^{4} M_1^{4} \over N_7}  + C_k^2D^3m_1^{4} M_1^{4}m_7^{4}\alpha_7 t +  2 \eta_2 \, t M_1^3 D^2 C_k m_1^3\right.\\
&&\left.+   {t^2 C_k^2 D^2 m_1^{4} M_1^{4} \over N_8}  + C_k^2D^3m_1^{4} M_1^{4}m_8^{4}\alpha_8 t  \right)   e^{-\nu_0 t}\underset{s\in[0,t]}{\sup}\left[ e^{\nu_0s}\lVert   f_s\rVert_{L^{\infty}_{\omega_1}(\OO^\eps)}\right] \\
&&+ 2C_k D^2 m_1^{4} M_1^{4} t e^{-\nu_0 t} \lVert   f_0\rVert_{L^{\infty}_{\omega_1}(\OO^\eps)} +  2 {\nu_1 C_k D^2m_1^{4} M_1^{4} t\over \nu_0 - \nu_2}  e^{-\nu_2 t}  \underset{s\in [0,t]}{\sup}\left[e^{\nu_2 s}\lVert   G_s\rVert_{L^{\infty}_{\omega_1\nu^{-1}}(\OO^\eps)}\right] \\
&&+ \left(\sum_{j=7}^8 \omega_1(m_j) m_j^{1+3/2} \alpha_j^{-3/2}\right)  2^{1/2} D^2 M_1^{1+3/2}C_k^2 m_1^{1+3/2}  t^2e^{-\nu_0 t}\underset{s\in[0,t]}{\sup}\left[ e^{\nu_0 s}\lVert  f_s\rVert_{{\HH}}\right] ,
  \eean
for every $\nu_2\in (0, \nu_0)$. We then conclude this step by putting together the previous informations and we have obtained that
\begin{multline*}
\IIII^S \leq C^S_1 e^{-\nu_0 t}\underset{s\in[0, t]}{\sup}\left[ e^{\nu_0s}\lVert  f_s\rVert_{L^{\infty}_{\omega_1}(\OO^\eps)}\right]  + 2C_k D m_1^{4}( 1+ DM_1^4){\nu_1t \over \nu_0 - \nu_2}   e^{-\nu_2 t} \underset{s\in [0,t]}{\sup}\left[e^{\nu_2 s}\lVert  G_s\rVert_{L^{\infty}_{\omega_1\nu^{-1}}(\OO^\eps)}\right] 
\\
 +2C_k D m_1^{4}  (1+ D M_1^{4}) t e^{-\nu_0 t} \lVert  f_0\rVert_{L^{\infty}_{\omega_1}(\OO^\eps)} 
 + C^S_2  t^2e^{-\nu_0 t}\underset{s\in[0,t]}{\sup}\left[ e^{\nu_0 s}\lVert  f_s\rVert_{{\HH}}\right]  ,
\end{multline*}
with
\begin{multline*}
C^S_1 = C_k Dm_1^{4} { t \over N_6} + C_k^2D^2m_1^{4}m_6^{4}\alpha_6 t + C_k Dm_1^{4} { t \over N_6} + C_k^2D^2m_1^{4}m_6^{4}\alpha_6 t +  \lambda_2 C_k D m_1^{4} t \\
   +   {t^2 C_k^2 D^2 m_1^{4} M_1^{4} \over N_7}  + C_k^2D^3m_1^{4} M_1^{4}m_7^{4}\alpha_7 t +  2 \eta_2 \, t M_1^3 D^2 C_k m_1^3+   {t^2 C_k^2 D^2 m_1^{4} M_1^{4} \over N_8}  + C_k^2D^3m_1^{4} M_1^{4}m_8^{4}\alpha_8 t  ,
   \end{multline*}
   and
  $$
C^S_2 =  2^{1/2} D C_k^2 m_1^{1+3/2}\left[ \left(\sum_{j=5}^6 \omega_1(m_j) m_j^{1+3/2} \alpha_j^{-3/2}\right)   + \left(\sum_{j=7}^8 \omega_1(m_j) m_j^{1+3/2} \alpha_j^{-3/2}\right)  D M_1^{1+3/2}\right].
$$

\medskip\noindent
\emph{Step 3. (Choice of the parameters)} Putting together the estimates obtained from Steps 1, 2 we get that for every $\eps \in (0,  \min(\eps_6^1,\eps_6^2,\eps_6^3, \eps_6^4, \eps_4^5, \eps_6^6, \eps_6^7))$ and $\nu_2\in (0, \nu_0)$ there holds
\begin{multline*}
\IIII\leq  C_1\left(t+t^2\right) e^{-\nu_0 t}\underset{s\in[0,t]}{\sup}\left[ e^{\nu_0s}\lVert  f_s\rVert_{L^{\infty}_{\omega_1}(\OO^\eps)}\right] + C_2 t e^{-\nu_0t} \lVert  f_0\rVert_{L^{\infty}_{\omega_1}(\OO^\eps)} \\
 + C_3t^2e^{-\nu_0 t}\underset{s\in[0,t]}{\sup}\left[ e^{\nu_0 s}\lVert  f_s\rVert_{{\HH}}\right] + C_4 t e^{-\nu_2 t} \underset{s\in [0,t]}{\sup}\left[e^{\nu_2 s}\lVert   G_s\rVert_{L^{\infty}_{\omega_1\nu^{-1}}(\OO^\eps)}\right]  ,
\end{multline*}
with
\bean 
 C_1 &=& {1\over N_1}  + 2C_kDm_1^{4}\alpha_0+  {C_kD m_1^{4}\over N_2}+ C_k^2D^2m_1^{4}m_2^{4} \alpha_1  + \lambda_1 C_k D m_1^{4} t  +   {t^2 C_k^2 D^2 m_1^{4} M_1^{4} \over N_3}  \\
&& + C_k^2D^3m_1^{4} M_1^{4}m_3^{4}\alpha_3 t +   {t^2 C_k^2 D^2 m_1^{4} M_1^{4} \over N_4}  + C_k^2D^3m_1^{4} M_1^{4}m_4^{4}\alpha_4 t   + C^S_1,
\eean
and 
\bean
C_2 &=& 3C_k D m_1^{4}  + 2C_k D m_1^{4}  (1+ D M_1^{4}), \qquad C_3 = C_3^0 +C_2^S,\\
C_3^0 &=& 2^{1/2} D C_k^2 m_1^{1+3/2}\left[ \left(\sum_{j=1}^2 \omega_1(m_j) m_j^{1+3/2} \alpha_j^{-3/2}\right) + D M_1^{1+3/2} \left(\sum_{j=3}^4 \omega_1(m_j) m_j^{1+3/2} \alpha_j^{-3/2}\right) \right]  , \\
C_4 &=&  {\nu_1 \over \nu_0 - \nu_2}\left( 3C_kDm_1^{4} + 4C_k D^2m_1^{4} M_1^{4}  \right).
\eean
We then set the constants in a similar way as in the Step 4 of the proof of the Proposition~\ref{prop:Kregularization}  so that $C_1\leq \lambda$. We define $C = \max(C_2, C_3, C_4)$ and we observe that, since the constants that define $C$ come from \ref{item:K1}, \ref{item:K4} and \ref{item:K5}, there is a constant $p>0$ such that $ C\lesssim \lambda^{-p}$. We conclude by setting $\eps_6 = \min(\eps_6^1,\eps_6^2,\eps_6^3, \eps_6^4, \eps_6^5, \eps_6^6, \eps_6^7)$.
\end{proof}

\subsection{Estimate on the trajectories}\label{ssec:TrajectoryLinftyCylinder} In this subsection we  use then the regularization property given by Proposition~\ref{prop:KregularizationCylinder} to proof a $L^\infty$ estimate of the solutions of the equation following a similar argument as during the proof of Proposition~\ref{prop:GainLinftyL2LB1}.

\begin{prop} \label{prop:GainLinftyL2LB2}
Consider Assumption \ref{item:RH2} to hold, $\omega_1\in \WWWW_1$ a strongly confining admissible weight function, and let $f$ be a solution to Equation~\eqref{eq:PLRBE_T}. For every $t\in [0, T]$ and every $\lambda>0$ there is $\eps_7= \eps_7( \lambda, T)>0$ such that for every $\eps\in (0, \eps_7)$ and every $\nu_2\in (0, \nu_0)$, we have that for every $(x, v)\in  \OO^\eps$, with $\lvert v\rvert \leq \MMMM$ for any arbitrary $\MMMM>0$, and $v\notin S_x\cup W_{t, x}$ there holds
\begin{multline}\label{eq:GainLinftyL2LB2}
\omega_1(v) \lvert f  (t, x,v) \rvert \leq  \lambda (1+ t + t^2) e^{-\nu_0 t}\underset{s\in[0,t]}{\sup}\left[ e^{\nu_0s}\lVert  f_s\rVert_{L^{\infty}_{\omega_1}(\OO^\eps)}\right] 
+ C(1+t) e^{-\nu_0 t} \lVert f_0\rVert_{L^{\infty}_{\omega_1}(\OO^\eps)} \\
  + C t^2 e^{-\nu_0 t}\underset{s\in[0,t]}{\sup}\left[ e^{\nu_0 s}\lVert  f_s\rVert_{\HH}\right]   + C(1+t) e^{-\nu_2t}\underset{s\in [0,t]}{\sup}\left[e^{\nu_2 s}\lVert G_s\rVert_{L^{\infty}_{\omega_1\nu^{-1}}(\OO^\eps)}\right] ,
\end{multline}
uniformly in $\MMMM$. Moreover, $C= C(\lambda)>0$ and there is $p>0$ such that $C \lesssim \lambda^{-p}$.
\end{prop}

\begin{proof}
Starting from $(t,x,v)$, and since $\lv v\rv\leq \MMMM$ then Lemma~\ref{lem:CircunferenceSpecularNNNN} implies that there are only two possible scenarios to consider:

\noindent
\hspace{0.2cm} $\bullet$ \emph{Case 1.} There are only specular reflections through the backwards trajectories. In particular, there is $\NNNN= \NNNN(T,x,\MMMM) <\infty$, given by Lemma~\ref{lem:CircunferenceSpecularNNNN}, such that there are no more than $\NNNN$ consecutive specular reflections through the backwards trajectories. 

\noindent
\hspace{0.2cm} $\bullet$ \emph{Case 2.} There is a reflection against the diffusive boundary subsets before $\NNNN$ specular reflections.

We treat now each case separately.

\medskip\noindent
\emph{Case 1. (Purely specular reflection)} We recall that there are at most $\NNNN $ collisions against the boundary through the specular reflection boundary condition starting at $(t,x,v)$, where $\NNNN$ is given by Lemma~\ref{lem:CircunferenceSpecularNNNN}. Then the iterated Duhamel formula gives
\begin{multline*}
f(t,x,v) =  e^{-\nu(v)(t-t_\NNNN)}  f(t_\NNNN, x_\NNNN, v_\NNNN) + \int_{0}^{t} \sum_{j=1}^{\NNNN}  \Ind_{t_j\leq s \leq t_{j-1}} (s) S_{\TTT}(t_{j-1}-s) K f (s, x_{j-1},v_{j-1}) \ds \\
+\int_{0}^{t} \sum_{j=1}^{\NNNN}  \Ind_{t_j\leq s \leq t_{j-1}} (s) S_{\TTT}(t_{j-1}-s)G (s, x_{j-1},v_{j-1}) \ds,
\end{multline*}
where we have set $(t_0,x_0,v_0) = (t,x,v)$ and we have defined $(t_j,x_j,v_j)$ as in \eqref{eq:BackwardsSpecTrajectory} for $j\in \llbracket 1,\NNNN\rrbracket$. 

We then define $(\bar t, \bar x, \bar v)$ as the point of collision with the boundary set $\Lambda_3^\eps$ through the backwards trajectory starting at $(t_\NNNN, x_\NNNN, v_\NNNN)$. From the fact that there have already been $\NNNN$ specular reflections we remark that $\bar t<0$, and using one more time the Duhamel formula we obtain
\bean
f(t,x,v) &=&  e^{-\nu(v)(t-t_\NNNN)} S_\TTT (t_\NNNN) f_0(x_\NNNN, v_\NNNN) +  S_{\TTT}* K f (t_\NNNN, x_\NNNN,v_\NNNN)  + S_{\TTT} * G (t_\NNNN, x_{\NNNN},v_{\NNNN}) \\
&&+ \int_{0}^{t} \sum_{j=1}^{\NNNN+1}  \Ind_{t_j\leq s \leq t_{j-1}} (s) S_{\TTT}(t_{j-1}-s) K f (s, x_{j-1},v_{j-1}) ds \\
&&+\int_{0}^{t} \sum_{j=1}^{\NNNN}  \Ind_{t_j\leq s \leq t_{j-1}} (s) S_{\TTT}(t_{j-1}-s)G (s, x_{j-1},v_{j-1}) ds ,
\eean
where we have denoted $t_{\NNNN+1} = 0$. 
Using now the fact that \ref{item:K3} and Proposition~\ref{prop:KregularizationCylinder} are uniform in $(t,x,v)\in \UU^\eps_T$, $v\notin S_x \cup W_{t,x}$, and the fact that  
\beqn
 \int_{0}^{t} \sum_{j=1}^{\NNNN+1}  \Ind_{t_j\leq s \leq t_{j-1}} (s) \, \ds = t,
\eeqn
we deduce that, by arguing as during the proof of Proposition~\ref{prop:GainLinftyL2LB1}, there is $\eps_7^1 = \eps_6(\lambda, T)$, where $\eps_6(\lambda, T)$ is given by Proposition~\ref{prop:KregularizationCylinder}, such that for every $\eps \in (0, \eps_7^1)$ there holds
\begin{multline*}
\omega_1(v) \lvert f  (t, x,v) \rvert \leq \lambda (t + t^2) e^{-\nu_0 t}\underset{s\in[0,t]}{\sup}\left[ e^{\nu_0s}\lVert  f_s\rVert_{L^{\infty}_{\omega_1}(\OO^\eps)}\right] + C_S(1+t) e^{-\nu_0 t} \lVert f_0\rVert_{L^{\infty}_{\omega_1}(\OO^\eps)}  \\
  + C_S t^2 e^{-\nu_0 t}\underset{s\in[0,t]}{\sup}\left[ e^{\nu_0 s}\lVert  f_s\rVert_{\HH}\right] + C_S(1+t) e^{-\nu_2t}\underset{s\in [0,t]}{\sup}\left[e^{\nu_2 s}\lVert G_s\rVert_{L^{\infty}_{\omega_1\nu^{-1}}(\OO^\eps)}\right] ,
\end{multline*}
for some constant $C_S\lesssim \lambda^{-p}$ with $p>0$.
We conclude this case by emphasizing that $\eps_7^1$ and the constant $C_S$, are independent of $\NNNN$ and $\MMMM$ due to the above arguments.

\medskip\noindent
\emph{Case 2. (Possible diffussive reflection)}
We assume without loss of generality that the diffusive collision against the diffusive boundary through the backwards trajectory happens at the first collision. Otherwise, we just repeat the Duhamel formulation through the specular reflections like performed on the Step 1, until the diffusive reflection happens and then we proceed as follows. 

We denote $(t_1, x_1, v_1)$ the first collision through the backwards trajectory as given by \eqref{eq:BackwardsTrajectory}, and we remark that $x_1 = x-v(t- t_1)$. Using then the Duhamel formula we obtain
\bean
    \lvert \omega_1(v) f  (t, x,v) \rvert &\leq& \omega_1(v) S_{\TTT}(t) \left(   \lv f_0 (x,v)\rv \right)   + \omega_1 (v) \lvert S_{\TTT}*  K f (t,x,v)\rvert  + \omega_1(v) \lvert S_{\TTT}* G (t,x,v)\rvert \\
    &&+ e^{-\nu (v) (t-t_1)} \omega_1 (v) \lvert  f(t_1, x_1, v)\rvert\\
    &=:& \III_1 + \III_2 + \III_3 + \III_4,
\eean
and we proceed as during the proof of Proposition~\ref{prop:GainLinftyL2LB1}. On the one hand, using \ref{item:K3} and Proposition~\ref{prop:KregularizationCylinder} we have that for every $\lambda_1>0$ there is $\eps_7^2 = \eps_6(\lambda_1, T)>0$ and a constant $C_1=C_1( \lambda_1)>0$ such that for every $\eps \in(0,  \eps_7^2)$ and every $\nu_2\in (0, \nu_0)$ there holds
\begin{multline*}
\III_1 + \III_2 + \III_3 \leq  \lambda_1 (t+t^2)  e^{-\nu_0t}\underset{s\in[0,t]}{\sup}\left[ e^{\nu_0s}\lVert f_s\rVert_{L^{\infty}_{\omega_1}(\OO^\eps)}\right] +  (1+C_1t) e^{-\nu_0t} \lVert f_0\rVert_{L^{\infty}_{\omega_1}(\OO^\eps)} \\
 +  C_1 t^2 e^{-\nu_0 t}\underset{s\in[0,t]}{\sup}\left[ e^{\nu_0 s}\lVert f_s\rVert_{\HH}\right] + (1+C_1)(1+t) e^{-\nu_2 t} \underset{s\in [0,t]}{\sup}\left[e^{\nu_2 s}\lVert G_s\rVert_{L^{\infty}_{\omega_1\nu^{-1}}(\OO^\eps)}\right],
\end{multline*}
furthermore, there are $c_1, p_1>0$ such that $C_1\leq  c_1 \lambda_1^{-p_1}$.

We then control the boundary term $\III_4$, and since we have assumed the reflection at the boundary to be diffusive, then using \ref{item:K5} we have that for every $\lambda_2>0$ there are $M, \eta>0$ such that 
\beqn
\III_4 \leq e^{-\nu (v) (t-t_1)} \int_{\R^3} \lvert f(t_1, x_1, u) \rvert (n_{x_1} \cdot u)_+ du   \leq \lambda_2  e^{-\nu_0 t}\underset{s\in[0,t]}{\sup}\left[ e^{\nu_0 s}\lVert f_s\rVert_{L^\infty_{\omega_1}(\OO^\eps)}\right]  + \III_0^D,
\eeqn
where we have used that $ \omega_1\MMM\leq 1$, and we have defined
$$
\III_0^D := e^{-\nu(v)(t-t_1)} \int_{\{\lvert u\rvert \leq M,\, \lvert n(x_1)\cdot u\rvert >\eta\}}   \lvert  f(t_1, x_1, u)\rvert  \, (n(x_1)\cdot u)_+ \d u.
$$
By arguing exactly as during the Step 1 of the proof Proposition~\ref{prop:KregularizationCylinder} controlling the term $\IIII_0^D$ defined in \eqref{eq:IIIID0def} we deduce that for any $\eta_1$ there is $A_1 = A_1(\eta_1)>0$, such that there is $\eps_7^4 = \eps_S(A_1, M,T)$ given by Lemma~\ref{lemma3.2:GeneralReflCylinder} such that for every $\eps\in (0,\min(\eps_7^2, \eps_7^3, \eps_7^4))$ there holds that for every $\lambda_3, \lambda_4>0$ there are $\eps_7^5 = \eps_6(\lambda_3, T)>0$, $\eps_7^6 = \eps_6(\lambda_4, T)>0$ and constants $C_3=C_3( \lambda_3)>0$, $C_4=C_3( \lambda_4)>0$, given by the application of Proposition~\ref{prop:KregularizationCylinder}, such that 
\begin{multline*}
\omega_1\lvert f  (t, x,v) \rvert \leq \left(  \lambda_1 + \lambda_2 + 2\eta_1 M^3D + \lambda_3 M^4D + \lambda_4 M^4D \right) (t + t^2) e^{-\nu_0 t}\underset{s\in(0,t]}{\sup}\left[ e^{\nu_0s}\lVert  f_s\rVert_{L^{\infty}_{\omega_1}(\OO^\eps)}\right] \\
+ \left(  1+ C_1+M^4D+C_3M^4D + M^4D+ C_4M^4D     \right)(1+t) e^{-\nu_0 t} \lVert f_0\rVert_{L^{\infty}_{\omega_1}(\OO^\eps)}  \\
  + \left(   C_1 + C_3M^4D  +C_4M^4D   \right) t^2 e^{-\nu_0 t}\underset{s\in[0,t]}{\sup}\left[ e^{\nu_0 s}\lVert  f_s\rVert_{\HH}\right] \\
  + \left(  1+C_1+M^4D+ C_3M^4D+M^4D + C_4M^4D            \right)(1+t) {\nu_1\over \nu_0-\nu_2} e^{-\nu_2t} \underset{s\in [0,t]}{\sup}\left[e^{\nu_2 s}\lVert G_s\rVert_{L^{\infty}_{\omega_1\nu^{-1}}(\OO^\eps)}\right]  ,
\end{multline*}
furthermore, there are $c_3, c_4, p_3, p_4>0$ such that $C_3\leq  c_3 \lambda_4^{-p_3}$ and $C_4\leq  c_4 \lambda_4^{-p_4}$.  
We take then $\lambda_1 = \lambda_2 = \lambda/5$, $\eta_1 = \lambda /(10M^3D)$ and $\lambda_3 = \lambda_4 = \lambda/(5M^4D)$, and we define 
$$
C_D =  1+ C_1+M^4D+C_3M^4D + M^4D+ C_4M^4D,
$$
thus there is $p'>0$ such that $C_D \lesssim  \lambda^{-p'}$. 

We set $\eps_7= \min(\eps_7^1, \eps_7^2, \eps_7^3, \eps_7^4, \eps_7^5, \eps_7^6)$ and we conclude the proof by putting together the estimates from both of the possible scenarios.
\end{proof}

\subsection{Weighted $L^\infty$ control for solutions of Equation~\eqref{eq:PLRBE_T}}
In this subsection we use the estimate obtained in Proposition~\ref{prop:GainLinftyL2LB2} to deduce a weighted $L^\infty$ control on the solutions of Equation~\eqref{eq:PLRBE_T}.
\begin{prop}\label{prop:LinftyEstimatePerturbedFiniteT_0Cylinder} 
Consider Assumption \ref{item:RH2} to hold, $\omega_1\in\WWWW_1$ a strongly confining admissible weight function, and let $f$ be a solution of Equation~\eqref{eq:PLRBE_T}. There is $\eps_8 = \eps_8 (T)>0$ such that for every $\eps\in (0,\eps_9)$ there holds 
\begin{multline*}
\lVert  f_t  \rVert_{L^{\infty}_{\omega_1}(\bar \OO^\eps)} \leq   C(1+T) (1+T+T^2)^pe^{-\nu_0 t} \lVert f_0\rVert_{L^{\infty}_{\omega_1}(\OO^\eps)} 
+  C (1+T+T^2)^p  T ^2 e^{-\nu_0  t }\underset{s\in[0, t]}{\sup}\left[ e^{\nu_0 s}\lVert  f_s\rVert_{\HH} \right]  \\
+  C (1+T+T^2)^p (1+T) e^{-\nu_2t}\underset{s\in [0,t]}{\sup}\left[e^{\nu_2 s}\lVert G_s\rVert_{L^{\infty}_{\omega_1\nu^{-1}}(\OO^\eps)}\right] ,
\end{multline*}
for every $t\in [0, T]$, and some universal constants $C, p>0$.
\end{prop}

\begin{proof}
We follow the proof of Proposition~\ref{prop:LinftyEstimatePerturbedFiniteT_0}. Let $\MMMM>0$ and consider a point 
$$
( t , x, v)\in (0,T)\times  \Omega ^\eps\times \{v\in \R^3,\, \lvert v\rvert \leq \MMMM\}
$$ 
such that $v\notin S_{x}\cup W_{t,x}$. Applying then Proposition~\ref{prop:GainLinftyL2LB2} we have that for every $\lambda>0$ there is $\eps_8 =\eps_7( \lambda, T)$ such that for every $\eps \in(0, \eps_8)$  and every $\nu_2\in (0, \nu_0)$ there holds \eqref{eq:GainLinftyL2LB2}
for some constant $C>0$ such that $C \lesssim \lambda^{-p}$ for some $p>0$. 

Using now Lemmas \ref{lem:ZeroMeasureSetCylinderSx} and \ref{lem:ZeroMeasureSetCylinder} we have that the sets $S_{x}$ and $W_{t,x}$ have Lebesgue measure zero, moreover we recall that the grazing set $\Sigma_0^\eps$ also has Lebesgue measure zero. Therefore we may take the $L^{\infty}_v(B_\MMMM)$ norm in the previous inequality followed by the supremum in $\MMMM$ due to the fact that $\eps_8$ and the constants given by Proposition~\ref{prop:GainLinftyL2LB2} do not depend on $\MMMM$, and finally taking the $L^{\infty}_x({\bar \Omega^\eps})$ 
We conclude by choosing $\lambda =\displaystyle (2(1+t+t^2))^{-1}$ and following the ideas from the proof of Proposition~\ref{prop:LinftyEstimatePerturbedFiniteT_0}. 
\end{proof}

\subsection{Proof of Proposition~\ref{prop:LinftyEstimatePerturbedFiniteTCylinder}} 
The proof follows exactly the proof of Proposition~\ref{prop:LinftyEstimatePerturbedFiniteT} by using Proposition~\ref{prop:LinftyEstimatePerturbedFiniteT_0Cylinder} in the place of Proposition~\ref{prop:LinftyEstimatePerturbedFiniteT_0}. \qed

\section{A priori estimates for weakly confining weights}\label{sec:AprioriWeakConf}

We consider the function $G:\UU^\eps\to \R$, and during this section we study the following evolution equation
 \be
	\left\{\begin{array}{llll}
		\partial_{t} f_1 &=& \TTT f_1 + \AA_\delta f_1 + G &\text{ in }\UU^\eps\\
		\gamma_- f_1&=&\RRR \gamma_+f_1 &\text{ on }\Gamma_{-}^\eps\\
		 f_{1, t=0}&=& f _0 &\text{ in }\OO^\eps,
	\end{array}\right.\label{eq:PLRBE1}
\ee
where we recall that the free transport operator $\TTT$ is given by \eqref{eq:DefTransportOp} and $\AA_\delta$ is defined in \eqref{eq:defSplittingLL}. We dedicate this section to prove the following result.

\begin{prop}\label{prop:LinftyEstimatePerturbedFiniteT_Summary_WWWW0} 
Consider either Assumption \ref{item:RH1} or \ref{item:RH2} to hold, $\omega_0\in\WWWW_0$ a weakly confining admissible weight function, and let $f_1$ be a solution of Equation~\eqref{eq:PLRBE1}. There are constructive constants $\eps_9, \delta_0>0$ such that for every $\eps\in (0, \eps_9)$ and every $\delta \in (0,\delta_0)$ there holds
\be
\lVert  f_{1, t} \rVert_{L^{\infty}_{\omega_0}(\bar\OO^\eps)} \leq  C e^{-{\nu_0\over 2} t} \left(   \lVert f_0\rVert_{L^{\infty}_{\omega_0}(\OO^\eps)} + \underset{s\in[0,t]}{\sup}\left[ e^{{\nu_0 \over 2} s} \lVert G_s \rVert_{L^{\infty}_{\omega_0\nu^{-1}}(\OO^\eps)}\right]\right)  \qquad \forall t\geq 0, \label{eq:LinftyPerturbedFiniteTimeDecay_Summary_WWWW0}
\ee
for some universal constant $C>0$, independent of $\eps$.
\end{prop}

\begin{rem}
The proof of Proposition~\ref{prop:LinftyEstimatePerturbedFiniteT_Summary_WWWW0} follows the arguments developed during Sections \ref{sec:Linfty} and \ref{sec:LinftyCylinder}. We pay particular attention to polynomial weights in order to determine the validity of the lower bounds $q_\iota^\star$ on their degree, as defined in Remark~\ref{def:nu*}.  
\end{rem}

%
%

\subsection{Weighted $L^\infty$ estimate in smooth domains}\label{ssec:WeakWeightsRH1}
We prove the following long-time behavior result under Assumption \ref{item:RH1}.

\begin{prop}\label{prop:DissipationLinftyL1}
We consider Assumption \ref{item:RH1} to hold, $\omega_0 \in \WWWW_0$ a weakly confining admissible weight function, and let $f_1$ be a solution to Equation~\eqref{eq:PLRBE1}. There are constructive constants $\eps_{10}>0$ and $\delta_1>0$ such that for every $\eps \in (0, \eps_{10})$ and every $\delta \in (0,\delta_1)$ there holds
\be\label{eq:DissipationLinftyL1}
 \lVert  f_{1,t}\rVert_{L^{\infty}_{\omega_0}\left( {\bar \OO^\eps}\right)}   \leq  e^{-{\nu_0\over 2} t}  \left( \lVert f_0\rVert_{L^\infty_{\omega_0}(\OO^\eps)}    +  C \underset{s\in [0,t]}{\sup}\left[ e^{{\nu_0\over 2}s} \lVert  G_{s}\rVert_{L^{\infty}_{\omega_0\nu^{-1}}\left( \OO^\eps\right)} \right]  \right) \qquad \forall t \geq 0,
\ee
    for some constant $C>0$, independent $\eps$.
\end{prop}

\begin{proof}
The proof follows the main ideas of  Proposition~\ref{prop:GainLinftyL2LB1}, Proposition~\ref{prop:LinftyEstimatePerturbedFiniteT_0} and Proposition~\ref{prop:LinftyEstimatePerturbedFiniteT}, thus we only sketch it. 

We define $\widetilde G = \AA_\delta f_1 + G$, we take an arbitrary $T>0$ to be defined later.

\medskip\noindent
\emph{Step 1.} Let $(t, x, v)\in \UU_T^\eps$ such that $v\notin S_{x}(v)$, we denote $(t_1, x_1, v_1)$ the first collision through the backwards trajectory starting at $(t,x,v)$ as defined in \eqref{eq:BackwardsTrajectory}. 

Using the Duhamel formula, together with \ref{item:K4}, and \ref{item:K5}, we have that for every $\lambda>0$ there are $M, \eta>0$ such that 
\bean
\omega_0    \lvert   f_1  (t, x,v) \rvert &\leq& e^{-\nu_0t} \lVert f_0\rVert_{L^{\infty}_{\omega_0}(\OO^\eps)} +    {\nu_1\over \nu_0 - \nu_2} e^{-\nu_2 t} \underset{s\in [0,t]}{\sup}\left[e^{\nu_2 s}\lVert \widetilde G_s\rVert_{L^{\infty}_{\omega_0\nu^{-1}}(\OO^\eps)}\right] \\
&&+ (1-\iota_0) e^{-\nu_0 t} \underset{s\in [0,t]}{\sup}\left[ e^{\nu_0 s} \lVert  f_{1,s}\rVert_{L^{\infty}_{\omega_0}\left( \bar\OO^\eps \right)} \right] + \CCCC_0 \lambda  e^{-\nu_0 t}\underset{s\in[0,t]}{\sup}\left[ e^{\nu_0 s}\lVert f_{1,s}\rVert_{L^\infty_{\omega_0}(\OO^\eps)}\right] \\
&&+ \CCCC_0 e^{-\nu(v)(t-t_1)} \int_{\lvert u\rvert \leq M,\, \lvert n(x_1)\cdot u\rvert >\eta}  \lvert  f_1(t_1, x_1, u)\rvert  (n(x_1)\cdot u)_+ \d u,
\eean
where we have defined  $\omega_0=\omega_0(v)$, $\sigma =\max (0, t_1)$, and $\CCCC_0>0$ is such that $\omega_0 \MMM \leq \CCCC_0$.

By arguing then as during the Step 2 of Proposition~\ref{prop:GainLinftyL2LB1} we apply Lemma~\ref{lemma3.2:GeneralRefl} and there is $\eps_{10}^1=\eps_R( \eta,  M, T)$ such that for every $\eps \in (0, \eps_{10}^1)$ there is no bounce against the boundary through the backwards trajectory starting at $(t_1, x_1, u)$, and given by \eqref{eq:BackwardsTrajectory}. Therefore the Duhamel formula together with  \ref{item:K3} gives that 
\begin{multline*}
\int_{\lvert u\rvert \leq M,\, \lvert n(x_1)\cdot u\rvert >\eta}  \lvert  f_1(t_1, x_1, u)\rvert  (n(x_1)\cdot u)_+ \d u \leq \lambda  e^{-\nu_0 t}\underset{s\in[0,t]}{\sup}\left[ e^{\nu_0 s}\lVert f_{1,s}\rVert_{L^\infty_\omega(\OO^\eps)}\right] \\+  DM^{4} e^{-\nu_0t} \lVert f_0\rVert_{L^{\infty}_{\omega_0}(\OO^\eps)} 
+  {\nu_1\over \nu_1 - \nu_0} \underset{s\in [0,t]}{\sup}\left[e^{\nu_2 s}\lVert\widetilde  G_s\rVert_{L^{\infty}_{\omega_0\nu^{-1}}(\OO^\eps)}\right]  e^{-\nu_0 (t-t_1)} e^{-\nu_2 t_1}  \int_{\lvert u_*\rvert \leq M} \omega_0^{-1}(u)\lvert  u\rvert \d u,
\end{multline*}
and we will provide a precise control on the last integral.

First, when $\omega_0$ is an admissible polynomial weight function as considered in Subsection \ref{sec:MainResult}, then $\omega_0 (v) = \la v\ra ^q$ and we compute
\bean
\int_{\lvert u_*\rvert \leq M} \omega_0^{-1}(u)\lvert  u\rvert du &= & \int_{\lvert u_*\rvert \leq M} \la u\ra ^{-q} \lvert  u\rvert \d u  = \int_0^{2\pi} \int_0^\pi \int_0^M { r^3 \over (1+ r^2)^{q/2}}  \sin(\beta) \d r \d\beta \d\alpha  \\
&= &  4\pi \int_0^M { r (1+ r^2 -1) \over (1+ r^2)^{q/2}} \d r \leq  4\pi  \int_0^M { r \over (1+ r^2)^{-1+ q/2}} \d r   \\
&\leq & {4\pi \over q-4}  \left[-(1+ r^2)^{2- q/2}  \right]_0^M \leq  {4\pi \over q -4} 
\eean
where we have used the spherical change of variables on the first line and we have used that $q>q^\star_\iota \geq 5$ on the last line. 

Second, when $\omega_0$ is a weakly confining admissible weight function, which is not a polynome, we easily deduce that there is a constant $C_{\omega_0} > 1$ such that 
$$
\int_{\lvert u_*\rvert \leq M} \omega_0^{-1}(u)\lvert  u\rvert du \leq C_{\omega_0} .
$$
Altogether we have obtained that for every $\eps \in (0, \eps_{10}^1)$ there holds 
\bean
  \omega_0 (v) \lvert f_1  (t, x,v) \rvert
 &\leq &   (1-\iota_0+ \CCCC_0 \lambda)  e^{-\nu_0t}\underset{s\in[0,t]}{\sup}\left[ e^{\nu_0s}\lVert f_{1,s}\rVert_{L^{\infty}_{\omega_0}( \bar \OO^\eps)}\right]  +   (1+\CCCC_0 D M^{4})     e^{-\nu_0t} \lVert f_0\rVert_{L^{\infty}_{\omega_0}(\OO^\eps)}   \\
 && + {\nu_1\over \nu_0 - \nu_2}\left(  1   +    \CCCC_0 C_{\omega_0}\right) e^{-\nu_2 t} \underset{s\in [0,t]}{\sup}\left[e^{\nu_2 s}\lVert \widetilde G_s\rVert_{L^{\infty}_{\omega_0\nu^{-1}}(\OO^\eps)}\right] .
\eean
with $C_{\omega_0} = 4\pi /( q-4) $ when $\omega_0(v) =\la v\ra^q$, is a polynomial admissible weight function.
Recalling that $\widetilde G =  \AA_\delta f_1 +G$, choosing $\eps_{10}=\eps_{10}^1$ and $\nu_2 = \nu_0/2$, and observing that, since $S_{x}(v)$, and $\Sigma_0^\eps$ are sets of zero Lebesgue measure as discussed on Proposition~\ref{prop:GainLinftyL2LB1}, 
we can take the $L^\infty( \bar \OO^\eps)$ norm on both sides of the estimate above and we deduce that
\bean
 \lVert f_{1,t}\rVert_{L^{\infty}_{\omega_0}({\OO^\eps})} &\leq & (1-\iota_0 + \CCCC_0\lambda)  e^{-\nu_0t}\underset{s\in[0,t]}{\sup}\left[ e^{\nu_0s}\lVert f_{1,s}\rVert_{L^{\infty}_{\omega_0}( \bar \OO^\eps)}\right]  +  (1+ \CCCC_0 D M^{4})   e^{-\nu_0t} \lVert f_0\rVert_{L^{\infty}_{\omega_0}(\OO^\eps)}   \\
 && + 2 {\nu_1\over \nu_0} \left( 1  +  \CCCC_0 C_{\omega_0}   \right)  e^{-{\nu_0\over 2} t} \underset{s\in [0,t]}{\sup}\left[e^{{\nu_0\over 2} s}\lVert \AA_\delta f_{1,s} \rVert_{L^{\infty}_{\omega_0\nu^{-1}}(\OO^\eps)}\right] \\
 && + 2{\nu_1\over \nu_0}\left(  1   +  \CCCC_0  C_{\omega_0} \right) e^{-{\nu_0\over 2} t} \underset{s\in [0,t]}{\sup}\left[e^{{\nu_0\over 2} s}\lVert G_s\rVert_{L^{\infty}_{\omega_0\nu^{-1}}(\OO^\eps)}\right] ,
\eean
and we proceed to use the dissipative properties of $\AA_\delta$ and to fix the constants for each type of weight function.

\smallskip\noindent
\textbullet \, \textbf{Weak inverse gaussian weights:} If $\omega_0(v) = e^{\zeta \lv v\rv^2}$, with $\zeta \in (0,1/4]$, we recall that $\AA_\delta f_1= K((1-\chi_\delta)f)$ as defined in \eqref{eq:defSplittingLL}, therefore arguing as during the proof of \ref{item:K1} we deduce that for every $N>0$ there is $\delta_1^1 = m(N)^{-1}$, where $m(N)$ is given by the Step 2 of the proof of \ref{item:K1} in Lemma~\ref{lem:Kproperties}, such that for every $\delta \in (0, \delta_1^1)$ there holds
$$
\lVert \AA_\delta f_{1,s} \rVert_{L^{\infty}_{\omega_0}(\OO^\eps)} \leq \frac 1N \lVert f_{1,s} \rVert_{L^{\infty}_{\omega_0}(\OO^\eps)}. 
$$
We then choose $N =  8 {\nu_1} \left( 1  + \CCCC_0  C_{\omega_0}  \right) / ( \nu_0 \iota_0)$, $\lambda = \iota_0/(4\CCCC_0)$,
and we obtain \eqref{eq:DissipationLinftyL1} for weak inverse gaussian weights, in the time interval $[0,T]$. 

\smallskip\noindent
\textbullet \, \textbf{Stretched exponential weights:} If $\omega_0(v) = e^{\zeta\la v\ra^s}$, with $s\in (0,2)$ and $\zeta >0$, we observe that  \cite[Lemma 4.12]{MR3779780} implies that
\be\label{eq:DissipativityVarpi_1}
\lVert \AA_\delta f_{1,s} \rVert_{L^{\infty}_{\omega_0\nu^{-1}}(\OO^\eps)} \leq \varpi _{\delta,\omega_0}  \lVert f_{1,s} \rVert_{L^{\infty}_{\omega_0}(\OO^\eps)} , 
\ee
and, furthermore, that we may choose $\delta_1^2>0$ such that for every $\delta\in (0, \delta_1^2)$ there holds
$$
2 {\nu_1\over \nu_0}  \left( 1  + \CCCC_0  C_{\omega_0}\right)  \varpi _{\delta,\omega_0} \leq \frac {\iota_0} 4.
$$
We obtain then \eqref{eq:DissipationLinftyL1} for stretched exponential weights, in the time interval $[0,T]$, by choosing $\lambda = \iota_0 /(4\CCCC_0)$. 

\smallskip\noindent
\textbullet \, \textbf{Polynomial weights:} If $\omega_0 = \la v\ra^q $, with $q >q_\iota^\star$ as defined in \eqref{def:Poly_degree}, using again  \cite[Lemma 4.12]{MR3779780}, we obtain \eqref{eq:DissipativityVarpi_1} with $\varpi _{\delta,\omega_0} \to  4/( q-1)$ as $\delta\to 0$.
We observe then that since $q>q^\star_\iota$, there holds 
$$
2 {\nu_1\over \nu_0}  \left( 1  +  {4\pi\over q-4} \CCCC_0  \right) {4\over q-1} <\iota_0.
$$
We may choose then $\delta_1^3>0$ small enough such that for every $\delta \in (0, \delta_1^3)$ there holds 
$$
2 {\nu_1\over \nu_0}  \left( 1  +   C_{\omega_0} {4\pi\over q-4}  \right)  \varpi _{\delta,\omega_0} \leq \alpha \iota_0,
$$
for some $\alpha\in (0,1)$. We obtain then \eqref{eq:DissipationLinftyL1} for polynomial weights, in the time interval $[0,T]$, by taking $\lambda = (1-\alpha)\iota_0/2$.

We conclude the proof by taking $T>0$ such that $ C e^{-{\nu_0\over 2} T} \leq 1/2$, choosing $\delta_1 = \delta_1^1$ in the case of weak inverse gaussian weight functions, $\delta_1 = \delta_1^2$ in the case of stretched exponential weight functions, and $\delta_1 = \delta_1^3$ in the case of polynomial weight functions, and extending the result for all time by repeating the analysis from the Step 3 of the proof of Proposition~\ref{prop:LinftyEstimatePerturbedFiniteT} (in a much simpler setting).
\end{proof}

\subsection{Weighted $L^\infty$ estimate in cylindrical domains}\label{ssec:WeakWeightsRH1}
We will now prove a similar long-time behavior result to the one on the previous section under Assumption \ref{item:RH2}.

\begin{prop}\label{prop:DissipationLinftyL2}
Consider Assumption \ref{item:RH2} to hold, $\omega_0 \in \WWWW_0$ a weakly confining admissible weight function, and let $f_1$ be a solution to Equation~\eqref{eq:PLRBE1}. There are constructive constants $\eps_{11} >0 $ and $\delta_2>0$ such that for every $\eps \in (0, \eps_{11})$ and every $\delta \in ( 0, \delta_2)$ there holds
$$
 \lVert  f_{1,t}\rVert_{L^{\infty}_{\omega_0}\left( { \bar\OO^\eps}\right)}   \leq  e^{-{\nu_0\over 2} t}  \left( \lVert f_0\rVert_{L^\infty_{\omega_0}(\OO^\eps)}    +  C  \underset{s\in [0,t]}{\sup}\left[ e^{{\nu_0\over 2}s} \lVert  G_{s}\rVert_{L^{\infty}_{\omega_0\nu^{-1}}\left( \OO^\eps\right)} \right]\right) \qquad \forall t\geq 0,
    $$
for some constant $C>0$, independent of $\eps$.
  \end{prop}

\begin{proof} 
The proof is a repetition of the proof of Proposition~\ref{prop:DissipationLinftyL1}, by using the stretching method argument for cylindrical domains as during the proof of Proposition~\ref{prop:GainLinftyL2LB2}, Proposition~\ref{prop:LinftyEstimatePerturbedFiniteT_0Cylinder} and Proposition~\ref{prop:LinftyEstimatePerturbedFiniteTCylinder}, thus we skip it.
\end{proof}

\subsection{Proof of Proposition~\ref{prop:LinftyEstimatePerturbedFiniteT_Summary_WWWW0}}
We conclude Proposition~\ref{prop:LinftyEstimatePerturbedFiniteT_Summary_WWWW0} by choosing $\eps_9= \eps_{10}$ and $\delta_0 = \delta_1$, where $\eps_{10}, \delta_1>0$ are given by Proposition~\ref{prop:DissipationLinftyL1}, if Assumption \ref{item:RH1} holds, or $\eps_9 = \eps_{11}$ and $\delta_0 = \delta_2$, where $\eps_{11}, \delta_2>0$ are given by Proposition~\ref{prop:DissipationLinftyL2}, under Assumption \ref{item:RH2}. 
\qed

\section{Well-posedness of transport equations with Maxwell boundary conditions}\label{sec:Toolbox}
Before studying well-posedness results for transport equations, we summarize the results obtained during Section~\ref{sec:Linfty} and Section~\ref{sec:LinftyCylinder} for the solutions of Equation~\eqref{eq:PLRBE}. 

\begin{prop}\label{prop:LinftyEstimatePerturbedFiniteT_Summary} 
Consider either Assumption \ref{item:RH1} or \ref{item:RH2} to hold, $\omega_1\in\WWWW_1$ a strongly confining admissible weight function, $G:\UU^\eps\to \R$ satisfying $\lla G_t\rra_{\OO^\eps}=0$ for every $t\geq 0$, and let $f$ be a solution of Equation~\eqref{eq:PLRBE}. There are constructive constants $\eps_{12},\theta >0$ such that for every $\eps\in (0,\eps_{12})$ there holds 
\beqn
\lVert  f_t \rVert_{L^{\infty}_{\omega_1}(\OO^\eps)} \leq  C e^{-\theta \eps^2 t} \left(   \lVert f_0\rVert_{L^{\infty}_{\omega_1}(\OO^\eps)} + \underset{s\in[0,t]}{\sup}\left[ e^{\theta \eps^2 s} \lVert G_s \rVert_{L^{\infty}_{\omega_1\nu^{-1}}(\OO^\eps)}\right]\right)  \qquad \forall t\geq 0,
\label{eq:LinftyPerturbedFiniteTimeDecay_Summary}
\eeqn
for some constant $C>0$, independent of $\eps$.
\end{prop}

\begin{proof}
The proof is immediate from Proposition~\ref{prop:LinftyEstimatePerturbedFiniteT} and Proposition~\ref{prop:LinftyEstimatePerturbedFiniteTCylinder}. We remark that we take $\eps_{12} =\eps_1$ if Assumption \ref{item:RH1}holds, where $\eps_1$ is given by Proposition~\ref{prop:LinftyEstimatePerturbedFiniteT}, and $\eps_{12} = \eps_5$ if Assumption \ref{item:RH2}holds, where $\eps_5$ is given by Proposition~\ref{prop:LinftyEstimatePerturbedFiniteTCylinder}. 
\end{proof}

We dedicate then this section to prove the following well-posedness results. 

\begin{theo}\label{theo:ExistenceL_infty_linear_transport}
Assume either Assumption \ref{item:RH1} or \ref{item:RH2} to hold, consider $\omega_1\in \WWWW_1$ a strongly confining admissible weight function, let $f_0 \in L^\infty_{\omega_1}(\OO^\eps)$, and $G\in L^\infty_{\omega_1\nu^{-1}}(\UU^\eps)$. 

For every $\eps\in (0,\eps_{12})$, where we recall that $\eps_{12}>0$ is given by Proposition~\ref{prop:LinftyEstimatePerturbedFiniteT_Summary}, there is $f\in L^\infty_{\omega_1} (\UU^\eps)$, with a trace function $\gamma f\in L^\infty_{\omega_1}(\Gamma^\eps)$, unique solution to Equation~\eqref{eq:PLRBE} in the distributional sense, i.e for any $\varphi \in \DD(\bar \UU^\eps)$ there holds 
\begin{multline}\label{eq:renormalizedFormulation}
\int_{\OO^\eps} f(t,\cdot) \, \varphi(t,\cdot) \, \dx\dv-\int_0^t \int_{\OO^\eps} f \, K^* \varphi  + f \left( \partial_t \varphi  - v\cdot \grad_x \varphi -\nu \, \varphi \right) +G\, \varphi \,  \dx\dv\ds \\
= \int_{\OO^\eps} f_0(\cdot ) \varphi(0,\cdot ) \dx\dv
+ \int_0^t \int_{\Sigma^\eps} \gamma f\,  \varphi\,  (n_x\cdot v) \d\sigma_x \dv ,
\end{multline}
for every $t\geq 0$, and where we have defined the formal adjoint operator 
\be\label{defAdjointK}
K^* \varphi (v):= \int_{\R^3} k(v_*, v) \,  \varphi(v_*) \, \d v_*.
\ee
Furthermore, there holds the results from Proposition~\ref{prop:LinftyEstimatePerturbedFiniteT_Summary}.
\end{theo}

\begin{theo}\label{theo:ExistenceL_infty_linear_transportAA_delta}
Assume either Assumption \ref{item:RH1} or \ref{item:RH2} to hold, consider $\omega_0\in \WWWW_0$ a weakly confining admissible weight function, let $f_0 \in L^\infty_{\omega_0}(\OO^\eps)$, and $G\in L^\infty_{\omega_0\nu^{-1}}(\UU^\eps)$. 

For every $\eps\in (0,\eps_9)$ and every $\delta\in (0,\delta_0)$, where we recall that $\eps_9, \delta_0>0$ are given by Proposition~\ref{prop:LinftyEstimatePerturbedFiniteT_Summary_WWWW0}, there is $f_1\in L^\infty_{\omega_0} (\UU^\eps)$, with a trace function $\gamma f_1\in L^\infty_{\omega_0}(\Gamma^\eps)$, unique solution to Equation~\eqref{eq:PLRBE1} in the distributional sense, i.e for any $\varphi \in \DD(\bar \UU^\eps)$ there holds 
\begin{multline}\label{eq:renormalizedFormulationAA_delta}
\int_{\OO^\eps} f_1(t,\cdot) \, \varphi(t,\cdot) \, \dx\dv-\int_0^t \int_{\OO^\eps} f_1 \, \AA_\delta ^* \, \varphi  + f \left( \partial_t \varphi  - v\cdot \grad_x \varphi -\nu \, \varphi \right) +G\, \varphi \,  \dx\dv\ds \\
= \int_{\OO^\eps} f_0(\cdot ) \varphi(0,\cdot ) \dx\dv
+ \int_0^t \int_{\Sigma^\eps} \gamma f_1\,  \varphi\,  (n_x\cdot v) \d\sigma_x \dv ,
\end{multline}
for every $t\geq 0$, and where we have defined the formal adjoint operator 
\be\label{def:AdjointAA_delta}
\AA_\delta^* \, \varphi (v) :=  (1-\chi_\delta(v) ) \int_{\R^3} k(v_*, v) \,  \varphi(v_*) \, \d v_*,
\ee
where we recall that $\chi_\delta$ is defined in Subsection \ref{ssec:Strategy}.
Furthermore, there holds the result from Proposition~\ref{prop:LinftyEstimatePerturbedFiniteT_Summary_WWWW0}.
\end{theo}


\subsection{Extra a priori estimates}\label{ssec:Extra_a_priori}

We provide now extra a priori estimates needed for the well-posedness results of this section.

\begin{prop}\label{prop:L2_Apriori}
There is $\kappa \in \R$ such that for every $f$ solution of Equation~\eqref{eq:LRBE} there holds
\be\label{eq:A_priori_bound_HH}
\lvv f_t \rvv_\HH \leq e^{\kappa t} \lvv f_0\rvv_\HH,
\ee
for every $t\geq 0$.
\end{prop}

 \begin{proof}
 This result is classical in the study of the Boltzmann equation thus we only sketch it.  We first observe that \cite[Theorem 7.2.4]{MR1307620} implies that
\be\label{eq:Cercignani_CK}
\lVert K h \rVert_{\HH} \leq C_K \lVert  h \rVert_{\HH} ,
\ee
for some constant $C_K>0$.
Arguing then at a formal level, we have that if $f$ is a solution of Equation~\eqref{eq:LRBE} there holds
\bean
\frac 12 \frac d{dt} \int_{\OO^\eps} f_t^2\MM^{-1} &=& \int_{\OO^\eps} f_t \left( -v\cdot \grad_x f_t + Kf_t -\nu f_t  \right) \MM^{-1} \\
&\leq & - \frac 12 \int_{\Sigma^\eps} \gamma f_t^2 \,  (n_x\cdot v)  \,  \MM^{-1} + \left(C_K - \nu_0 \right) \int_{\OO^\eps} f_t^2 \MM^{-1} , 
\eean
where we have used the Cauchy-Schwartz inequality and \eqref{eq:Cercignani_CK} to obtain the second line. We conclude by using the fact that the Maxwell collision operator satisfies 
\be\label{eq:ControlRRRRMM}
\RRR: L^2(\Sigma_+; \lv n_x \cdot v\rv \MM^{-1}) \to   L^2(\Sigma_+; \lv n_x \cdot v\rv \MM^{-1})
\ee
with norm less than $1$ (see for instance \cite[Lemma 2.2]{CMKFP24}), using the Grönwall lemma and setting $\kappa = C_K-\nu_0 \in \R$.
\end{proof}

Moreover, we have now the following lemmas on the a priori estimates for the trace.

For a set $\SS \in \{\bar \Lambda_1^\eps \cup \bar \Lambda_2^\eps, \, \bar \Lambda_3^\eps, \, \SSSS^\eps\}$, where we recall that $\SSSS^\eps$ has been defined in \eqref{def:SSSS}, we define $\delta_\SS$ as a smooth $C^2$ function coinciding, in a small neighborhood of $\SS$, with the distance function to the compact set $\SS$. We remark that regularity of $\delta_\SS$ is justified since $\CC$ is always a smooth compact submanifold (or the union of disjoint compact submanifolds) of $\R^3$ making the distance function to such set a smooth function in a small neighborhood around it, see for instance \cite{MR749908, MR737190, MR614221}. 

\smallskip
We define then the function
$$
\zeta_\SS(x) := \left\{\begin{array}{cc}
1 & \text{ if Assumption~\ref{item:RH1} holds}\\
\\
\displaystyle {(\delta_\SS (x))^2 \over 1+ (\delta_\SS (x)) ^2} & \text{ if Assumption~\ref{item:RH2} holds},
\end{array}\right.
$$
and we observe that $\zeta_\SS \in C^1 (\bar \Omega^\eps)$ for any $\SS \in \{\bar \Lambda_1^\eps \cup \bar \Lambda_2^\eps, \, \bar \Lambda_3^\eps, \, \SSSS^\eps\}$. We have then the following extra a priori estimate for the boundary term.

 \begin{prop}\label{prop:L2_AprioriBdy}
 Assume there to hold either Assumption~\ref{item:RH1} or Assumption~\ref{item:RH2}. 
There is $\kappa \in \R$ such that for every $f$ solution of Equation~\eqref{eq:LRBE} there holds
\be\label{eq:A_priori_bound_Bdy}
\int_0^t \int_{\Sigma^\eps}  (\gamma f_s)^2 \, (n_x\cdot v)^2 \, \zeta_{\SSSS^\eps}(x) \, \la v\ra^{-2} \, \MM^{-1} \,  \dv\d\sigma_x\ds  \leq e^{\kappa t} \lvv f_0\rvv_\HH,  \qquad \forall t\geq 0.
\ee
\end{prop}

\begin{rem}\label{rem:L2_AprioriBdy}
We note that on cylindrical domains (i.e under Assumption~\ref{item:RH2}) the estimates for the boundary are more degenerate than for smooth domains. The extra term $\zeta_{\SSSS^\eps}$, making at all possible the control of the trace, serves to---in a sense---\emph{smooth out} the normal vector when approaching the singular set $\SSSS^\eps$. 
\end{rem}

\begin{proof}[Proof of Proposition~\ref{prop:L2_AprioriBdy}]
We divide the proof in two steps. First we obtain \eqref{eq:A_priori_bound_Bdy} for smooth domains, and after we repeat and adapt those computations for the setting of the cylinder. 

\medskip\noindent
\emph{Case 1. (Smooth domains---Assumption~\ref{item:RH1})} We recall that $\zeta_{\SSSS^\eps} \equiv 1$ under Assumption~\ref{item:RH1}, and that in this setting the normal vector is a $C^2$ vector field of $\R^3$. We then compute
\be\label{eq:BdyApriori1}
\begin{aligned}
\frac 12\frac \d{\dt} \int_{\OO^\eps} (f_t)^2 \MM^{-1} \la v\ra^{-2} (n_x\cdot v)  &= \int_{\OO^\eps} f_t \left( -v\cdot \grad_x f_t + Kf_t -\nu f_t  \right) \MM^{-1} \la v\ra^{-2} (n_x\cdot v)\\
&=- \frac 12 \int_{\OO^\eps} v\cdot \grad_x (f_t^2) \MM^{-1} (n_x \cdot v) \la v\ra^{-2}  + C \lvv f_t\rvv_\HH^2,
\end{aligned}
\ee
for some constant $C>0$, and we remark that we have used that $\lv \la v\ra^{-2} (n_x\cdot v) \rv\leq 1$, the Cauchy-Schwartz inequality, and \eqref{eq:Cercignani_CK} to obtain the second line, in the same spirit as during the proof of Proposition~\ref{prop:L2_Apriori}.

Using integration by parts we now have that 
$$
-\int_{\OO^\eps} v\cdot \grad_x (f_t^2) \MM^{-1} (n_x \cdot v) \la v\ra^{-2} \dv\dx 
\leq
 - \int_{\Sigma^\eps} f_t^2 \MM^{-1} (n_x \cdot v)^2 \la v\ra^{-2}\dv\d\sigma_x  + C \lvv f_t \rvv_\HH^2,
$$
for some constant $C>0$, and we remark that we have used the regularity of the normal vector to deduce the last line. 
We conclude the proof by putting everything together, integrating \eqref{eq:BdyApriori1} in time, and using \eqref{eq:A_priori_bound_HH}.

\medskip\noindent
\emph{Case 2. (Cylindrical domains---Assumption~\ref{item:RH2})} 
During this proof we write any element $x\in \R^3$ by its components as $x = (x_1, x_2, x_3)$. We consider the vector fields $n_1, n_2: \R^3\to \R^3$ defined respectively by
$$
n_1(x_1, x_2, x_3) := {(x_1, 0, 0) \over \eps^{-1} L } \quad \text{ and } \quad n_2 (x_1, x_2, x_3) := {(0,x_2, x_3) \over \eps^{-1} \RRRR}.
$$
We then observe that each of them is smooth and we remark that $n_1(x) = n(x)$ for every $x\in \Lambda_1^\eps\cup \Lambda_2^\eps$ and $n_2(x) = n(x)$ for every $x\in \Lambda_3^\eps$. Moreover, from their very definition there holds 
\beqn\label{eq:BoundednessN1}
\lv n_i (x) \rv \leq 1 \qquad      \forall x\in \bar\Omega^\eps ,
\eeqn
for any $i\in \{1,2\}$. We then compute 
\be\label{eq:BdyApriori_Cylinder1}
\begin{aligned}
&\frac 12\frac \d{\dt} \int_{\OO^\eps} (f_t)^2 \MM^{-1} \la v\ra^{-2} \, \zeta_{\Lambda_3^\eps}(x)\,   (n_1(x) \cdot v)  \, \dv \dx\\
=& \int_{\OO^\eps} f_t \left( -v\cdot \grad_x f_t + Kf_t -\nu f_t  \right) \MM^{-1} \la v\ra^{-2} \, \zeta_{\Lambda_3^\eps}(x)\,   (n_1(x) \cdot v)  \, \dv \dx\\
=&- \frac 12 \int_{\OO^\eps} v\cdot \grad_x (f_t^2)\MM^{-1} \la v\ra^{-2} \, \zeta_{\Lambda_3^\eps}(x)\,   (n_1(x) \cdot v)  \, \dv \dx + C \lvv f_t\rvv_\HH^2,
\end{aligned}
\ee
for some constant $C>0$, and we remark that we have used that $\lv \la v\ra^{-2} (n_1(x)\cdot v) \rv\leq 1$, that $\zeta_{\Lambda_3^\eps}$ is uniformly bounded from its very definition, the Cauchy-Schwartz inequality, and \eqref{eq:Cercignani_CK} to obtain the last line. 

We then recall that by its very definition $\zeta_{\Lambda_3^\eps} (x) = 0$ for every $x\in \Lambda_3^\eps$. We then have that
$$
\begin{aligned}
& -\int_{\OO^\eps} v\cdot \grad_x (f_t^2) \MM^{-1} \la v\ra^{-2} \, \zeta_{\Lambda_3}(x)\,   (n_1(x) \cdot v)  \, \dv \dx \\
=& - \int_{\Lambda_1^\eps\cup \Lambda_2^\eps} f_t^2 \MM^{-1} \zeta_{\Lambda_3}(x)\, (n_x \cdot v)^2  \la v\ra^{-2} \dv\d\sigma_x + \int_{\OO^\eps} f_t^2 \MM^{-1} \, (v\cdot \grad_x (n_1(x) \cdot v))\,  \la v\ra^{-2}\dv\dx  \\
&\quad \qquad + \int_{\OO^\eps}  f_t^2 \MM^{-1} \la v\ra^{-2} \, (v\cdot \grad_x \zeta_{\Lambda_3}(x))\,   (n_1(x) \cdot v)  \, \dv \dx\\
\leq&  - \int_{\Lambda_1^\eps\cup \Lambda_2^\eps} f_t^2 \MM^{-1} \zeta_{\Lambda_3}(x)\, (n_x \cdot v)^2 \la v\ra^{-2}\dv\d\sigma_x  + C \lvv f_t \rvv_\HH^2,
\end{aligned}
$$
where we have simultaneously used the regularity and boundedness properties of $\zeta_{\Lambda_3^\eps}$ and $n_1$, and the fact that $n_1(x) = n(x)$ for every $\Lambda_1^\eps\cup \Lambda_2^\eps$, and $\zeta_3(x) =0$ for every $x\in \Lambda_3^\eps$. 

We now observe that 
$$
 \zeta_{\Lambda_3}(x) = \zeta_{\SSSS^\eps} (x) \qquad \forall x\in \Lambda_1^\eps\cup \Lambda_2^\eps,
$$
and integrating in time and arguing as during the Case 1 of this proof we deduce that there is a constant $\kappa >0$ such that 
$$
\int_0^t \int_{\Lambda_1^\eps\cup \Lambda_2^\eps} f_s^2 \MM^{-1} \zeta_{\SSSS^\eps} (x)\, (n_x \cdot v)^2 \la v\ra^{-2}\dv\d\sigma_x \ds \lesssim e^{\kappa t} \lvv f_0 \lvv_\HH^2. 
$$
Repeating this exact arguments but using $\zeta_{\Lambda_1^\eps \cup \Lambda_2^\eps}(x)\,   (n_2(x) \cdot v) $ instead of $\zeta_{\Lambda_3^\eps}(x)\,   (n_1(x) \cdot v)$ in \eqref{eq:BdyApriori_Cylinder1} we obtain that
$$
\int_0^t \int_{\Lambda_3^\eps} f_s^2 \MM^{-1} \zeta_{\SSSS^\eps} (x)\, (n_x \cdot v)^2 \la v\ra^{-2}\dv\d\sigma_x \ds \lesssim e^{\kappa t} \lvv f_0 \lvv_\HH^2. 
$$
We conclude this case by putting together the two last estimates. 
\end{proof}

Furthermore, we provide a priori estimates for the solutions of Equation~\eqref{eq:PLRBE} in weighted $L^\infty$ spaces, when $\lla G\rra_{\OO^\eps} \neq 0$.

\begin{prop}\label{prop:LinftyEstimatePerturbedFiniteT_Summary_NoZeroMeasure} 
Consider either Assumption \ref{item:RH1} or \ref{item:RH2} to hold, $\omega_1\in\WWWW_1$ a strongly confining admissible weight function, $G:\UU^\eps\to \R$, and let $f$ be a solution of Equation~\eqref{eq:PLRBE}. For every $\eps\in (0,\eps_{12})$, where we recall that $\eps_{12}>0$ is given by Proposition~\ref{prop:LinftyEstimatePerturbedFiniteT_Summary}, there holds 
\be
\lVert  f_t \rVert_{L^{\infty}_{\omega_1}(\OO^\eps)} \leq  C  \left(  e^{-\theta \eps^2 t} \lVert f_0\rVert_{L^{\infty}_{\omega_1}(\OO^\eps)} +  \lVert G \rVert_{L^{\infty}_{\omega_1\nu^{-1}}(\UU^\eps)}\right)  \qquad \forall t\geq 0 ,\label{eq:LinftyPerturbedFiniteTimeDecay_Summary_NoZeroMeasure}
\ee
for some constant $C>0$, independent of $\eps$.
\end{prop}

\begin{proof}
We emphasize that the only difference between Proposition~\ref{prop:LinftyEstimatePerturbedFiniteT_Summary_NoZeroMeasure} and Proposition ~\ref{prop:LinftyEstimatePerturbedFiniteT_Summary} is the fact that $G$ does not have total mass zero for all time. This implies that during the proof we will not be able to apply the hypocoercivity decay \eqref{eq:Hypo} to $G$, and in its place we will use Proposition~\ref{prop:L2_Apriori}. We then put the extra exponential growth in time given by \eqref{eq:A_priori_bound_HH} with the constants outside and shifting our estimate to control instead $\lVert G \rVert_{L^{\infty}_{\omega_1\nu^{-1}}(\UU^\eps)}$, the proof of Proposition~\ref{prop:LinftyEstimatePerturbedFiniteT_Summary_NoZeroMeasure} follows exactly that of Proposition ~\ref{prop:LinftyEstimatePerturbedFiniteT_Summary}, thus we skip it.
\end{proof}

\subsection{Well-posedness of Equation \eqref{eq:LRBE} in a weighted $L^2$ framework}\label{sec:L2WellPosedness}
We define now the boundary measures
$$\d \xi_1:=\MM^{-1} (n_x\cdot v)\dv\d\sigma_x ,\quad \text{ and } \quad \d\xi_2:= \MM^{-1}  (n_x\cdot v)^2  \, \zeta_{\SSSS^\eps} (x) \, \la v\ra^{-2} \, \dv\d\sigma_x,
$$
and we define the space of renormalizing  functions $C_{pw, *}^1(\R)$ as the space of $C^1$ piecewise functions with finite limits at $\pm\infty$, and such that $s\mapsto \la s\ra\beta'(s)$ is bounded in $\R$.

We then prove the following well-posedness of Equation~\eqref{eq:LRBE} in $\HH$ in order to justify the computations from Section \ref{sec:Hypo}.

\begin{theo}\label{theo:ExistenceL2}
Assume there to hold either Assumption~\ref{item:RH1} or Assumption~\ref{item:RH2} and consider $f_0 \in  \HH$. There is $f\in C(\R_+, \HH)$ with an associated trace function $\gamma f\in L^2 (\Gamma^\eps;\,   \d\xi_2 \dt )$, 
 unique global solution to Equation~\eqref{eq:LRBE} in the distributional sense (see for instance \eqref{eq:renormalizedFormulation}). 
\end{theo}

\begin{rem}\label{rem:ExistenceL2}
In particular, Theorem~\ref{theo:ExistenceL2} implies the existence of a strongly continuous semigroup $S_\LLL:\HH\to \HH$ associated to the solutions of Equation~\eqref{eq:LRBE}.
\end{rem}

\begin{rem}
We observe that, in the case of cylindrical domains (i.e under Assumption~\ref{item:RH2}), the functional space where the trace function is well defined is more singular than in the case of smooth domains. This is in particular reminiscent of our comments from Remark~\ref{rem:L2_AprioriBdy}
\end{rem}

The problem of well-posedness for transport equations with non-local terms presenting boundary conditions has been deeply addressed in the literature, see for instance  \cite{sanchez2024kreinrutmantheorem, Bardos70, Beal87, MR1022305, MR2072842, Mischler2000, Mischler2010, MR1132764}.
However, due to the lack of a precise reference, we provide a sketch of the proof for Theorem~\ref{theo:ExistenceL2}.

\begin{proof}[Proof of Theorem~\ref{theo:ExistenceL2}] 
We divide the proof into three steps.

\medskip\noindent
\emph{Step 1.} We consider a function $\ffff\in L^2(\Gamma^\eps;\dt \d\xi_1)$, and we study the following evolution equation
\begin{equation}
	\left\{\begin{array}{llll}
		 \partial_t f &=&  \LLL f   &\text{ in }\UU^\eps\\
		\gamma_-f&=& \ffff  &\text{ on }\Gamma^\eps_{-}\\
		f_{t=0}&=&f_0 &\text{ in }\OO^\eps.
	\end{array}\right.\label{eq:TransportL2K_Inflow}
\end{equation}
A direct application of \cite[Proposition 8.16]{sanchez2024kreinrutmantheorem} gives the existence of $f\in C(\R_+, \HH)$, with a trace $\gamma f\in L^2(\Gamma^\eps; \d\xi_1 \dt)$, unique renormalized solution to Equation~\eqref{eq:TransportL2K_Inflow}.
Furthermore, there is a constant $\kappa>0$ for which there holds the enegy estimate
\be\label{eq:Inflow_energy_estimate}
\lvv f_t\rvv_\HH^2 + \int_0^t e^{2\kappa  (t-s)}  \lvv \gamma_+ f_s\rvv^2_{L^2(\Sigma^\eps_+; \d\xi_1)}   \ds \leq e^{2\kappa t} \lvv f_0\rvv_\HH^2 + \int_0^t e^{2\kappa (t-s)}  \lvv \ffff_s \rvv^2_{L^2(\Sigma^\eps_-; \d\xi_1)}   \ds,
\ee
for every $t\geq 0$.

\medskip\noindent
\emph{Step 2. } 
Using then again \eqref{eq:ControlRRRRMM} and the energy estimate \eqref{eq:Inflow_energy_estimate} in a Banach fixed point argument in the same spirit as in the proof of \cite[Proposition 3.3]{CGMM24}, we obtain the existence of $f\in C(\R_+, \HH)$, with a trace function $\gamma f\in L^2(\Gamma^\eps_T; \d\xi_1 \dt)$, unique renormalized solution of the equation
\beqn
	\left\{\begin{array}{llll}
		 \partial_t f &=&  -v\cdot \grad_x f + K f  - \nu f  &\text{ in }\UU^\eps_T\\
		\gamma_- f&=& \alpha \, \RRR \gamma_+ f  &\text{ on }\Gamma^\eps_{T, -}\\
		f_{t=0}&=&f_0 &\text{ in }\OO^\eps, 
	\end{array}\right.
\eeqn
for any $\alpha\in (0,1)$. We remark that this solution is in the sense that for every $\phi \in \DD(\bar \UU^\eps)$, and every $\beta \in C_{pw, *}^1$ there holds
\begin{multline}\label{eq:renormalizedFormulation_HH_alpha}
\int_{\OO^\eps} \beta(f)(t,\cdot) \, \varphi(t,\cdot) -\int_0^t \int_{\OO^\eps}  \beta'(f)\varphi Kf  + \beta(f) \left( \partial_t \varphi  - v\cdot \grad_x \varphi  \right) - \nu \beta'(f) f \\
 - \int_0^t \int_{\Sigma^\eps} \gamma_+ \beta(f)\,  \varphi\,  (n_x\cdot v)_+ 
= \int_{\OO^\eps} \beta(f_0)(\cdot ) \varphi(0,\cdot ) ,
\end{multline}
for every $t\in [0,T]$. Furthermore, there holds the energy estimate
\be\label{eq:linear_gak_bdd}
\lvv f_{kt} \rvv_\HH + (1-\alpha _k) \int_0^t e^{2\kappa  (t-s)}  \lvv \gamma_+ f_{ks}\rvv^2_{L^2(\Sigma^\eps_+;  \d\xi_1)}  \leq e^{\kappa t} \lvv f_0\rvv_\HH, 
\ee
for any $t \in [0,T]$

\medskip\noindent
\textit{Step 3.} For a sequence $\alpha_k \in (0,1)$ such that $\alpha_k \nearrow 1$, we consider the sequence $(f_k)$ obtained by using the Step~2 as the solution to the modified Maxwell reflection boundary condition problem
\begin{equation}\label{eq:linear_gak}
\left\{
\begin{array}{rcll}
 \partial_t f_k &=& - v \cdot \nabla_x f_k + \LLL f_k   &\text{ in } \UU^\eps_T \\
 \gamma_{-} f_k   &=&  \alpha_k \, \RRR \gamma_{+} f_k   &\text{ on }  \Sigma_{T, -}^\eps \\
 f_{k, t=0} &=& f_0   &\text{ in }   \OO^\eps.
\end{array}
\right.
\end{equation}
Using the renormalized formulation \eqref{eq:renormalizedFormulation_HH_alpha} with the choices  $\beta (s) = \beta_M(s) = M\wedge s^2$, for any $M>0$ and as tests functions $\varphi$, the ones considered during the proof of Proposition~\ref{prop:L2_AprioriBdy}, we may argue as during the proof of Proposition~\ref{prop:L2_AprioriBdy}, and using the integral version of the Grönwall lemma we additionally have that
$$
\int_{\Gamma^\eps_T} (\gamma f_k)^2 \d\xi_2  \, \d t  
\lesssim   \| f_0 \|^2_{\HH} e^{\kappa T}. 
$$

From the above estimates, we deduce that, up to the extraction of a subsequence, there exist $f \in L^2(0,T; \HH) \cap L^\infty(0,T; \HH)$ and $\mathfrak{f}_\pm \in L^2(\Gamma_{T, \pm}^\eps; \d\xi_2 \dt)$ such that 
$$
f_k \wto f \hbox{ weakly in } \ L^2(0,T; \HH)  \cap L^\infty(0,T; \HH), 
\quad
\gamma_\pm f_k \wto \mathfrak{f}_\pm \hbox{ weakly in } \  L^2(\Gamma_{T, \pm}^\eps; \d\xi_2 \d t). 
$$

Since $\langle v \rangle  \MM \in L^2(\R^d)$, 
 we have that $L^2(\Gamma^\eps_T;  \d\xi_2\dt) \subset L^1(\Gamma^\eps_T; \zeta_{\SSSS^\eps} (x) \, (n_x\cdot v)  \dv\d\sigma_x\dt)$. 
Moreover, from the very definition of the rescaled Maxwell boundary condition \eqref{eq:RBEBC}, we have that as a map
\beqn\label{eq:KolmogorovWP-hypL1}
\RRR: L^1(\Sigma_+^\eps; \zeta_{\SSSS^\eps}(x)\,  (n_x\cdot v)  \dv\d\sigma_x) \to L^1(\Sigma_-^\eps; \zeta_{\SSSS^\eps}(x)\,  (n_x\cdot v)  \dv\d\sigma_x),
\eeqn
has norm less than 1.
Altogether this implies that $\RRR(\gamma_+ f_{k})  \wto \RRR(\mathfrak{f}_+)$   weakly in $L^1(\Gamma_-; \zeta_{\SSSS^\eps} (x)\,  (n_x\cdot v)  \dv\d\sigma_x)$. 

Fur\-ther\-mo\-re, from \cite[Proposition 8.10]{sanchez2024kreinrutmantheorem}, we have that $\gamma f_k \wto \gamma f$ weakly in $\Lloc^2 (\Gamma^\eps_T;  \d\xi_2 \dt)$. 
Using both convergences in the boundary condition $\gamma_- f_k = \RRR(\gamma_+ f_k)$, we obtain $\gamma_- f = \RRR(\gamma_+ f)$.

We may thus pass to the limit in the weak formulation of Equation~\eqref{eq:linear_gak}, obtained from the Step 2, and we obtain that $f \in C([0,T]; \HH)$ is a weak solution to Equation~\eqref{eq:LRBE}. 
Moreover, passing to the limit in \eqref{eq:linear_gak_bdd}, we also have that \eqref{eq:A_priori_bound_HH} holds. This and the linearity give the uniqueness of the solution to Equation~\eqref{eq:LRBE}, and repeating this argument in every time interval we conclude the existence and uniqueness of a weak global solution. 
\end{proof}

\subsection{Proof of Theorem~\ref{theo:ExistenceL_infty_linear_transport}}\label{ssec:LinftyWellPosednessK}
We split the proof into four steps.

\medskip\noindent
\emph{Step 1. (Classical solutions to the transport equation with smooth coefficients)} 
We consider now $H_1, H_2, \HHHH \in C^1_c(\bar\UU^\eps)$, $h_0\in C^1_c(\OO^\eps)$ and we study the following transport evolution equation
\begin{equation}
	\left\{\begin{array}{llll}
		 \partial_t h &=&  \TTT h + H_1+H_2  &\text{ in }\UU^\eps \\
		\gamma_-h&=& \HHHH  &\text{ on }\Lambda^\eps_{-} \\
		h_{t=0}&=&h_0 &\text{ in }\OO^\eps.
	\end{array}\right.\label{eq:TransportLinftyInflowCCCC}
\end{equation}
We classically have (see for instance the proof of \cite[Lemma 8.12]{sanchez2024kreinrutmantheorem} and the references therein) that a solution to Equation~\eqref{eq:TransportLinftyInflowCCCC} is given by the representation formula
\begin{multline}\label{eq:TransportRepresentation}
h(t,x,v) = e^{-\nu(v) t} h_0(x,v) \, \Ind_{t_1\leq 0} + \int_{\max(0, t_1)}^t e^{-\nu(v) (t-s)} \, (H_1+H_2)(s,x,v)\, \ds \\
+ e^{-\nu(v)(t-t_1)} \HHHH(t_1, x_1, v) \, \Ind_{t_1>0}, 
\end{multline}
and, in a smooth framework, we further have that $h$ also satisfies the weak formulation 
\begin{multline}\label{eq:TransportWeakFormulation}
\int_{\OO^\eps} h_t \, \varphi(t, \cdot) - \int_0^t \int_{\OO^\eps} h \, \left( v\cdot \grad_x \varphi - \nu(v) \varphi \right)  + \int_0^t \int_{\Sigma_+^\eps} \gamma_+ h \, \varphi \,  (n_x \cdot v)_+ \\= \int_{\OO^\eps}h_0 \, \varphi(0,\cdot) + \int_0^t \int_{\Sigma_-^\eps} \HHHH \, \varphi \, (n_x \cdot v)_+ + \int_0^t \int_{\OO^\eps} (H_1+H_2) \, \varphi .
\end{multline}
for every $\varphi \in \DD(\bar \UU^\eps)$ and $t\in \R_+$.

Moreover, multiplying \eqref{eq:TransportRepresentation} by $\omega_1$, taking the $L^\infty(\OO^\eps)$ norm, and using \ref{item:K3} with $\nu_2 = \nu_0/2$, we deduce the energy estimates
\be\label{eq:TransportEenergy}
\lvv h \rvv_{L^\infty_{\omega_1}(\bar \UU^\eps)} \leq  \lvv h_0\rvv_{L^\infty_{\omega_1}(\OO^\eps)} + \lvv \HHHH \rvv_{L^\infty_{\omega_1}(\Gamma_-^\eps)} \\
+ \lvv H_1\rvv_{L^\infty_{\omega_1}(\UU^\eps)} + 2 {\nu_1\over \nu_0} \lvv H_2\rvv_{L^\infty_{\omega_1\nu^{-1}}(\UU^\eps)},
\ee

\medskip\noindent
\emph{Step 2. (Solutions to the transport equation with $L^\infty$ coefficients)} We assume now  $H_1 \in L^\infty_{\omega_1}(\UU^\eps)$, $H_2 \in L^\infty_{\omega_1\nu^{-1}}(\UU^\eps)$, $\HHHH \in L^\infty_{\omega_1}(\Gamma_-^\eps)$ and $h_0\in L^\infty_{\omega_1}(\OO^\eps)$. We take then the sequence of smooth functions $H^n_1, H^n_2 \in C^1_c(\bar \UU^\eps)$, $\HHHH^n\in C^1_c(\Gamma^\eps)$ and $h_0^n\in C^1_c(\OO^\eps)$ such that, as $n\to \infty$, there holds
\begin{equation}\label{eq:StrongConv_h_n}
\begin{array}{lllll}
H^n_1\to H_1 &\text{ strongly in } L^\infty_{\omega_1}(\UU^\eps),& \qquad&H^n_2\to H_2 &\text{ strongly in } L^\infty_{\omega_1\nu^{-1}}(\UU^\eps),\\
\HHHH^n\to \HHHH &\text{ strongly in } L^\infty_{\omega_1}(\Gamma_-^\eps),& \qquad  & h^n_0\to g_0 &\text{ strongly in } L^\infty_{\omega_1}(\OO^\eps).
\end{array}
\end{equation}
From the analysis performed during the Step 1, there is $h^n$ unique solution to Equation~\eqref{eq:TransportLinftyInflowCCCC} associated with $H_i^n$ for $i=1, 2$, $\HHHH^n$, and $h_0^n$. Moreover, there holds the representation formula \eqref{eq:TransportRepresentation}, the weak formulation \eqref{eq:TransportWeakFormulation}, and the energy estimate \eqref{eq:TransportEenergy}. 

In particular, defining the trace using the representation formula \eqref{eq:TransportRepresentation} as detailed in \cite[Definition 8.1]{sanchez2024kreinrutmantheorem}, we deduce from \eqref{eq:TransportEenergy}, the linearity of Equation~\eqref{eq:TransportLinftyInflowCCCC}, and the convergences from \eqref{eq:StrongConv_h_n} that $h^n$ and $\gamma h^n$ are Cauchy sequences in $L^\infty_{\omega_1}(\UU^\eps)$ and $L^\infty_{\omega_1}(\Gamma^\eps)$ respectively. Therefore, there are functions $h\in L^\infty_{\omega_1}(\UU^\eps)$ and  $\gamma h \in L^\infty_{\omega_1}(\Gamma^\eps)$ such that, as $n\to \infty$, there holds 
$$
h^n \to h \quad \text{ strongly in } L^\infty_{\omega_1}(\OO^\eps) \quad \text{ and } \quad \gamma h^n \to \gamma h \quad \text{ strongly in } L^\infty_{\omega_1}(\Gamma_-^\eps).
$$
Moreover, passing to the limit in the weak formulation of Equation~\eqref{eq:TransportLinftyInflowCCCC} associated with $H_i^n$ for $i=1, 2$, $\HHHH^n$, and $h_0^n$, we deduce that $h$, with the trace function $\gamma h$, is a weak solution of Equation~\eqref{eq:TransportLinftyInflowCCCC} associated with $H_i$ for $i=1, 2$, $\HHHH$, and $h_0$, satisfying the energy estimate 
\begin{multline}\label{eq:TransportEenergyL_infty}
\lvv h \rvv_{L^\infty_{\omega_1}( \UU^\eps)} + \lvv \gamma h \rvv_{L^\infty_{\omega_1}( \Gamma^\eps)} \leq  \lvv h_0\rvv_{L^\infty_{\omega_1}(\OO^\eps)} + \lvv \HHHH \rvv_{L^\infty_{\omega_1}(\Gamma_-^\eps)} \\
+ \lvv H_1\rvv_{L^\infty_{\omega_1}(\UU^\eps)} + 2 {\nu_1\over \nu_0} \lvv H_2\rvv_{L^\infty_{\omega_1\nu^{-1}}(\UU^\eps)}.
\end{multline}
In particular, \eqref{eq:TransportEenergyL_infty} and the linearity of Equation~\eqref{eq:TransportLinftyInflowCCCC} imply the uniqueness of this solution.

Furthermore, we may pass to the limit in \eqref{eq:TransportRepresentation} a.e in $(x,v) \in \OO^\eps$ such that $h$ satisfies the Duhamel formulation \eqref{eq:TransportRepresentation}.

\medskip\noindent
\emph{Step 3.} We now set $\psi^0=0$, $\psi_0 \in L^\infty_{\omega_1}(\OO^\eps)$, $\alpha \in (0,1)$, we recall that $G \in L^\infty_{\omega_1\nu^{-1}}(\UU^\eps)$ and we consider the recurrent sequence of solutions given by the following evolution equation
\begin{equation}
	\left\{\begin{array}{llll}
		 \partial_t \psi^{k+1} &=&  \TTT \psi^{k+1} + K \psi^k + G  &\text{ in }\UU^\eps\\
		\gamma_-\psi^{k+1}&=& \alpha\,  \RRR \gamma_+ \psi^k  &\text{ on }\Gamma^\eps_{-}\\
		\psi^{k+1}_{t=0}&=& \psi_0 &\text{ in }\OO^\eps.
	\end{array}\right.\label{eq:TransportLinftyK}
\end{equation}
Indeed, we remark that if we assume that, for a certain $k$, $\psi^k\in L^\infty_{\omega_1}(\UU^\eps)$, then \ref{item:K1} implies that $K\psi^k \in L^\infty_{\omega_1} (\UU^\eps)$. Therefore Step 2 implies the existence of $\psi^{k+1} \in L^\infty_{\omega_1}(\UU^\eps)$, with a trace $\gamma \psi^{k+1} \in L^\infty_{\omega_1} (\Gamma^\eps)$, unique weak solution of Equation~\eqref{eq:TransportLinftyK} in the sense provided by the Step 2.

We take now $A>0$ to be fixed later and, inspired by the recent series of papers \cite{CMKFP24, CM_Landau_domain, CGMM24}, we define the modified weight functions
$$
\omega_1^A = \omega_1^A (v):=  \MMM^{-1} \xi_A(v) + (1-\xi_A(v)) \omega_1 ,
$$
where $\xi_A(v):= \xi(|v|/A)$, for a function $\xi \in C^2(\R_+, \R)$, such that $\mathbf{1}_{[0,1]} \le \xi \le \mathbf{1}_{[0,2]}$. Moreover, we remark that there is a constant $c_A >0$ such that 
\be\label{eq:equivalency_omega_omega_A}
c_A^{-1} \omega_1 \leq \omega_1^A \leq c_A \omega.
\ee
Using now the representation formula given by the Step 2 for the solutions of Equation~\eqref{eq:TransportLinftyK}, multiplying by $\omega_1^A$, using \ref{item:K1} together with \eqref{eq:equivalency_omega_omega_A}, and arguing as during the proof of \ref{item:K3} with $\nu_2 = \nu_0/2$, we deduce that there is a constant $C_0>0$ satisfying 
\begin{multline}\label{eq:TransportRepresentationPsi_omega_A}
\lv \psi^{k+1} (t,x,v)\rv \omega_1^A = e^{-\nu_0 t} \lvv \psi_0\rvv_{L^\infty_{\omega_1^A}(\OO^\eps)} + t \, C_0 \lvv \psi^k\rvv_{L^\infty_{\omega_1^A}(\OO^\eps)}  \\
+ C_0 \lvv G \rvv_{L^\infty_{\omega_1^A\nu^{-1}}(\OO^\eps)}  + \alpha \, \omega_1^A(v)\, \lv \RRR \gamma_+ \psi^k (t_1, x_1, v) \rv.
\end{multline}
To bound the boundary term from \eqref{eq:TransportRepresentationPsi_omega_A} we remark that
\bean
\lv \RRR \gamma_+ \psi^k (t_1, x_1, v) \rv &\leq& (1-\iota^\eps) \lvv \psi^k \rvv_{L^\infty_{\omega_1^A}(\OO^\eps)} + \iota^\eps \omega_1^A \MM \left\lv \int_{\R^3} \gamma_+ \psi^k (t_1, x_1, u) \, (n_x\cdot u)_+ \, \d u \right\rv \\
&\leq &(1-\iota^\eps) \lvv \psi^k \rvv_{L^\infty_{\omega_1^A}(\OO^\eps)} + \iota^\eps \lvv \psi^k\rvv_{L^\infty_{\omega_1^A}(\OO^\eps)} \omega_1^A \MMM  \int_{\R^3} (n_x\cdot u)_+ \, \omega_1^A(u)^{-1} \d u .
\eean
On the one hand, we observe that from the very definition of $\omega_A$ there holds
$$
\omega_1^A \MMM =  \xi_A +  (1-\xi_A) \omega_1 \MMM \leq 1. 
$$
On the other hand, we have that
$$
1\leq \int_{\R^3} (n_x\cdot u)_+ \, \omega_1^A(u)^{-1} \d u = \int_{\R^3} \MMM(u) (n_x\cdot u)_+ \left(  \xi_A + (1-\xi_A) \omega \MMM\right)^{-1} \d u := \Theta_A \underset{A\to \infty}\longrightarrow 1,
$$
where we have used the fact that $ \xi_A + (1-\xi_A) \omega \MMM\to 1$ as $A\to \infty$ and $\widetilde \MMM = 1$. 

Using then that $\alpha \in (0,1)$ and $\Theta_A\to 1$ as $A\to \infty$, we can choose $A>0$ large enough such that $ \Theta_A < 1/\alpha$, this implies that
$$
\alpha \left( 1-\iota^\eps + \iota^\eps \Theta_A\right) = \alpha \left( 1 + \iota^\eps ( \Theta_A-1) \right) \leq \alpha \Theta_A <1,
$$
where we have used the fact that $\Theta_A \geq 1$ for every $A>0$ and $\iota^\eps(x) \leq 1$ for almost every $x\in \partial \Omega^\eps$. 
Using then the linearity of Equation~\eqref{eq:TransportLinftyK} and the previous analysis we have that
$$
\lvv \psi^{k+1} - \psi^k\rvv_{L^\infty_{\omega_1^A}(\OO^\eps)} + \lvv \gamma \psi^{k+1} - \gamma \psi^k\rvv_{L^\infty_{\omega_1^A}(\OO^\eps)}  \leq \left( t\, C_0  + \alpha \Theta_A\right)\lvv \psi^{k} - \psi^{k-1}\rvv_{L^\infty_{\omega_1^A}(\OO^\eps)} ,
$$
These informations imply that, for a small time $T_0>0$, $\psi^k$ and $\gamma \psi^k$ are Cauchy sequences in the the Banach spaces $L^\infty_{\omega_1^A}(\UU^\eps_{T_0})$ and $L^\infty_{\omega_1^A}(\Gamma^\eps_{T_0})$. Therefore there are functions $\psi \in  L^\infty_{\omega_1^A}(\UU^\eps_{T_0})$ and $\gamma \psi \in L^\infty_{\omega_1^A}(\Gamma^\eps_{T_0})$ such that, as $k\to \infty$, there holds 
$$
\psi^k \to \psi \, \text{ strongly in }  L^\infty_{\omega_1^A}(\UU^\eps_{T_0}) \quad \text{ and } \quad \gamma \psi^k \to \gamma \psi \,  \text{ strongly in }  L^\infty_{\omega_1^A}(\Gamma^\eps_{T_0}). 
$$
Moreover, using \eqref{eq:equivalency_omega_omega_A} and passing to the limit in the weak formulation associated with Equation~\eqref{eq:TransportLinftyK} we deduce that $\psi$ solves the evolution equation
\begin{equation}
	\left\{\begin{array}{llll}
		 \partial_t \psi &=&  \TTT \psi + K \psi + G  &\text{ in }\UU^\eps\\
		\gamma_-\psi&=& \alpha\,  \RRR \gamma_+ \psi  &\text{ on }\Gamma^\eps_{-}\\
		\psi _{t=0}&=& \psi_0 &\text{ in }\OO^\eps,
	\end{array}\right.\label{eq:TransportLinftyK_limit}
\end{equation}
in the time interval $[0, T_0]$. Repeating this argument in every time interval $[nT_0, (n+1) T_0]$ for every $n\in \N$ we deduce the existence of a global solution $\psi$. Furthermore, from the representation formula associated with Equation~\eqref{eq:TransportLinftyK}, and the previous strong convergence result we have that
\begin{multline}\label{eq:TransportRepresentationPsi}
\psi(t,x,v) = e^{-\nu(v) t} \psi_0(x,v) \, \Ind_{t_1\leq 0} + \int_{\max(0, t_1)}^t e^{-\nu(v) (t-s)} \, K\psi(s,x,v)\, \ds  \\
+ \int_{\max(0, t_1)}^t e^{-\nu(v) (t-s)} \, G(s,x,v)\, \ds
+ e^{-\nu(v)(t-t_1)}  \gamma \psi (t_1, x_1, v) \, \Ind_{t_1>0}, 
\end{multline}
for every $t\geq 0$ and for almost every $(x,v)\in \OO^\eps$. 

We can then use the well-posedness from the Step 3 of the proof of Theorem~\ref{theo:ExistenceL2} and Remark \ref{rem:ExistenceL2} to justify the repetition of the hypocoercivity result from Theorem~\ref{theo:Hypo} applied to Equation~\eqref{eq:TransportLinftyK_limit}. We may thus repeat the computations leading to the conclusion of Proposition~\ref{prop:LinftyEstimatePerturbedFiniteT_Summary_NoZeroMeasure}. Therefore, there are $\eps_{12}, \theta>0$, given by Proposition~\ref{prop:LinftyEstimatePerturbedFiniteT_Summary}, such that for every $\eps\in (0,\eps_{12})$, there holds the following energy estimate
\be \label{eq:Linfty_ineq_0_Summary_Linfty}
\lvv \psi_t \rvv_{L^\infty_{\omega_1} (\OO^\eps)} + \lvv \gamma \psi_t \rvv_{L^\infty_{\omega_1} (\Sigma^\eps)} \leq C e^{-\theta \eps^2 t} \left( \lvv \psi_0\rvv_{L^\infty_{\omega_1}(\OO^\eps)} + \underset{s\in [0,t]}{\sup}  \lvv G\rvv_{L^\infty_{\omega_1\nu^{-1}}(\UU^\eps)}  \right),
\ee
for every $t\geq 0$, and where $C>0$ is an universal constant independent of $\eps$. Finally, it is worth remarking that \eqref{eq:Linfty_ineq_0_Summary_Linfty} is uniform in $\alpha$ from the fact that this estimate comes from Proposition~\ref{prop:LinftyEstimatePerturbedFiniteT_Summary_NoZeroMeasure}.

\medskip\noindent
\emph{Step 4.} We now take a sequence $\alpha_k \nearrow 1$, with $k\in \N$, and we consider the sequence $(f_k)_{k\in \N}$ obtained in Step~3 as the solution to the modified Maxwell reflection boundary condition problem
\begin{equation}\label{eq:linear_gak_Linfty}
\left\{\begin{array}{rrll}
		 \partial_t f_k &=&  \TTT f_k + K f_k + G  &\text{ in }\UU^\eps\\
		\gamma_- f_k&=& \alpha_k\,  \RRR \gamma_+ f_k  &\text{ on }\Gamma^\eps_{-}\\
		 f_{k, t=0} &=& f_0 &\text{ in }\OO^\eps.
	\end{array}\right.
\end{equation}
Moreover, it is worth remarking that there holds the representation formula \eqref{eq:TransportRepresentationPsi} for $f_k$.

Using then the energy estimate \eqref{eq:Linfty_ineq_0_Summary_Linfty} we obtain that $f^k$ and $\gamma f^k$ are bounded sequences in $L^\infty_{\omega_1}(\UU^\eps)$ and $L^\infty_{\omega_1}(\Gamma^\eps)$, therefore there is $f\in L^\infty_{\omega_1}(\UU^\eps)$ and $\gamma f \in L^\infty_{\omega_1}(\Gamma^\eps)$ such that, as $k\to \infty$, there holds
$$
f^k \wto f \hbox{ weakly-$*$ in } L^\infty_{\omega_1}(\UU^\eps) \quad \text{ and } \quad \gamma f^k \wto \gamma f \hbox{ weakly-$*$ in } L^\infty_{\omega_1}(\Gamma^\eps). 
$$
We can thus pass to the limit in the weak formulation associated to Equation~\eqref{eq:linear_gak_Linfty} and we deduce that $f$, with its associated trace $\gamma f$, solves Equation~\eqref{eq:PLRBE} in the sense of \eqref{eq:renormalizedFormulation}. 

Furthermore, we remark that by taking the representation formula \eqref{eq:TransportRepresentationPsi} for $f^k$, multiplying it by any arbitrary function $g \in L^1_{m_1} (\UU^\eps)$, where $m_1:=\omega^{-1}_1$, we may pass to the limit as $k\to \infty$ and by density we deduce that $f$ also satisfies the representation formula \eqref{eq:TransportRepresentationPsi}.

Finally, we conclude by remarking that the conclusion of Proposition~\ref{prop:LinftyEstimatePerturbedFiniteT_Summary} holds from the fact that we have access to a Duhamel-type formulation and using Remark \ref{rem:ExistenceL2}.
\qed

\subsection{Proof of Theorem~\ref{theo:ExistenceL_infty_linear_transportAA_delta}} \label{ssec:LinftyWellPosednessAA_delta}
The proof of Theorem~\ref{theo:ExistenceL_infty_linear_transportAA_delta} follows exactly the same arguments used for the proof of Theorem~\ref{theo:ExistenceL_infty_linear_transport}, using Proposition~\ref{prop:LinftyEstimatePerturbedFiniteT_Summary_WWWW0} instead of Proposition~\ref{prop:LinftyEstimatePerturbedFiniteT_Summary_NoZeroMeasure}.
\qed

\section{Proof of Theorem~\ref{theo:MainRescaled} for strongly confining weights}\label{sec:SolutionsWWWWinfty}
We dedicate this section to prove a first part of Theorem~\ref{theo:MainRescaled}.

 \begin{theo}\label{theo:existenceWWWWinfty}
Consider either Assumption \ref{item:RH1} or \ref{item:RH2} to hold, and let $\omega_1 \in \WWWW_1$ a strongly confining admissible weight function.
For every $\eps\in (0, \eps_{12})$, where we recall that $\eps_{12}>0$ is given by Proposition~\ref{prop:LinftyEstimatePerturbedFiniteT_Summary}, there is $\eta_0^1(\eps)\in (0,1)$ satisfying $\eta_0^1(\eps) \to 0$ as $\eps \to 0$, such that if $f_0\in L^\infty_{\omega_1}(\OO^\eps)$ satisfies
\beqn\label{eq:smallness}
\lVert f_0\rVert_{L^\infty_{\omega_1}(\OO^\eps)} \leq (\eta_0^1 (\eps))^2,
\eeqn
there is a function $f\in L^\infty_{\omega_1}(\UU^\eps)$ with an associated trace $\gamma f\in L^\infty_{\omega_1}(\Gamma^\eps)$,  unique solution to the linearized rescaled Boltzmann Equation~\eqref{eq:RBE}-\eqref{eq:RBEBC}-\eqref{eq:RBEIC}  in the distributional sense, i.e for every test function $\varphi \in \DD(\bar \UU^\eps)$ there holds 
\begin{multline}\label{eq:WeakFormulationGaussian}
\int_{\OO^\eps} f(t,\cdot) \, \varphi(t,\cdot) \dv\dx  -  \int_0^t \int_{\OO^\eps}  f \left( \partial_t \varphi  +  v\cdot \grad_x \varphi + K^*\varphi -\nu\varphi \right)  +\varphi \, \QQ(f,f)\,  \dv\dx \ds \\
+ \int_0^t \int_{\Sigma^\eps} \gamma f\,  \varphi\,  (n_x\cdot v) \dv \d\sigma_x  
= \int_{\OO^\eps} f_0\,  \varphi(0,\cdot )\dv\dx ,
\end{multline}
for every $t\geq 0$, and we recall that $K^*$ is defined in \eqref{defAdjointK}. Furthermore, there holds
\be
  \lVert  f_t \rVert_{L^{\infty}_{\omega_1}( \OO)}   \leq e^{-\theta \eps^2 t}  \, \eta_0^1(\eps)\label{eq:RBEdecayFinal} \qquad  \forall t\geq 0.
\ee
\end{theo}

\subsection{Estimate on the bilinear Boltzmann collision operator}
Before studying the non-linear Boltzmann equation we state the following control on the bilinear Boltzmann collision operator.

\begin{lem}\label{lem:NonlinearGaussianWeightEstimate}
Let $\omega$ be an admissible weight function, and let $g,h\in L^\infty_{\omega}(\OO^\eps)$. There holds
$$
\lVert \QQ(g, h) \rVert_{L^\infty_{\omega\nu^{-1}}(\OO^\eps)} \leq C_\QQ \lVert  g \rVert_{L^\infty_{\omega}(\OO^\eps)}  \lVert   h\rVert_{L^\infty_{\omega}(\OO^\eps)}.
$$
for some constant $C_\QQ = C_\QQ(\omega)>0$.
\end{lem}

\begin{rem}\label{rem:NonlinearGaussianWeightEstimate}
This type of estimate is classical in the study of the Boltzmann equation, see for instance \cite{MR3779780, MR2679358, MR1307620, MR3740632}. Thus we do not prove it, the interested reader may find a prove in either of the above references. 
\end{rem}

\subsection{Proof of Theorem~\ref{theo:existenceWWWWinfty}}\label{sssec:ProofExistenceWWWWinfty}
We define the norm
$$
[[f]]_1:= \underset{s\in [0, \infty)}\sup \left[ e^{\theta \eps^2 s }\lvv f_s \rvv_{L^\infty_{\omega_1}(\OO^\eps)}\right],
$$
where we recall that $\theta>0$ is given by Proposition~\ref{prop:LinftyEstimatePerturbedFiniteT_Summary}. We consider a $\lambda>0$ to be fixed later, and we define the Banach space
$$
\ZZ_1:= \{ g\in L^\infty_{\omega_1}(\UU^\eps), \, [[g]]_1 \leq \lambda\},
$$
equipped with the strong topology in $L^\infty_{\omega_1}(\UU^\eps)$.

We denote $\Psi_1$ as the map that to $g\in \ZZ_1$ assigns $f$ the solution of 
    \be
	\left\{\begin{array}{llll}
		\partial_{t} f &=& \LLL f + \QQ(g,g) &\text{ in }\UU^\eps \\
		\gamma_- f&=&\RRR \gamma_+f , &\text{ on }\Gamma_{-}^\eps\\
		 f_{t=0}&=& f_0, &\text{ in }\OO^\eps,
	\end{array}\right.\label{eq:PLRBEGauss}
\ee
which is given by Theorem~\ref{theo:ExistenceL_infty_linear_transport}. Indeed, we observe that Theorem~\ref{theo:ExistenceL_infty_linear_transport} holds by using Lemma~\ref{lem:NonlinearGaussianWeightEstimate} and the fact that $g\in \ZZ_1$.
We remark then that Proposition~\ref{prop:LinftyEstimatePerturbedFiniteT_Summary} implies that  for every $\eps\in (0,\eps_{12})$ there is $C_0=C_0(\eps)>0$ satisfying $C_0(\eps)\to \infty$ as $\eps\to 0$, such that
\be\label{eq:ControlExistenceNormStrongW}
\bal
[[f]]_1 &\leq C_0 \lvv f_0\rvv_{L^\infty_{\omega_1}(\OO^\eps)} + C_0 \underset{s\in[0,t]}{\sup}\left[ e^{\theta \eps^2 s} \lVert \QQ(g_s, g_s) \rVert_{L^{\infty}_{\omega_1\nu^{-1}}(\OO^\eps)}\right]\\
&\leq C_0 \lvv f_0\rvv_{L^\infty_{\omega_1}(\OO^\eps)} + C_0 C_\QQ [[g]]^2_1
\eal
\ee
where we have used Lemma~\ref{lem:NonlinearGaussianWeightEstimate} to obtain the second line and the constant $C_\QQ=C_\QQ(\omega_1) >0$ is given by Lemma~\ref{lem:NonlinearGaussianWeightEstimate}.

We then choose $\eta_0^1 (\eps) = \lambda(\eps)$ and 
$$
\lambda =\lambda(\eps) := \min\left( {1\over C_0(\eps) (1+C_\QQ)} , {1\over 4C_0(\eps)C_\QQ}, 1\right) \underset{\eps\to 0}{\longrightarrow} 0 .
$$
In particular, this choice of $\lambda$ together with \eqref{eq:ControlExistenceNormStrongW} implies that $f\in \ZZ_1$, thus $\Psi_1:\ZZ_1 \to \ZZ_1$. 

Furthermore, if we take $g_1, g_2\in \ZZ_1$ and denoting $f_i = \Psi_1(g_i)$ for $i=1,2$, we have that $\psi = f_1-f_2$ is the weak solution of the evolution equation
    \beqn
	\left\{\begin{array}{llll}
		\partial_{t} \psi &=& \LLL \psi + \QQ(g_1+g_2, g_1-g_2) &\text{ in }\UU^\eps \\
		\gamma_- \psi&=&\RRR \gamma_+ \psi  &\text{ on }\Gamma_{-}^\eps\\
		 \psi_{t=0}&=& 0 &\text{ in }\OO^\eps,
	\end{array}\right.\label{eq:PLRBEGauss_psi_k}
\eeqn
in the sense of Theorem~\ref{theo:ExistenceL_infty_linear_transport}.
Using then Proposition~\ref{prop:LinftyEstimatePerturbedFiniteT_Summary} and Lemma~\ref{lem:NonlinearGaussianWeightEstimate} as before we have that 
$$
[[\psi]]_1 \leq C_0 C_\QQ [[g_1+g_2]]_1 [[g_1-g_2]]_1 \leq 2 C_0C_\QQ \lambda  [[g_1-g_2]]_1 \leq \frac 12  [[g_1-g_2]]_1. 
$$
This implies that $\Psi_1$ is a contraction in $\ZZ_1$, thus there is a unique fixed point $f\in \ZZ_1$ for this map. We deduce that $f$ is a weak solution of the nonlinear Boltzmann Equation~\eqref{eq:RBE}-\eqref{eq:RBEBC}-\eqref{eq:RBEIC} in the sense of \eqref{eq:WeakFormulationGaussian}. Furthermore, \eqref{eq:RBEdecayFinal} comes from the very fact that $f\in \ZZ_1$.
\qed

\section{Proof of Theorem~\ref{theo:MainRescaled} for weakly confining weights}\label{sec:SolutionsWWWW0}
We dedicate this section to prove the remaining part of Theorem~\ref{theo:MainRescaled} on the well-posedness of the Boltzmann equation for weakly confining admissible weight functions. 

 \begin{theo}\label{theo:existenceWWWW0}
Consider either Assumption \ref{item:RH1} or \ref{item:RH2} to hold, and let $\omega_0 \in \WWWW_0$ a weakly confining admissible weight function. Define furthermore $\eps_{13}: = \min(\eps_9, \eps_{12})$, where we recall that $\eps_9>0$ is given by Proposition~\ref{prop:LinftyEstimatePerturbedFiniteT_Summary_WWWW0} and $\eps_{12}>0$ is given by Proposition~\ref{prop:LinftyEstimatePerturbedFiniteT_Summary}.

For every $\eps \in (0, \eps_{13})$ there is $ \eta_0^0 (\eps) \in (0,1)$ such that $\eta_0^0(\eps)\to 0$ as $\eps\to 0$, such that for every $f_0\in L^\infty_{\omega_0}(\OO^\eps)$ satisfying
\beqn\label{eq:smallnessf1}
\lVert f_0\rVert_{L^\infty_{\omega_0}(\OO^\eps)} \leq (\eta_0^0(\eps))^2,
\eeqn
there is $f\in L^\infty_{\omega_0}(\UU^\eps)$ with an associated trace $\gamma f\in L^\infty_{\omega_0}(\Gamma^\eps)$, unique solution to the linearized rescaled Boltzmann Equation~\eqref{eq:RBE}-\eqref{eq:RBEBC}-\eqref{eq:RBEIC} in the distribution sense, i.e for every test function $\varphi \in \DD(\bar \UU^\eps)$ there holds 
\begin{multline}\label{eq:WeakFormulationGaussian_0}
\int_{\OO^\eps} f(t,\cdot) \, \varphi(t,\cdot) \dv\dx  -  \int_0^t \int_{\OO^\eps}  f \left( \partial_t \varphi  +  v\cdot \grad_x \varphi + K^*\varphi -\nu\varphi \right)  + \varphi \, \QQ(f,f)\,  \dv\dx \ds \\
+ \int_0^t \int_{\Sigma^\eps} \gamma f\,  \varphi\,  (n_x\cdot v) \dv \d\sigma_x  
= \int_{\OO^\eps} f_0\,  \varphi(0,\cdot )\dv\dx ,
\end{multline}
for every $t\geq 0$, and we recall that $K^*$ is defined in \eqref{defAdjointK}. Furthermore, there is $\theta>0$ such that there holds
\be\label{eq:SolutionsDecayWWWW0}
 \lVert  f_t \rVert_{L^{\infty}_{\omega_0}\left({\OO^\eps}\right)}  \leq  e^{-\theta \eps^2 t} \, {\eta_0^0(\eps)} \qquad \forall t\geq 0,
\ee
 \end{theo}

\begin{rem}
The proof follows the main ideas of \cite[Section 6]{BriantGuo16}.
\end{rem}

\subsection{Dissipative equation} Let $h:\UU^\eps\to \R$, we study the following equation
  \be
	\left\{\begin{array}{llll}
		\partial_{t} f_1 &=& \TTT f_1 + \AA_\delta f_1 + \QQ(f_1 + h , f_1 +h) &\text{ in }\UU^\eps \\
		\gamma_- f_1&=&\RRR \gamma_+f_1 &\text{ on }\Gamma_{-}^\eps\\
		 f_{1,t=0}&=& f _0 &\text{ in }\OO^\eps.
	\end{array}\right.\label{eq:PLRBEDiss_h}
\ee
We dedicate this subsection to prove the following well-posedness result. 
 \begin{prop}\label{prop:LinftySplitDissip}
Consider either Assumption \ref{item:RH1} or \ref{item:RH2} to hold, let $\omega_0 \in \WWWW_0$ a weakly confining admissible weight function, and let $f_0\in L^\infty_{\omega_0}(\OO^\eps)$ and such that 
\beqn\label{eq:smallnessf1}
\lVert f_0\rVert_{L^\infty_{\omega_0}(\OO^\eps)} \leq \eta_1,
\eeqn
some $\eta_1 >0$ small enough. Assume furthermore that for every $\eps\in (0,\eps_{8})$ and $\delta\in (0,\delta_0)$ there holds
\be\label{eq:smallnessf1h}
\underset{s\in[0,\infty) }{\sup}\left[ e^{\theta\eps^2 s} \lVert h_s\rVert_{L^\infty_{\omega_0}( \OO^\eps)} \right]\leq \eta_1,
\ee
where $\eps_9, \delta_0 >0$ are given by Proposition~\ref{prop:LinftyEstimatePerturbedFiniteT_Summary_WWWW0}, and $\theta\in (0,\nu_0/2)$ is given by Proposition~\ref{prop:LinftyEstimatePerturbedFiniteT_Summary}.  
For every $\eps\in (0,\eps_{8})$ and $\delta\in (0,\delta_0)$ there is $f_1\in L^\infty_{\omega_0}(\UU^\eps)$ with an associated trace $\gamma f_1\in L^\infty_{\omega_0}(\Gamma^\eps)$, unique solution to Equation~\eqref{eq:PLRBEDiss_h} in the weak sense,  i.e for any $\varphi \in \DD(\bar \UU^\eps)$ there holds 
\begin{multline}\label{eq:renormalizedFormulationAA_delta}
\int_{\OO^\eps} f_1(t,\cdot) \, \varphi(t,\cdot) \, \dx\dv-\int_0^t \int_{\OO^\eps} f_1 \, \AA_\delta ^* \, \varphi  + f \left( \partial_t \varphi  + v\cdot \grad_x \varphi -\nu \, \varphi \right) \,  \dx\dv\ds \\
-\int_0^t \int_{\OO^\eps} \QQ(f_1+h, f_1+h)\, \varphi \,  \dx\dv\ds 
+ \int_0^t \int_{\Sigma^\eps} \gamma f_1\,  \varphi\,  (n_x\cdot v) \d\sigma_x \dv 
= \int_{\OO^\eps} f_0(\cdot ) \varphi(0,\cdot ) \dx\dv,
\end{multline}
for every $t\geq 0$, and where we recall that $\AA_\delta^*$ has been defined in \eqref{def:AdjointAA_delta}. Furthermore, there holds 
\be\label{eq:SolutionsDecayAA_delta}
 \lVert  f_{1,t}\rVert_{L^{\infty}_{\omega_0}\left( {\OO^\eps}\right)}  \leq  C e^{-\theta \eps^2 t} \left( \lVert f_0\rVert_{L^\infty_{\omega_0}(\OO^\eps)}  +  \underset{s\in[0,\infty) }{\sup}\left[ e^{\theta\eps^2 s} \lVert h_s\rVert_{L^\infty_{\omega_0}( \OO^\eps)} \right]^2 \right) \qquad \forall t\geq0,
\ee
for some constant $C>0$, independent of $\eps$.
 \end{prop}

 \begin{proof} The proof is mainly a repetition of the arguments from the proof of Theorem~\ref{theo:existenceWWWWinfty}, thus we skip it.
 \end{proof}

 \subsection{Control on the operator $\KK_\delta$} We recall the definition of $\KK_\delta$ given by \eqref{eq:defSplittingLL} and from \cite[Section 4]{MR3779780} we have the following control.

 \begin{lem}\label{lem:KKcontrol}
Let $\omega$ be an admissible weight function. For every $\delta>0$ there holds
 \be
 \lVert \KK_\delta f\rVert_{L^\infty_{\varsigma}(\OO^\eps)} \leq C_{\delta} \lVert f\rVert_{L^\infty(\OO^\eps)},\label{eq:KKcontrol1W}
 \ee
 for some constructive constant $C_\delta>0$.
 \end{lem}

\subsection{Residual linear equation}\label{ssec:Residual_Equation_f2}

We consider $g:\UU^\eps \to \R$, and we study the following linear evolution equation
 \be
	\left\{\begin{array}{rlll}
		\partial_{t} f_2 &=& \LLL f_2 +  \KK_\delta g&\text{ in }\UU^\eps\\
		\gamma_- f_2&=&\RRR \gamma_+f_2 &\text{ on }\Gamma_{-}^\eps\\
		 f_{2, t=0}&=& 0 &\text{ in }\OO^\eps,
	\end{array}\right.\label{eq:PLRBE2}
\ee
where we recall that $\KK_\delta$ is defined in \eqref{eq:defSplittingLL}.
We define the projection operator 
$$
\PPP f (t,x,v):= \MM(v) \int_{\OO^\eps} f(t,x,u) \d u,
$$
and we dedicate this subsection to prove the following well-posedness and long-time behavior result
 
  \begin{prop}\label{prop:LinftySplitReg}
Consider either Assumption \ref{item:RH1} or \ref{item:RH2} to be satisfied, let $g\in L^\infty_{\omega_0}(\UU^\eps)$, for some weakly confining admissible weight function $\omega_0 \in \WWWW_0$.

For every strongly confining admissible weight function $\omega_1\in \WWWW_1$, every $\delta>0$, and every $\eps\in (0, \eps_{13})$, where we recall that $\eps_{13}>0$ is given in Theorem~\ref{theo:existenceWWWW0}, there is $f_2 \in L^\infty_{\omega_1}(\UU^\eps)$ unique weak solution of Equation~\eqref{eq:PLRBE2} in the sense of the Theorem~\ref{theo:ExistenceL_infty_linear_transport}.

Moreover, if $\PPP (f_2 + g) = 0$, and for every $\eps\in (0,\eps_{13})$ there holds
\be
\lVert  g_t\rVert_{L^{\infty}_{\omega_0}\left( {\OO^\eps}\right)} \leq C_g   \, e^{-\theta \eps^2 t}     \qquad \forall t\geq0, \label{eq:smallnessf2g}
\ee
for some constant $C_g>0$, and where $\theta >0$ is given by Proposition~\ref{prop:LinftyEstimatePerturbedFiniteT_Summary}, then there is a constant $C= C(\eps)>0$ satisfying $C(\eps) \to \infty$ as $\eps\to 0$, such that
\be
     \lVert  f_{2,t}\rVert_{L^{\infty}_{\omega_1}\left( {\OO^\eps}\right)} \leq C C_g  e^{-\theta\eps^2 t}   \label{eq:Decayf2},
    \ee
for every $t\geq0$.
 \end{prop}

\begin{proof} 
The proof follows the main ideas from \cite[Proposition 6.8]{BriantGuo16} by using Lemma~\ref{lem:KKcontrol}, the existence results from Theorem~\ref{theo:ExistenceL_infty_linear_transport} and the decay estimate from Proposition~\ref{prop:LinftyEstimatePerturbedFiniteT_Summary}, thus we skip it.
\end{proof}

 \subsection{Proof of Theorem \ref{theo:existenceWWWW0}}
We consider $h\in L^\infty_{\omega_0}(\OO^\eps)$ such that that for every $\eps\in (0,\eps_{13})$ and $\delta\in (0,\delta_0)$ there holds
\be\label{eq:smallnessf1h_0}
\underset{s\in[0,\infty) }{\sup}\left[ e^{\theta\eps^2 s} \lVert h_s\rVert_{L^\infty_{\omega_0}( \OO^\eps)} \right]\leq \eta_1,
\ee
where $\delta_0 >0$ is given by Proposition~\ref{prop:LinftyEstimatePerturbedFiniteT_Summary_WWWW0}, $\theta\in (0,\nu_0/2)$ is given by Proposition~\ref{prop:LinftyEstimatePerturbedFiniteT_Summary}, and $\eta_1>0$ is given by Proposition~\ref{prop:LinftySplitDissip}.  
We define $f_1\in L^\infty_{\omega_0}(\UU^\eps)$ as the solution of the evolution equation 
    \be
	\left\{\begin{array}{rlll}
		\partial_{t} f_1 &=& \TTT f_1 + \AA_\delta f_1 + \QQ(f_1+h , f_1+h) &\text{ in }\UU^\eps \\
		\gamma_- f_1&=&\RRR \gamma_+f_1 , &\text{ on }\Gamma_{-}^\eps\\
		 f_{1, t=0}&=& f_0, &\text{ in }\OO^\eps,
	\end{array}\right.\label{eq:PLRBEGauss_existence1}
\ee
which is given by Proposition~\ref{prop:LinftySplitDissip}. Moreover, we fix $\omega_1\in \WWWW_1$ a strongly confining admissible weight function, and we define $f_2 \in L^\infty_{\omega_1}(\UU^\eps)$ as the solution of the evolution equation 
 \be
	\left\{\begin{array}{rlll}
		\partial_{t} f_2 &=& \LLL f_2 +  \KK_\delta f_1&\text{ in }\UU^\eps\\
		\gamma_- f_2&=&\RRR \gamma_+f_2 &\text{ on }\Gamma_{-}^\eps\\
		 f_{2, t=0}&=& 0 &\text{ in }\OO^\eps,
	\end{array}\right.\label{eq:PLRBEGauss_existence2}
\ee
which is given by Proposition~\ref{prop:LinftySplitReg}. We emphasize that, defined this way, $f_2$ depends of $h$. 

We define the norm
\beqn\label{def:norm0}
[[g]]_0:= \underset{s\in [0, \infty)}\sup \left[ e^{ \theta\eps^2 s }\lvv g_s \rvv_{L^\infty_{\omega_0}(\OO^\eps)}\right],
\eeqn
and, for $\lambda\in(0, \eta_1)$ to be fixed later, we define the Banach space $\ZZ_0:= \{ g \in L^\infty_{\omega_0}(\UU^\eps), \, [[g]]_0 \leq \lambda  \}$, equipped with the strong topology in $L^\infty_{\omega_0}(\UU^\eps)$, which makes $\ZZ_0$ a bounded, convex, closed subset of $L^\infty_{\omega_0}(\UU^\eps)$. 
Furthermore, we denote $\Psi$ as the map that to $h \in \ZZ_0$ assigns $f_2$ as defined above.

We remark that $f_1+f_2$ solves Equation \eqref{eq:PLRBE} with $G=  \QQ(f_1+h , f_1+h)$, which together with \eqref{eq:ConservationLaws} implies that $\PPP(f_1 + f_2) = 0$. We set then the constant $C_{\omega_1}>0$ such that 
$$
\lvv g \rvv_{L^\infty_{\omega_0}(\UU^\eps)} \leq C_{\omega_1} \lvv g \rvv_{L^\infty_{\omega_1}(\UU^\eps)} ,
$$  
and using \eqref{eq:smallnessf1h_0}, \eqref{eq:Decayf2} and \eqref{eq:SolutionsDecayAA_delta} we obtain that
$$
[[f_1]]_0 \leq C_1 \left( \lVert f_0\rVert_{L^\infty_{\omega_0}(\OO^\eps)}  +  [[ h]]_0^2 \right) \leq C_1 \left( (\eta_0^0)^2 + \lambda^2 \right), \quad \text{ and } \quad [[ f_2 ]]_0 \leq  C_{\omega_1} C_2 C_1 \left( (\eta_0^0)^2 + \lambda^2 \right) , 
$$
where $C_1>0$ is given by \eqref{eq:SolutionsDecayAA_delta} in Proposition~\ref{prop:LinftySplitDissip}, and $C_2 = C_2(\eps)>0$, satisfying $C_2(\eps) \to \infty$ as $\eps\to 0$, is given by \eqref{eq:Decayf2} in Proposition~\ref{prop:LinftySplitReg}. 

We define $\eta_0^0=\lambda$ and
$$
\lambda = \lambda(\eps):= \min\left({1\over 8C_{\omega_1}C_1C_2 (\eps)}, {1\over 6C_1},  {1\over 8C_\QQ C_3} , {1\over 16   C_3C_\QQ } ,\frac{\eta_1}2,  1\right) \underset{\eps\to 0}\longrightarrow 0
$$
where $C_\QQ>0$ is given by Lemma~\ref{lem:NonlinearGaussianWeightEstimate} and $C_3>0$ is the constant given by Proposition~\ref{prop:LinftyEstimatePerturbedFiniteT_Summary_WWWW0}. This choices of parameters implies in particular that $[[f_1]] \leq \lambda$ and $[[f_2 ]]_0\leq \lambda$, thus $f_1, f_2\in \ZZ_0$, therefore $\Psi: \ZZ_0\to \ZZ_0$.

Furthermore, we consider $h_1, h_2\in \ZZ_0$ and we define $f_1, g_1\in \ZZ_0$ as the solutions of Equation~\eqref{eq:PLRBEGauss_existence1} associated with $h_1$ and $h_2$ respectively. We also define $f_2=\Psi(h_1)$ and $g_2=\Psi(h_2)$, i.e the solutions of Equation~\eqref{eq:PLRBEGauss_existence2} associated with $f_1$ and $g_1$ respectively. 
We denote $\phi=f_1-g_1$, $\psi:= f_2 - g_2$, and we observe that $\phi$ and $\psi$ are the weak solutions of
    \beqn
	\left\{\begin{array}{rlll}
		\partial_{t} \phi &=& \TTT \phi + \AA_\delta \phi + \QQ( f_1+g_1 + h_1+h_2 ,\phi + h_1-h_2) &\text{ in }\UU^\eps \\
		\gamma_- \phi&=&\RRR \gamma_+ \phi , &\text{ on }\Gamma_{-}^\eps\\
		 \phi_{t=0}&=& 0, &\text{ in }\OO^\eps,
	\end{array}\right.\label{eq:PLRBEGauss_existence1_uniq}
\eeqn
and
 \beqn
	\left\{\begin{array}{rlll}
		\partial_{t} \psi &=& \LLL \psi +  \KK_\delta \phi&\text{ in }\UU^\eps\\
		\gamma_- \psi&=&\RRR \gamma_+ \psi &\text{ on }\Gamma_{-}^\eps\\
		 \psi_{t=0}&=& 0 &\text{ in }\OO^\eps,
	\end{array}\right.\label{eq:PLRBEGauss_existence2_uniq}
\eeqn
respectively. 
On the one hand using Proposition~\ref{prop:LinftyEstimatePerturbedFiniteT_Summary_WWWW0} and Lemma~\ref{lem:NonlinearGaussianWeightEstimate} we have that
\be\label{eq:Control_phi1}
[[\phi]]_0 \leq  C_3 C_\QQ [[f_1+g_1 + h_1+h_2]]_0[[\phi + h_1-h_2]]_0
\leq  \frac 12 [[\phi]]_0 + 4\lambda C_3 C_\QQ [[h_1-h_2]]_0
\ee
where we have used the triangular inequality and the definition of $\lambda$ to obtain the last inequality. In particular, \eqref{eq:Control_phi1} and our very definition of $\lambda$ implies that 
\beqn\label{eq:Control_phi2}
[[\phi]]_0 \leq 8\lambda C_3C_\QQ [[h_1-h_2]]_0 \leq \frac 12 [[h_1-h_2]]_0. 
\eeqn

This implies that $\Psi$ is a contraction in $\ZZ_0$, thus using the Banach fixed point theorem we deduce that there is a unique fixed point $f_2\in \ZZ_0$ for this map. 
Moreover, from the very definition of the map $\Psi$ we deduce that $f_2$ is a weak solution of Equation~\eqref{eq:PLRBEGauss_existence2}, and Proposition~\ref{prop:LinftySplitDissip} further gives the uniqueness of $f_1$, weak solution of Equation~\eqref{eq:PLRBEGauss_existence1} with $h=f_2$. 

 Putting together the aforementioned weak formulations we deduce that $f=f_1+f_2$ is a weak solution of the Boltzmann equation \eqref{eq:RBE}-\eqref{eq:RBEBC}-\eqref{eq:RBEIC} in the sense of \eqref{eq:WeakFormulationGaussian_0}. Finally, \eqref{eq:SolutionsDecayWWWW0} comes from the very fact that $f=f_1+f_2$ with $f_1, f_2 \in \ZZ_0$. This concludes the proof. 
\qed

\bigskip

\textbf{Acknowledgements.} The author deeply thanks Kleber Carrapatoso and Stéphane Mischler for pressenting the problem, for pointing important bibliographical references, for the discussions and their comments during the process of the paper. 

This project has received funding from the European Union’s Horizon 2020 research and innovation programme under the Marie Skłodowska-Curie grant agreement No 945332. 

\bigskip

The author has no conflicts of interest to declare that are relevant to the content of this article.
No datasets were generated or analyzed during the current study.

For the purpose of open access, the author has applied a Creative Commons Attribution (CC-BY) license to any Author Accepted Manuscript version arising from this submission.

\bigskip

 \medskip

\bigskip
\bigskip

\end{document}